\documentclass[10pt]{amsart}
\usepackage[utf8]{inputenc}

\usepackage{amsmath, amsthm, amsfonts,stmaryrd,amssymb}
\usepackage[top=1in, bottom=1.25in, left=1.25in, right=1.25in]{geometry}
\usepackage{mathrsfs}
\usepackage{dsfont}                  
\usepackage{mathtools}                   
\usepackage{booktabs}
\usepackage{tikz}
\usepackage{dirtytalk}
\usepackage{todonotes}
\usepackage{enumerate} 
\usepackage{hyperref}
\usepackage{theoremref}
\usepackage{verbatim}
\usepackage{bm}

\usepackage{nomencl}
\makenomenclature

\usepackage{pgfplots}
\pgfplotsset{compat=1.15}
\usepackage{mathrsfs}
\usetikzlibrary{arrows}

\newtheorem{theorem}{Theorem}[section]
\newtheorem{remark}[theorem]{Remark}
\newtheorem{lemma}[theorem]{Lemma}
\newtheorem{corollary}[theorem]{Corollary}

\newtheorem{proposition}[theorem]{Proposition}

\theoremstyle{definition}
\newtheorem{definition}[theorem]{Definition}
\def\R{{\mathbb R}}

\def\E{{\mathbb E}}

\title{Global regularity estimates for optimal transport via entropic regularisation}

\author{Nathael Gozlan}
\address{N. G. : Université Paris Cité, CNRS, MAP5, F-75006 Paris}
\email{nathael.gozlan@u-paris.fr}

\author{Maxime Sylvestre}
\address{M. S. : CEREMADE, Université Paris-Dauphine, Université PSL, CNRS, MOKAPLAN, Inria Paris, 75016 Paris, France}
\email{maxime.sylvestre@dauphine.psl.eu}


\thanks{N.G. is supported by a grant of the Agence nationale de la recherche (ANR), Grant ANR-23-CE40-0017 (Project SOCOT)}

\date{\today}
\begin{document}

\maketitle

\begin{abstract}
    We develop a general approach to prove global regularity estimates for quadratic optimal transport using the entropic regularisation of the problem and the Prékopa–Leindler inequality.
\end{abstract}

\section{Introduction}
The aim of the paper is to obtain general global regularity estimates for the Brenier optimal transport map between two probability measures on $\R^n$, in the spirit of the celebrated Caffarelli's contraction theorem, recalled below:
\begin{theorem}[Caffarelli \cite{Caffarelli2000,Caf02}]\label{thm:Caf-intro}
Let $\mu(dx) = e^{-V(x)}dx$, $\nu(dy) = e^{-W(y)} dy$ be two probability measures on $\R^n$ such that $\mathrm{dom} V = \R^n$ and $\mathrm{dom}W$ is convex with nonempty interior. Further assume that $V,W$ are twice continuously differentiable on the interior of their domains and satisfy
\begin{equation*}
    \nabla^2 V \leq \alpha_V I_n, \quad \nabla^2 W \geq \beta_W I_n,
\end{equation*}
with $\alpha_V,\beta_W >0$.
Then the optimal transport map for the quadratic transport problem from $\mu$ to $\nu$ is $\sqrt{\alpha_V/\beta_W}$-Lipschitz.
\end{theorem}
Here, and throughout the paper, optimality refers to the quadratic transport problem on $\R^n$. Given two probability measures $\mu,\nu$ on $\R^n$ with finite second moments, recall that the quadratic transport cost between $\mu$ and $\nu$ is the quantity
\begin{equation}\label{eq:W_2^2}
W_2^2(\mu,\nu) = \inf_{\pi \in \Pi(\mu,\nu)} \int \vert y-x\vert^2\,\pi(dxdy),
\end{equation}
where $\vert\,\cdot\,\vert$ denotes the standard Euclidean norm on $\R^n$ and $\Pi(\mu,\nu)$ the set of all couplings between $\mu$ and $\nu$ that is the set of all probability measures $\pi$ on $\R^n \times \R^n$ such that $\pi$ has marginals $\mu$ and $\nu$. A map $T:\R^n \to \R^n$ is called optimal for this transport problem, whenever $\nu = T_\# \mu$ and 
\[
\int \vert T(x)-x\vert^2\,\mu(dx) = W_2^2(\mu,\nu),
\]
which means that the coupling $\pi(dxdy) = \mu(dx)\delta_{T(x)}(dy)$ induced by $T$ achieves the minimal value in \eqref{eq:W_2^2}. According to Brenier's Theorem \cite{Bre91}, if $\mu$ is absolutely continuous with respect to Lebesgue, then there is a $\mu$ almost surely unique transport map which is of the form $T=\nabla \phi$ where $\phi$ is a convex function on $\R^n.$

While in general studying the regularity of $\nabla \phi$ is a delicate question \cite{Caf92}, the interest of Caffarelli's result recalled above is to furnish simple conditions on the potentials $V$ and $W$ assuring that the optimal map is globally Lipschitz with an explicit and dimension free Lipschitz constant. This result has also found several applications in the field of geometric and functional inequalities \cite{Har99, Har01, Cor02, Mil18}. Namely, Caffarelli's contraction theorem applies in particular when $\mu = \gamma_n$ is the standard Gaussian probability measure and $\nu$ has a log concave density with respect to $\gamma_n$, and implies in this case that the Brenier transport map sending $\gamma_n$ onto $\nu$ is $1$-Lipschitz. This can be used for instance to transfer functional inequalities (Poincar\'e, Log-Sobolev, Gaussian Isoperimetric inequalities) known for $\gamma_n$ to all  of its log-concave perturbations with the same dimension free constant as $\gamma_n.$ 
In a slightly different spirit, Caffarelli's contraction result can also be used to control the deficit in some functional inequalities \cite{DPF17, CFP18}.
Caffarelli's contraction theorem has been extended to perturbations of log concave densities in \cite{CFJ17} and, more recently, to $1/d$-concave densities \cite{CFS24}. Construction of Lipschitz (but not necessarily optimal) transport maps between regular probability measures is also an active field of research, mainly motivated by geometric and functional inequalities. See for instance \cite{KM12, FMS24, MS23, MS24}.

The original proof of Theorem \ref{thm:Caf-intro} relies upon the Monge-Ampère equation satisfied by the function $\phi$ such that $T=\nabla \phi$ and in particular the regularity theory for such equations. Recently, Fathi, Prod'homme and the first author have proposed in \cite{FGP20} a different proof of Theorem \ref{thm:Caf-intro}, based on the entropic regularisation of the quadratic transport problem, that circumvents the use of Monge-Ampère equations. A simpler implementation of this idea has then been proposed by Chewi and Pooladian \cite{Chewi2022entropic}. Both proofs have in common to involve geometric tools related to the functional version of the Brunn-Minkowski inequality due to Prékopa-Leindler \cite{Pre71, Lei72, Pre73} (the stability of the set of log-concave functions under the Ornstein-Uhlenbeck semigroup in the case of \cite{FGP20} or the Poincaré type Brascamp-Lieb inequality in the case of \cite{Chewi2022entropic}). We refer to \cite{BL00} for a panorama of applications of the Prékopa-Leindler inequality to functional inequalities. In this paper, we propose a new adaptation of the entropic regularisation approach of \cite{FGP20,Chewi2022entropic} that enables us to obtain other explicit global regularity results for the Brenier map. We recover in particular the H\"older regularity result by Kolesnikov \cite{Kol10}. In this new twist of the entropic regularisation approach, the Prékopa-Leindler inequality will play a central role. 

Before presenting into more details this entropic regularization approach, let us give a flavor of the kind of global regularity estimates this method can reach. One of the main result of the paper is the following (which is a direct consequence of Corollary \ref{cor:OT-reg}): suppose that $\mu(dx) = e^{-V(x)}\,dx$ and $\nu(dy) = e^{-W(y)}\,dy$ are two probability measures on $\R^n$ with finite second moments associated to potentials $V:\R^n \to \R$ and $W:\R^n \to \R\cup \{+\infty\}$ such that
\begin{equation}\label{eq:V-intro}
V((1-t)x_0+tx_1) + t(1-t) \sigma(\vert x_1-x_0\vert) \geq (1-t)V(x_0)+tV(x_1),\qquad \forall x_0,x_1 \in \R^n, \forall t\in [0,1]
\end{equation}
and 
\begin{equation}\label{eq:W-intro}
W((1-t)y_0+ty_1) + t(1-t) \rho(\vert y_1-y_0\vert) \leq (1-t)W(y_0)+tW(y_1)\qquad \forall y_0,y_1\in \R^n, \forall t\in [0,1],
\end{equation}
where $\sigma: \R_+ \to \R_+$ and $\rho:\R_+ \to \R\cup\{+\infty\}$ are two measurable functions, then the Brenier transport map $T$ that takes $\mu$ on $\nu$ satisfies
\begin{equation}\label{eq:reg-intro}
    \vert T(x) - T(y) \vert \leq \frac{2}{\vert x-y \vert} \int_0^{\vert x-y \vert} (\rho^{\ast\ast})^{-1}(\sigma(s))\,ds, \qquad \forall x\neq y,
\end{equation}
where $h^*(v)=\sup_{u\geq0} \{uv-h(u)\}$, $v\geq 0$, is the monotone conjugate of a function $h:\R_+\to \R$ and $(\rho^{\ast\ast})^{-1}$ is the generalized inverse of the non-decreasing function $\rho^{\ast\ast}$ defined by
\[
(\rho^{\ast\ast})^{-1}(s) = \sup\{u \geq 0 : \rho^{\ast\ast}(u)\leq s\},\qquad s\geq0.
\]
Under the assumptions of Theorem \ref{thm:Caf-intro}, \eqref{eq:V-intro} holds with $\sigma(u)=\alpha_V \frac{u^2}{2}$ and \eqref{eq:W-intro} with $\rho(u) = \beta_W \frac{u^2}{2}$, and so the bound \eqref{eq:reg-intro} exactly yields 
\[
\vert T(x)-T(y)\vert \leq \sqrt{\frac{\alpha_V}{\beta_W}} \vert x-y\vert,\qquad \forall x,y\in \R^n.
\]
A function $V$ satisfying \eqref{eq:V-intro} is called $\sigma$-smooth and a function $W$ satisfying \eqref{eq:W-intro} is called $\rho$-convex. These classes of functions have been introduced in \cite{VNC78,Zal83} (see also \cite{AP} and \cite{Zal02}) and can be characterized in terms of growth properties of the gradient (see Section \ref{sec:ex-R-conv-S-smooth} for some explicit examples).
For instance, if $V$ satisfies \eqref{eq:V-intro} with $\sigma(u)= \alpha_V u^p$, and $\rho(u) = \beta_W u^q$, with $p\leq 2 \leq q$, then, one obtains from \eqref{eq:reg-intro} that the Brenier map $T$ from $\mu$ to $\nu$ is H\"older of exponent $p/q$:
\[
    \vert T(x)-T(y)\vert \leq \left(\frac{q}{p}\right)^{\frac{p}{p+q}} \left(\frac{\alpha_V}{\beta_W}\right)^{\frac{1}{q}} \vert x-y \vert^{\frac{p}{q}},\qquad \forall x,y\in \R^n.
\]
This latter bound recovers a result by Kolesnikov \cite{Kol10}.

According to a classical result (recalled in Proposition \ref{eq:continuity} below), if  $\phi$ is a convex function assumed to be $\sigma_\phi$-smooth for some function $\sigma_\phi : \R_+ \to \R_+$, then it holds
\[
\vert\nabla \phi(x) - \nabla \phi(y)\vert \leq 2\frac{\sigma_\phi(\vert x-y\vert)}{\vert x-y\vert},\qquad \forall x\neq y.
\]
(and this last inequality implies back the $\sigma_\phi$-smoothness up to constant, if $\sigma_\phi$ is convex).
Therefore, in order to prove that the Brenier map $T = \nabla \phi$ pushing $\mu$ onto $\nu$ satisfies the bound \eqref{eq:reg-intro} it is enough to prove that the convex function $\phi$ is such that
\begin{equation}\label{eq:phi-intro}
\phi((1-t)x_0+tx_1) + \bar{\sigma} (\vert x_1-x_0\vert) \geq (1-t)\phi(x_0) + t\phi(x_1),\qquad \forall x_0,x_1 \in \R^n, \forall t\in [0,1],
\end{equation}
for the modulus of smoothness $\bar{\sigma}$ given by 
\[
\bar{\sigma}(u) = \int_0^u (\rho^{\ast\ast})^{-1}(\sigma(s))\,ds,\qquad u\geq0.
\]
A key feature of the bound \eqref{eq:phi-intro} is that it is stable under pointwise convergence. This will play a crucial role in the entropic regularisation method that we shall now present. First of all, let us informally recall the entropy regularized transport problem.
Given two probability measures $\mu,\nu$ on $\R^n$ having finite second moments, the entropic regularized transport cost between $\mu$ and $\nu$ is given, for all $\epsilon>0$, by
\[
\mathcal{C}_\epsilon(\mu,\nu) = \inf_{\pi \in \Pi(\mu,\nu)}\int \frac{1}{2}\vert x-y\vert^2\,d\pi + \epsilon H(\pi | \mu\otimes \nu),
\]
where the relative entropy of $\pi$ with respect to $\mu \otimes \nu$, is given by $H(\pi| \mu \otimes \nu) = \int \log \frac{d\pi}{d(\mu \otimes \nu)}\,d\pi$, if $\pi \ll\mu \otimes \nu$ (and is $+\infty$ otherwise).
It is now well known \cite{Carlier2017}, that $\mathcal{C}_\epsilon(\mu,\nu) \to \frac{1}{2}W_2^2(\mu,\nu)$ 
as $\epsilon \to 0$. Contrary to the classical transport problem \eqref{eq:W_2^2}, the minimizer $\pi^\epsilon$ of the entropic regularized transport problem (which exists and is unique) can be easily computed from $\mu,\nu$ using a fixed point algorithm. More precisely, the optimal coupling $\pi_\epsilon$ is of the form
\[
\pi_\epsilon(dxdy) = e^{\frac{\langle x,y \rangle - \phi_\epsilon(x)-\psi_\epsilon(y)}{\epsilon}}\,\mu(dx)\nu(dy),
\]
with $(\phi_\epsilon,\psi_\epsilon)$ a couple of convex functions solution of the following system
\begin{align*}
\phi_\epsilon(x) &= \epsilon \log \left(\int e^{\frac{\langle x,y \rangle-\psi_\epsilon(y)}{\epsilon}}\,\nu(dy)\right),\qquad \forall x\in \R^n\\
\psi_\epsilon(y) &= \epsilon \log \left(\int e^{\frac{\langle x,y \rangle-\phi_\epsilon(x)}{\epsilon}}\,\mu(dx)\right),\qquad \forall y\in \R^n.
\end{align*}
Moreover, the so-called entropic Kantorovich potential $\phi_\epsilon$ converges $\mu$-almost everywhere (along some sequence $\epsilon_k$) to the Kantorovich potential $\phi$ such that $T=\nabla \phi$ \cite{Nutz2021}. 
To get the desired inequality \eqref{eq:phi-intro} for $\phi$, we prove it for $\phi_\epsilon$ and then we let $\epsilon \to 0$. In view of the above system of equations, the main problem is to understand the effect of the Laplace transform on $\sigma$-smooth and $\rho$-convex functions. Let us introduce the operator $\mathcal{L}_{\epsilon}$ acting on measurable functions $f:\R^n \to \R\cup\{\pm \infty\}$ as follows:
\[
\mathcal{L}_{\epsilon} f(x) = \epsilon \log \left(\int e^{\frac{\langle x,y \rangle-f(y)}{\epsilon}}\,dy\right),\qquad \forall x\in \R^n.
\]
The second main contributions of this paper are Propositions \ref{prop:rho-convex} and \ref{prop:sigma-smooth} and can be be summarized as follows: if $f:\R^n \to \R$ and $g:\R^n \to \R\cup\{+\infty\}$, then 
\begin{enumerate}
\item[(i)] if $f$ is $\sigma$-smooth, then $\mathcal{L}_{\epsilon}f$ is $\sigma^*$-convex,
\item[(ii)] if $g$ is $\rho$-convex, then $\mathcal{L}_{\epsilon}g$ is $\rho^*$-smooth.
\end{enumerate}
Note that, as $\epsilon \to 0$, the rescaled log-Laplace operator $\mathcal{L}_\epsilon$ converges to the Fenchel-Legendre transform. Letting $\epsilon\to 0$ in $(i)$ and $(ii)$, one recovers the well known property \cite{AP} that the sets of $\sigma$-smooth and $\rho$-convex functions are in duality with respect to the Fenchel-Legendre conjugation (see Proposition \ref{prop:AP2} for a precise statement). 
While Property $(i)$ is a simple consequence of H\"older inequality, the proof of Property $(ii)$ requires the use of the Prékopa-Leindler inequality. Using $(i)$ and $(ii)$, the fixed point system of equations satisfied by $(\phi_\epsilon,\psi_\epsilon)$ yields an estimate of the modulus of smoothness of $\phi_\epsilon$ that does not depend on $\epsilon$, which eventually allows to get a global regularity estimate for $\nabla \phi$ as explained above.
More precisely, to obtain information on the  smoothness modulus  $\sigma_{\phi_\epsilon}$ of the potential $\phi_\epsilon$ (that is defined as the smallest $\sigma'$ for which $\phi_\epsilon$ is $\sigma'$-smooth), the key step is to establish the following fixed point inequality of independent interest:
\begin{equation}\label{eq:fixed-point}
\sigma_{\phi_\epsilon} \leq (\sigma_{\phi_\epsilon}+ \epsilon \sigma) \square \epsilon \rho^*(\,\cdot\,/\epsilon),
\end{equation}
where $\square$ denotes the infimum convolution operator. The relation \eqref{eq:fixed-point} follows by analysing the system linking $\phi_\epsilon$ and $\psi_\epsilon$ with the help of the duality properties $(i)$ and $(ii)$ stated above. Then, $\epsilon$ independent bounds for $\sigma_{\phi_\epsilon}$, such as $\bar{\sigma}$ can be deduced from \eqref{eq:fixed-point} using elementary arguments.

The scheme of proof sketched above turns out to be very flexible. It can be modified to allow directional dependencies in the convexity and smoothness moduli, which then allows to derive anisotropic control on the Brenier transport maps. The most general version of Caffarelli's contraction theorem obtained in this work is stated in Theorem \ref{thm:phi-epsilon-dir-smoothness} and Corollary \ref{cor:OT-reg-dir}. As a first application of this general principle, we prove in Theorem \ref{thm:CP}, that if $\mu(dx) = e^{-V(x)}dx$ and $\nu(dy) = e^{-W(y)}dy$ are such that $\nabla^2V \leq A^{-1}$ and $\nabla^2 W \geq B^{-1}$ with $A,B$ two symmetric positive definite matrices, then the Brenier map $\nabla \phi$ from $\mu$ to $\nu$ satisfies
\[
    \nabla^2 \phi \leq B^{1/2}\left(B^{-1/2}A^{-1}B^{-1/2}\right)^{1/2}B^{1/2},
\]
with equality in the case $\mu=\mathcal{N}(0, A)$ and $\nu = \mathcal{N}(0,B)$.
This result was first obtained by Chewi and Pooladian \cite{Chewi2022entropic} (generalizing a preceding result by Valdimarsson \cite{Val07}) under the additional assumption that the matrices $A,B$ commute, which is removed here. A second application is considered in Theorem \ref{thm:lpnorm}, where a new version of the Caffarelli contraction theorem involving $p$-norms is obtained.

Another modification of the entropic method, based on a slightly modified Prékopa-Leindler inequality (stated in Proposition \ref{prop:QPL}), enables us to recover and improve a recent result by De Philippis and Shenfeld \cite{DePS24} on the Laplacian of the Brenier map between a log-subharmonic probability measure $\mu$ and a strongly log-concave probability measure $\nu$. More precisely, De Philippis and Shenfeld proved that whenever $\mu(dx) = e^{-V(x)}\,dx$ with $\Delta V \leq \alpha_V n$ and $\nu(dy) = e^{-W(y)}\,dy$ with $\nabla^2W \geq \beta_W I_n$ with $\alpha_V,\beta_W >0$, then the Kantorovich potential $\varphi$ is such that 
\[
\Delta \varphi \leq \sqrt{\frac{\alpha_V}{\beta_W}}n.
\]
As we will see, this bound on the Laplacian actually holds under a strictly weaker assumption on $W$. We refer to Theorem \ref{thm:DePS} for a precise statement.

\medskip

The paper is organized as follows. In Section 2, we recall the definitions and properties of $\sigma$-smooth and $\rho$-convex functions. We also introduce directional variants of these notions, leading in particular to anisotropic regularity properties on the gradient. The section is completed by a collection of concrete examples. Section 3 deals with the proof of Items $(i)$ and $(ii)$ above. A directional variant is also considered. Section 4 implements the entropic regularization method for global regularity of optimal transport to prove \eqref{eq:reg-intro} as well as directional variants of this result. Section 5 presents several applications of the results of the preceding section and their links with the existing literature. Finally, Section 6 studies how $\mathcal{L}_\epsilon$ affects superharmonic functions and builds upon a recent result presented in  \cite{DePS24} on the subharmonicity of the optimal transport potential.

\bigskip

\textbf{Acknowledgments.} The authors would like to thank Giovanni Conforti and Guillaume Carlier for insightful discussions during the preparation of this work.

\section{Modulus of convexity and modulus of smoothness}
\subsection{Modulus of convexity/smoothness}
In order to extend regularity results on the hessian of (convex) functions we will look at quantification of convexity or default of convexity which will imply lower/upper bounds on the hessian if it exists. Thus we turn to the notion of modulus of convexity/smoothness studied in \cite{VNC78, Zal83, AP}. We refer to \cite[Chapter 3, Section 3.5]{Zal02} for a synthetic presentation of these notions. These notions of moduli of convexity and smoothness are close to the notions of semi concavity/convexity \cite{Cannarsa2004}. 
 
In all what follows, $f: \R^n \to \mathbb{R}\cup \{+\infty\}$ will be a function whose domain $\mathrm{dom}f = \{x\in \R^n : f(x)<+\infty\}$ is assumed to be convex. For any $x_0,x_1 \in \R^n$ and $t \in ]0,1[$, introduce the mean deviation of $f$ as
\begin{equation*}
    M_t^f(x_0,x_1) = (1-t)f(x_0)+tf(x_1) - f((1-t)x_0+tx_1).
\end{equation*}
This is a well defined quantity in the following two cases: 
\begin{itemize}
\item $x_0,x_1 \in \mathrm{dom}\,f$, in which case, by convexity of the domain, $M_t^f(x_0,x_1) \in \R$
\item $(1-t)x_0+tx_1 \in \mathrm{dom}f$, in which case $M_t^f(x_0,x_1)\in \R\cup \{+\infty\}$.
\end{itemize}
Following \cite{VNC78, Zal83, AP}, we introduce now the \emph{modulus of convexity} of $f$ as 
\begin{equation*}
    \rho_f(r) = \inf\left\{\frac{1}{t(1-t)}M_t^f(x_0,x_1)\mid t \in ]0,1[ , \vert x_0-x_1 \vert = r,x_0,x_1 \in \mathrm{dom}f\right\},\qquad r\geq0
\end{equation*}
and the \emph{modulus of smoothness} of $f$ as
\[
    \sigma_f(r) = \sup\left\{\frac{1}{t(1-t)}M_t^f(x_0,x_1)\mid t \in ]0,1[ , \vert x_0-x_1 \vert = r,x_0,x_1 \text{s.t.} (1-t)x_0+tx_1 \in \mathrm{dom} f\right\}, \qquad r\geq0,
\]
where, in all the paper, $\vert\,\cdot\,\vert$ denotes the standard Euclidean norm on $\R^n$.
Note that by definition $\rho_f(0) = \sigma_f(0)=0$. Note also that, for $r>0$, $\rho_f(r)\in \R\cup\{\pm\infty\}$ and takes the value $+\infty$ if $r$ is greater than the diameter of $\mathrm{dom}f$. For $r>0$, $\sigma_f(r) \in \R\cup\{\pm\infty\}$ and $\sigma_f(r)=-\infty$ for some $r>0$ if and only if the domain of $f$ is empty.

The function $\rho_f$ is the greatest function $\rho:\R_+ \to \R\cup\{\pm\infty\}$ such that 
\begin{equation}\label{eq:rho-convex}
f((1-t)x_0+tx_1) +t(1-t) \rho(\vert x_1-x_0 \vert)\leq (1-t)f(x_0)+tf(x_1),\qquad\forall  t \in ]0,1[, \forall x_0,x_1 \in \mathrm{dom} f,
\end{equation}
and the function $\sigma_f$ is the smallest function $\sigma:\R_+ \to \R\cup\{\pm\infty\}$ such that
\begin{equation}\label{eq:sigma-smooth-strong}
f((1-t)x_0+tx_1)+t(1-t) \sigma(\vert x_1-x_0 \vert)\geq (1-t)f(x_0)+tf(x_1),\qquad \forall  t \in ]0,1[, \forall x_0,x_1 \in \R^n \text{ s.t } (1-t)x_0+tx_1 \in \mathrm{dom}f.
\end{equation}
When \eqref{eq:rho-convex} is satisfied for some function $\rho$, we say that $f$ is $\rho$-convex, and when \eqref{eq:sigma-smooth-strong} is satisfied for some function $\sigma$, we say that $f$ is $\sigma$-smooth.
In particular, the function $f$ is convex if and only if $\rho_f \geq0$.
The finiteness of $\sigma_f$ ensures that $f$ has full domain as shown in the following simple result (see also Lemma 2.3 of \cite{AP} and Proposition 3.5.2 of \cite{Zal02}).
\begin{lemma}
If $\mathrm{dom}f$ is nonempty and convex and $\sigma_f(r_o) <+\infty$ for some $r_o >0$, then $\mathrm{dom} f= \R^n$.
\end{lemma}
For the sake of completeness, we include the short proof of this elementary fact .
\proof
Note that $\sigma_f$ admits the following slightly different expression:
\begin{equation}\label{eq:Zal-alt}
\sigma_f(r) = \sup\left\{\frac{(1-t)f(x-try)+tf(x+(1-t)ry)-f(x)}{t(1-t)} \mid t\in ]0,1[, x \in \mathrm{dom}\,f, \vert y \vert = 1\right\}.
\end{equation}
In particular, if $\sigma_f(r_0)<+\infty$ and $x \in \mathrm{dom}f$, then $B(x,r_0) \subset \mathrm{dom}f$, denoting by $B(x_0,r)$ the closed ball of radius $r_0$ centered at $x$. In other words, $\mathrm{dom}f + B(0,r_0) \subset \mathrm{dom} f$. Iterating, yields that $\mathrm{dom}f=\R^n.$
\endproof

In the applications, we will only consider functions $f$ for which $\sigma_f >0$ on $\R_+^*$, because of the following lemma.
\begin{lemma}
Let $f:\R^n \to \R$ be such that $\sigma_f(r_0) \leq 0$ for some $r_0>0$, then the function $e^{-f}$ is not integrable.
\end{lemma}
\begin{remark}\label{rem:sigma_V}
It follows immediately from the lemma that, if $\mu(dx) = e^{-V(x)}\,dx$ is a probability measure on $\R^n$ such that $\sigma_V <+\infty$, then one also has $\sigma_V>0$ on $\R_+^*.$
\end{remark}
\proof
Suppose that $\sigma_f(r_0) \leq 0$ for some $r_0>0$. Take $u \in \R^n$ such that $\vert u\vert=r_0$ ; then we get
\[
e^{-f(x+tu)} \leq e^{-(1-t)f(x)}e^{-tf(x+u)},\qquad \forall x\in \R^n.
\]
So integrating with respect to $x$ and applying H\"older inequality yields
\[
\int e^{-f} \leq \int e^{-(1-t)f(x)}e^{-tf(x+u)}\,dx \leq \left(\int e^{-f}\,dx\right)^{1-t} \left(\int e^{-f}\,dx\right)^{t} = \int e^{-f}\,dx.
\]
If $\int e^{-f} <+\infty$, then, by the equality case in H\"older inequality, one gets that that there is some constant $\lambda>0$ such that $e^{-f(x)} =\lambda e^{-f(x+u)}$ for almost every $x\in \R^n$. Integrating with respect to $x$ gives a contradiction. Thus $e^{-f}$ is not integrable.
\endproof
Moreover note that the $\sigma$-smoothness of $f$ implies a control on the growth at infinity.
\begin{lemma}\label{lem:sigma-growth}
Let $f:\R^n\to \R$ be a $\sigma$-smooth function with $\sigma:\R_+\to \R_+$. Let $x\in \R^n$ such that $f$ is differentiable at $x$. Then for any $y \in \R^n$ we have
\begin{equation*}
    f(y) \leq f(x) + \langle \nabla f(x), y-x\rangle + \sigma(\vert y-x\vert)
\end{equation*}
\end{lemma}

\proof
Let $x,y \in \R^n$ and for $t\in ]0,1[$ set $x_t = (1-t)x+ty$. By $\sigma$-smoothness of $f$ we have
\begin{equation*}
    f(y)-f(x)-\frac{f(x_t)-f(x)}{t} \leq (1-t)\sigma(\vert y-x \vert)
\end{equation*}
Letting $t\to0$ grants the result.
\endproof
Recall that if $h:\R^n \to \R\cup\{+\infty\}$, then $h^*$ denotes its Legendre transform defined by 
\[
h^*(y) = \sup_{x\in \R^n} \{\langle x,y\rangle - h(x)\},\qquad y\in \R^n.
\]
In the case of a function $\alpha : \R_+ \to \R\cup\{\pm\infty\}$, $\alpha^*$ denotes the monotone conjugate of $\alpha$, defined by 
\[
\alpha^*(u) = \sup_{t\geq0} \{tu - \alpha(t)\},\qquad u\geq0.
\]
The following result shows that $\rho$-convexity and $\sigma$-smoothness are properties dual from each other:

\begin{proposition}\label{prop:AP2}Let $f,g:\R^n \to \R\cup\{+\infty\}$ be two functions with  nonempty convex domains and $\sigma,\rho: \mathbb{R}_+ \to \mathbb{R}\cup\{\pm\infty\}$.
\begin{itemize}
\item[(i)] If $f$ is  $\sigma$-smooth, then $f^*$ is $\sigma^*$-convex.
\item[(ii)] If $g$ is $\rho$-convex, then $g^*$ is  $\rho^*$-smooth.
\end{itemize}
\end{proposition}
The proof of Proposition \ref{prop:AP2} for convex functions can be found in \cite[Proposition 2.6]{AP} or \cite[Proposition 3.5.3]{Zal02}. In the general case, the proof is exactly the same. The argument is recalled, for the sake of completeness, in the proof of the more general Proposition \ref{prop:directional-AP2} below.

In this paper, we will mainly be interested in the situation where $f$ is a convex function. In this case, the following result gathers known properties of $\sigma_f$ and $\rho_f$.

\begin{proposition}\label{prop:AP1}Let $f:\R^n \to \R\cup \{+\infty\}$ be a proper convex and lower semicontinuous function.
\begin{itemize}
\item[(a)] The modulus $\rho_f:\R_+ \to \R_+\cup\{+\infty\}$ is such that
\[
\rho_f(ct)\geq c^2\rho_f(t),\qquad \forall c\geq 1, \forall t\geq0,
\]
and is, in particular, non-decreasing. 
\item[(b)] The modulus $\sigma_f:\R_+ \to \R_+\cup\{+\infty\}$ is convex, non-decreasing and lower semicontinuous
and it holds $\sigma_f = (\rho_{f^*})^*$.
\item[(c)] Either $\mathrm{dom} \sigma_f = \{0\}$ or $\mathrm{dom} \sigma_f = \R_+$.
\end{itemize}
\end{proposition}

For the sake of completeness, we include a sketch of proof of Proposition \ref{prop:AP1} (we refer to \cite{Zal02} or \cite{AP} for detailed arguments).
\proof Item $(a)$ was first proved in \cite{VNC78} ; a proof can be found in \cite[Proposition 3.5.1]{Zal02}. See also Item $(a)$ of Proposition \ref{prop:directional-AP1} for a proof in a more general context. Let us prove $(b)$. Since $f$ is a proper convex and lower semicontinuous function, it satisfies $f=f^{**}$. Thus, applying $(i)$ and $(ii)$ of Proposition \ref{prop:AP2}, we first conclude that $f$ is $\sigma_f^{**}$ convex.
Thus, by minimality of $\sigma_f$, $\sigma_f^{**} \geq \sigma_f$. Since the inequality $\sigma_f^{**} \leq \sigma_f$ is always true, one gets that $\sigma_f^{**}=\sigma_f$, which proves that $\sigma_f$ is convex and lower semicontinuous (a property that can also be derived directly from \eqref{eq:Zal-alt}). On the other hand, by $(ii)$ of Proposition \ref{prop:AP2}, $f$ is also  $(\rho_{f^*})^*$-smooth, and so $\sigma_f \leq (\rho_{f^*})^*$.
Since $f^*$ is $\sigma_f^*$-convex, one gets, by maximality of $\rho_{f^*}$, that $\rho_{f^*} \geq \sigma_f^*$ and so, taking the conjugate, $\sigma_f = \sigma_f^{**} \geq (\rho_{f^*})^*$, which proves that $\sigma_f = (\rho_{f^*})^*$ (and in particular $\sigma_f$ is non-decreasing as a monotone conjugate). Now, let us sketch the proof of $(c)$. By conjugating the super-homogeneity property of  convexity modulus (Item $(a)$), one sees that $\sigma_f$ satisfies the following subhomogeneity property: for all $r \geq0$ and $c\in [0,1]$, it holds $\sigma_f(cr) \geq c^2 \sigma_f(r)$.   Therefore, if $\sigma_f(u) <+\infty$ then $\sigma_f(2u)<+\infty$ which yields the claim.
\endproof

Let us conclude this section by recalling the following interesting consequence of $\sigma$-smoothness, that will play a crucial role in Section \ref{sec:OT}.
Recall that a subgradient of $f$ at some point $x \in \mathrm{dom}f$ is a vector $a \in \R^n$ such that
\begin{equation}\label{eq:subgradient}
f(y) \geq f(x)+\langle a, y-x \rangle,\qquad \forall y\in \R^n.
\end{equation}
We denote by $\partial f(x)$ the set of all subgradients of $f$ at $x$. Note that $\partial f(x)$ may be empty. Whenever $f$ is convex and $x$ belongs to the relative interior of the domain of $f$, is is well known that $\partial f(x) \neq \emptyset$.

\begin{proposition}\label{prop:continuity} Let $f:\R^n \to \R \cup \{+\infty\}$ be a function with a nonempty convex domain.
Then for all $x_0\neq x_1 \in \mathrm{dom} f$ and $y_0 \in \partial f(x_0)$, $y_1 \in \partial f(x_1)$, it holds 
\begin{equation}\label{eq:continuity-improved}
\sigma_f^*(\vert y_1-y_0\vert) \leq \sigma_f(\vert x_1-x_0\vert)
\end{equation}
and
\begin{equation}\label{eq:continuity}
\vert y_1-y_0\vert \leq 2 \frac{\sigma_f(\vert x_1-x_0\vert)}{\vert x_1-x_0\vert}.
\end{equation}
\end{proposition}
Note that \eqref{eq:continuity-improved} implies \eqref{eq:continuity} because $\sigma^\ast(s) \geq rs - \sigma(r)$. Observe also that if $f = \frac{\vert\,\cdot\,\vert^2}{2}$, then $\sigma_f(r) = \frac{r^2}{2}$ and there is equality in \eqref{eq:continuity-improved} and \eqref{eq:continuity}.
Proposition \ref{prop:continuity} is usually stated for convex functions, but is true also for non-convex functions (provided the subdifferentials at $x_0$ and $x_1$ are assumed nonempty). The arguments of the proof of Proposition \ref{prop:continuity} can be found in \cite[Theorem 3.5.6]{Zal02}. We will give a more general result in Proposition \ref{prop:directional-regularity} below, with a complete proof.

\subsection{Directional moduli of convexity and smoothness}
The notions of modulus of convexity and modulus of smoothness introduced in the last section only depend on the distance separating the endpoints between which the mean deviation is computed. Here, in order to derive results in anisotropic situations, we introduce a generalized version of the moduli which also depends on the direction. 

Let $f : \R^n \to \R \cup \{+ \infty\}$ be a function with convex domain. The \emph{directional modulus of convexity} of $f$ is
\begin{equation*}
    R_f(d) = \inf \left\{\frac{1}{t(1-t)}M_t^f(x_0,x_1) \mid t \in ]0,1[, x_1-x_0 = d, x_0,x_1 \in \mathrm{dom}f\right\}, \qquad d \in \R^n
\end{equation*}
and the \emph{directional modulus of smoothness} is
\begin{equation*}
    S_f(d) = \sup \left\{\frac{1}{t(1-t)}M_t^f(x_0,x_1) \mid t \in ]0,1[, x_1-x_0 = d, (1-t)x_0+tx_1 \in \mathrm{dom}f\right\}, \qquad d \in \R^n.
\end{equation*}
By their very definition, the functions $R_f$ and $S_f$ are automatically even:
\[
R_f(-d)=R_f(d)\qquad \text{and} \qquad S_f(-d)=S_f(d),\qquad \forall d\in \R^n.
\]
As in the non directional case $R_f$ is the greatest even function $R: \R^n \to \R \cup \{\pm \infty\}$ such that
\begin{equation}\label{eq:dir-rho-convex}
f((1-t)x_0+tx_1)+t(1-t) R(x_1-x_0)\leq (1-t)f(x_0)+tf(x_1), \quad\forall  t \in ]0,1[, \forall x_0,x_1 \in \mathrm{dom} f.
\end{equation}
If \eqref{eq:dir-rho-convex} is satisfied, we will say that $f$ is directionally $R$-convex or $R$-convex if there is no ambiguity. Note that $R$ is allowed to take negative values, and therefore an $R$-convex function is not necessarily convex.

Likewise $S_f$ is the smallest even function $S:\R^n \to \R \cup \{\pm \infty\}$ such that
\begin{equation}\label{eq:dir-sigma-smooth}
f((1-t)x_0+tx_1)+t(1-t) S( x_1-x_0 )\geq (1-t)f(x_0)+tf(x_1),\qquad \forall t\in ]0,1[,\qquad \forall x_0,x_1 \in \R^n \text{ s.t } (1-t)x_0+tx_1 \in \mathrm{dom}f.
\end{equation}
If \eqref{eq:dir-sigma-smooth} is satisfied, we will say that $f$ is directionally $S$-smooth or $S$-smooth if there is no ambiguity. When an $S$-smooth function $f$ has a nonempty domain, then $S>-\infty$ everywhere. Note that as before the smoothness of a function has a strong implication on its domain.
\begin{lemma}\label{lem:smooth-domain}
If $\mathrm{dom}f$ is nonempty and convex and $S_f(v) < +\infty$ for some $v$, then $\mathrm{dom}f+\mathrm{Vect}(v)\subset \mathrm{dom}f$.
\end{lemma}
\proof
Let $x \in \mathrm{dom}f$. Since $f$ is $S_f$ smooth, then for $t =1/2$
\begin{equation*}
    +\infty > f(x) + \frac{1}{4}S_f(v) \geq \frac{1}{2}\left(f(x-\frac{1}{2}v)+f(x+\frac{1}{2}v)\right).
\end{equation*}
Thus $x-\frac{1}{2}v,x+\frac{1}{2}v \in \mathrm{dom}f$ for any $x$ in the domain of $f$. The convexity of the domain ensures that $\mathrm{dom}f + \frac{1}{2}[-v,v]\subset \mathrm{dom}f$ where $[a,b]$ is the line joining $a$ and $b$. The result follows by iteration. 
\endproof
The results that hold for the moduli of convexity and smoothness still apply to the directional moduli. 
We will make repeated use of the following proposition which extends Proposition \ref{prop:AP2}.
\begin{proposition}\label{prop:directional-AP2}
Let $f,g: \R^n \to \R\cup \{+\infty\}$ be two functions with nonempty convex domains and $S,R :\R^n \to \R \cup \{\pm\infty\}$ be even functions.
\begin{itemize}
\item[(i)] If $f$ is  $S$-smooth, then $f^*$ is $S^*$-convex.
\item[(ii)] If $g$ is $R$-convex, then $g^*$ is  $R^*$-smooth.
\end{itemize}
\end{proposition}
\proof
We begin with the proof of $(i)$. Since the domain of $f$ is nonempty, $S >-\infty$ everywhere. 
Let us prove that for all $y_0,y_1 \in \mathrm{dom} f^\ast$ and $t \in ]0,1[$, it holds
\begin{equation*}
    f^\ast(y_t) + t(1-t)S^\ast(y_1-y_0) \leq (1-t)f^\ast(y_0) +t f^\ast(y_1),
\end{equation*}
where $y_t := (1-t)y_0+ty_1.$ 
By convexity of $f^*$, one gets $f^\ast(y_t) <+\infty$, and since the domain of $f$ is nonempty, $f^\ast(y_t) > - \infty$.
Let $\epsilon > 0$ ; there is $x$ such that
\begin{equation*}
    -\infty < f^\ast(y_t) \leq \langle y_t,x \rangle - f(x) + \epsilon.
\end{equation*}
Thus $f(x) < + \infty$ and by $S$-smoothness of $f$ for $d\in \mathrm{dom}S$ we have
\begin{equation*}
    f(x) + t(1-t)S(d) \geq (1-t)f(x-td) + tf(x+(1-t)d).
\end{equation*}
Combining the two preceding inequalities grants
\begin{equation*}
    f^\ast(y_t) - \epsilon \leq \langle y_t,x\rangle - (1-t)f(x-td) - tf(x+(1-t)d) + t(1-t)S(d).
\end{equation*}
Thus by definition of $f^\ast$
\begin{align*}
    f^\ast(y_t) - \epsilon &\leq \langle y_t,x\rangle - (1-t)\langle y_0,x-td\rangle - t\langle y_1,x+(1-t)d\rangle \\
    &+ (1-t)f^\ast(y_0) + tf^\ast(y_1) + t(1-t)S(d)\\
    &\leq t(1-t)(S(d) - \langle y_1-y_0,d\rangle ) + (1-t) f^\ast(y_0) + tf^\ast(y_1).
\end{align*}
Finally minimizing over $d$ grants
\begin{equation*}
    f^\ast((1-t)y_0+ty_1) + t(1-t)S^\ast(y_1-y_0) -\epsilon \leq (1-t)f^\ast(y_0) +t f^\ast(y_1) 
\end{equation*}
The result follows by letting $\epsilon \to 0$.

The proof of $(ii)$ is similar but we include it for completeness. Let us prove that for all $y_0,y_1 \in \R^n$ such that $y_t = (1-t)y_0+ty_1 \in \mathrm{dom}g^\ast$, it holds 
\[
g^*(y_t) + t(1-t)R^*(y_1-y_0) \geq (1-t)g^*(y_0)+t g^*(y_1).
\]
Let $x_0,x_1 \in \mathrm{dom}f$. By $R$-convexity of $g$ we have
\begin{equation*}
    (1-t) g(x_0) + t g(x_1) \geq t(1-t) R(x_1-x_0) + g(x_t)
\end{equation*}
with $x_t = (1-t)x_0+tx_1$.
Thus
\begin{align*}
    (1-t)(\langle x_0,y_0 \rangle - g(x_0)) + t (\langle x_1,y_1 \rangle - g(x_1)) &\leq (1-t) \langle x_0,y_0 \rangle + t \langle x_1,y_1 \rangle -g(x_t) - t(1-t) R(x_1-x_0)\\
    &\leq t(1-t) (\langle x_1-x_0,y_1-y_0\rangle - R(x_1-x_0)) + g^\ast(y_t)\\
    &\leq t(1-t) R^\ast(y_1-y_0) + g^\ast(y_t)
\end{align*}
where the second and last inequalities hold by definition of $g^\ast$ and $R^\ast$. Finally optimizing over $x_0,x_1$ grants the result.
\endproof

In case the function $f$ is convex, Proposition \ref{prop:AP1} can also be extended in the following way.
\begin{proposition}\label{prop:directional-AP1}Let $f:\R^n \to \R\cup \{+\infty\}$ be a proper convex and lower semicontinuous function.
\begin{itemize}
\item[(a)] The modulus $R_f:\R^n \to \R_+\cup\{+\infty\}$ is such that
\[
R_f(cv)\geq c^2R_f(v),\qquad \forall c\geq 1, \forall v \in \R^n.
\]
\item[(b)] The modulus $S_f:\R^n \to \R_+\cup\{+\infty\}$ is convex, non-decreasing and lower semicontinuous
and it holds $S_f = (R_{f^*})^*$.
\item[(c)] The domain of $S_f$ is a vector space.
\end{itemize}
\end{proposition}

\proof 
We start with the proof of item $(a)$. Let $c \in ]1,2[$ and $v \in \R^n\setminus\{0\}$. First if $R_f(cv) = +\infty$ there is nothing to prove. Otherwise let $\epsilon > 0$, then there are $x_0,x_1$ and $t \in ]0,\frac{1}{2}[$ such that $x_1-x_0 = cv$ and
\begin{equation*}
    t(1-t)(R_f(cv) + \epsilon) > M^f_t(x_0,x_1).
\end{equation*}
Set $x_t = (1-t)x_0+tx_1$ and $x_c = (1-c^{-1})x_0+c^{-1}x_1$. Notice that $x_c-x_0 = v$ and $x_t = (1-ct)x_0 + ct x_c$. By $R_f$-convexity of $f$ we have
\begin{align*}
     M^f_t(x_0,x_1) &= (1-t)f(x_0) +tf(x_1) - ((1-ct)f(x_0)+ctf(x_c)) + M_{ct}^f(x_0,x_c) \\
     &\geq ct\left[(1-c^{-1})f(x_0) + c^{-1}f(x_1) - f(x_c)\right] + ct(1-ct) R_f(v)\\
     &\geq ct(1-c^{-1})c^{-1} R_f(cv)+ct(1-ct) R_f(v)
\end{align*}
because $ct \in ]0,1[$. Thus using the first inequality we deduce by letting $\epsilon \to 0$
\begin{equation*}
    t(1-t) R_f(cv) \geq ct(1-c^{-1})c^{-1} R_f(cv)+ct(1-ct) R_f(v)
\end{equation*}
Dividing by $t$ and rearranging the terms grants
\begin{equation*}
    \frac{1-ct}{c}R_f(cv) \geq c(1-ct)R_f(v)
\end{equation*}
which is the wanted inequality. An induction ensures that $R_f(c^n v) \geq (c^n)^2R_f(v)$ which gives the result.
Let us prove $(b)$. Since $f$ is a proper convex and lower semicontinuous function, it satisfies $f=f^{**}$. Thus, applying $(i)$ and $(ii)$ of Proposition \ref{prop:directional-AP2}, we first conclude that $f$ is $S_f^{**}$ convex.
Thus, by minimality of $S_f$, $S_f^{**} \geq S_f$. Since the inequality $S_f^{**} \leq S_f$ is always true, one gets that $S_f^{**}=S_f$, which proves that $S_f$ is convex and lower semicontinuous. On the other hand, by $(ii)$ of Proposition \ref{prop:directional-AP2}, $f$ is also  $(R_{f^*})^*$-smooth, and so $S_f \leq (R_{f^*})^*$.
Since $f^*$ is $S_f^*$-convex, one gets, by maximality of $R_{f^*}$, that $R_{f^*} \geq S_f^*$ and so, taking the conjugate, $S_f = S_f^{**} \geq (R_{f^*})^*$, which proves that $S_f = (R_{f^*})^*$. 
Now let us prove $(c)$. First $0\in \mathrm{dom}S_f$. By the preceding $S_f$ is a convex function thus its domain is convex. Moreover it is even and in order to prove that its domain is a vector space it remains to prove that it is a cone. If $\mathrm{dom}S_f = \{0\}$ there is nothing to prove. Otherwise let $d \in \mathrm{dom}S_f$ and $c > 1$ then since $S_f = R_{f^\ast}^\ast$ we have
\begin{align*}
    S_f(cd) &= \sup_v \langle cd,v\rangle - R_{f^\ast}(v)\\
    &= \sup_v \langle cd,cv\rangle - R_{f^\ast}(cv)\\
    &\leq \sup_v c^2\langle d,v\rangle - c^2R_{f^\ast}(v) = c^2S_f(d)
\end{align*}
where the inequality holds by item $(a)$. This inequality shows that $cd \in \mathrm{dom}S_f$. The case $c\in [0,1]$ is managed by convexity of the domain. Finally the domain is a convex cone containing $0$ thus it is a vector space.
\endproof

The following classical lemma is useful to establish $R$-convexity for convex functions.

\begin{lemma}\label{lem:BL}
Let $f:\R^n \to \R\cup \{+\infty\}$ be a proper convex and lower semicontinuous function, and $R : \R^n \to \R_+\cup \{+\infty\}$. 
If the function $f$ is such that
\begin{equation}\label{eq:BL}
f\left(\frac{x+y}{2}\right) \leq \frac{1}{2}f(x)+\frac{1}{2}f(y) - \frac{1}{2}R(y-x),\qquad \forall x,y \in \mathrm{dom} f,
\end{equation}
then, $f$ is $R$-convex.
\end{lemma}
\proof
The simple argument below is taken from \cite{BL00} and is recalled for the sake of completeness.
Assume that $f$ satisfies \eqref{eq:BL} and fix $x,y\in \mathrm{dom} f$, and consider the inequality 
\begin{equation}\label{eq:BL+}
(1-t)f(x)+tf(y)-f((1-t)x+ty)\geq  \min(t ; 1-t)R(y-x).
\end{equation}
 For $t \in [0,1/2]$, the function in the left hand side of \eqref{eq:BL+} is concave in $t$ while the function in the right hand side is linear. The inequality \eqref{eq:BL+} holds for $t=0$ and for $t=1/2$, so it holds for all $t \in [0,1/2]$. Similarly, \eqref{eq:BL+} holds for all $t \in [1/2,1]$. Finally, since $R$ is non-negative and $\min(t ; 1-t) \geq (1-t)t$, \eqref{eq:BL+} implies that $f$ is $R$-convex.
\endproof

We are now ready to state a directional estimate on the gradient of $S$-smooth functions which is similar to Proposition \ref{prop:continuity} (and implies it). Recall the definition of the subgradient of a function given in \eqref{eq:subgradient}.
\begin{proposition}\label{prop:directional-regularity}
Let $f: \R^n \to \R\cup \{+\infty\}$ be a function with a nonempty convex domain and $S:\R^n \to \R \cup \{+\infty\}$ be even. If $f$ is $S$-smooth, then for all $x_0, x_1 \in \mathrm{dom}f$ and $y_0 \in \partial f(x_0),y_1 \in \partial f(x_1)$, it holds 
\begin{equation*}
    S^\ast(y_1-y_0) \leq S(x_1-x_0).
\end{equation*}
\end{proposition}
Note that in Section \ref{sec:OT}, this result will be applied to a Kantorovich potential, which is naturally convex.  
\proof
Let $x_0,x_1 \in \mathrm{dom}f$ and $y_0 \in \partial f(x_0)$, $y_1 \in \partial f(x_1)$. 
If follows from the definition of a subgradient that, for $i=0,1$, $f^*(y_i) = \langle y_i, x_i \rangle - f(x_i) <+\infty$. Thus $y_0,y_1 \in \mathrm{dom} f^*$. 
Since $f$ is $S$-smooth, it follows from Proposition \ref{prop:directional-AP2} that the function $f^\ast$ is $S^\ast$-convex. Thus, for any $t \in ]0,1[$,
\begin{equation*}
    (1-t)f^\ast(y_0) + t f^\ast(y_1) \geq t(1-t) S^\ast(y_1-y_0) + f^\ast(y_t),
\end{equation*}
where $y_t = (1-t)y_0+ty_1$.
Moreover, since $y_0 \in \partial f(x_0)$ we have $x_0 \in \partial f^\ast(y_0)$ (this is always true, even if $f$ is not convex) and thus
\begin{equation*}
    (1-t)f^\ast(y_0) + t f^\ast(y_1) \geq t(1-t) S^\ast(y_1-y_0) + f^\ast(y_0) + t\langle x_0,y_1-y_0\rangle.
\end{equation*}
Subtracting $f^\ast(y_0)$, dividing by $t$ and letting $t \to 0$ grants
\begin{equation*}
    f^\ast(y_1) - f^\ast(y_0) \geq S^\ast(y_1-y_0) + \langle x_0,y_1-y_0\rangle.
\end{equation*}
A symmetric statement holds by exchanging $y_0,y_1$ and $x_0,x_1$. Since $S^\ast$ is even, summing the two inequalities obtained that way gives
\begin{equation*}
    \langle x_1-x_0,y_1-y_0\rangle \geq 2S^\ast(y_1-y_0)
\end{equation*}
and thus
\begin{equation*}
    S^\ast(y_1-y_0) \leq \langle x_1-x_0,y_1-y_0\rangle - S^\ast(y_1-y_0) \leq S(x_1-x_0)
\end{equation*}
which is the desired result.
\endproof

\subsection{Examples of $R$-convex and $S$-smooth functions}\label{sec:ex-R-conv-S-smooth}
As we have seen in Propositions \ref{prop:continuity} and \ref{prop:directional-regularity}, the modulus of smoothness implies some form of regularity of the subgradient of a function (similarly, the modulus of convexity is related to expansivity properties of the subgradient). 
In this section, we investigate the converse implication and show how regularity of the (sub)-gradient implies estimates on the smoothness/convexity moduli for some classical class of functions.
\paragraph{\textbf{Quadratic functions.}}
Let $\alpha \in \R$ and $f:\R^n \to \R$ such that $f(x) = \alpha \vert x \vert^2$ then 
\begin{equation*}
    \rho_f(r) = \sigma_f(r) = \alpha r^2, \quad R_f(d) = S_f(d) = \alpha \vert d \vert^2.
\end{equation*}
More generally, if $A$ is a $d\times d$ matrix, then the function $f(x) = \langle x,A x\rangle$ admits the following moduli
\begin{align*}
    \rho_f(r) &= r^2 \min \mathrm{Sp}\left(\frac{A+A^T}{2}\right), \, \sigma_f(r) = r^2 \max \mathrm{Sp}\left(\frac{A+A^T}{2}\right), \qquad r\geq0\\
    R_f(d) &= S_f(d) = \langle d,A d\rangle,\qquad d\in \R^n,
\end{align*}
where $\mathrm{Sp} (M)$ denotes the spectrum of a matrix $M$.

\paragraph{Functions with a bounded hessian.} Let $f$ be a twice continuously differentiable function on $\R^n$; we have the following equality for any $x_0,x_1 \in \R^n$
\begin{equation*}
    M^f_t(x_0,x_1) = \int_0^1 \langle x_1-x_0,\nabla^2f(x_s) (x_1-x_0)\rangle \min(s(1-t),t(1-s)) ds
\end{equation*}
where $x_s = x_0 + s(x_1-x_0)$. Since $\int_0^1 \min(s(1-t),t(1-s)) ds = t(1-t)/2$ we deduce the following moduli
\begin{align*}
    \rho_f(r) \geq \frac{1}{2}r^2 \inf_x \min \mathrm{Sp}\nabla^2 f(x), &\quad \sigma_f(r) \leq \frac{1}{2}r^2 \sup_x \max \mathrm{Sp}\nabla^2 f(x), \qquad r\geq0\\
    R_f(d) \geq \frac{1}{2}\inf_x \langle d,\nabla^2 f(x) d \rangle, &\quad S_f(d) \leq \frac{1}{2}\sup_x \langle d,\nabla^2 f(x) d\rangle,\qquad d\in \R^n.
\end{align*}

Note that these bounds on the moduli are non trivial if and only if we have a uniform lower or upper bound on the hessian.

\paragraph{\textbf{Functions with a uniformly continuous gradient.}} When one does not have a control on the Hessian but only a uniform continuity estimate for the gradient, one can still obtain explicit bounds on the modulus of smoothness. More precisely, suppose that $f$ is a continuously differentiable function such that $\nabla f$ admits a non-decreasing modulus of continuity $\omega$ then 
\[
\sigma_f(r) \leq 2r\omega(r),\qquad r\geq0.
\]
Indeed let $x_0,x_1 \in \R^n$ and $t \in ]0,1[$. Set $x_t = (1-t)x_0+tx_1$, then for $i=0,1$,
\begin{equation*}
    f(x_i) - f(x_t) - \langle \nabla f(x_t), x_i-x_t\rangle = \int_0^1 \langle \nabla f(x_t + s (x_i -x_t))-\nabla f(x_t),x_i-x_t \rangle ds \leq \vert x_i-x_t \vert \omega(\vert x_i-x_t \vert).
\end{equation*}
Multiplying this inequality by $t$ if $i=1$ and $(1-t)$ if $i=0$ and summing grants
\begin{equation*}
    M^f_t(x_0,x_1) \leq t(1-t)\vert x_1-x_0 \vert (\omega(t\vert x_1-x_0\vert)+\omega((1-t)t\vert x_1-x_0\vert)) \leq  t(1-t)2\vert x_1-x_0 \vert \omega(\vert x_1-x_0\vert)
\end{equation*}
which proves the claim.

For instance, this case encompasses functions $f \in C^{1,\alpha}(\R^n)$, with $0<\alpha\leq 1$. Indeed, in this case $\omega(r)=\| f \|_{1,\alpha}r^\alpha$, $r\geq 0,$ with $\| f \|_{1,\alpha}$ the H\"older constant of $\nabla f$, and so we have 
\[
\sigma_f(r) \leq 2r^{1+\alpha}\| f \|_{1,\alpha},\qquad r\geq0.
\]
Finally, Lipschitz functions are also encompassed by this case. Indeed we have $\omega(r) = 2L$ with $L$ the Lipschitz constant of the function $f$.

\paragraph{\textbf{Radial functions.}}
The following result is due to Vladimirov, Nesterov and Chekanov \cite{VNC78}. It is stated, without proof, in \cite{Zal83}. Since Reference \cite{VNC78} is difficult to find, an elementary proof of this result is given in Appendix \ref{app:proof}.

\begin{proposition}\label{prop:VNC78}
Let $\alpha : \R_+\to\R_+$ be a non-decreasing function such that $\alpha(ct)\geq c\alpha(t)$ for all $t\geq0$ and $c\geq1$. Define $A(r) = \int_0^r \alpha(u)\,du$, $r\geq0$, and $f_\alpha(x)=A(\vert x\vert)$, $x\in \R^n$. Then, the function $f_\alpha$ is $\rho$-convex, with $\rho(r) = 2A(r/2)$, $r\geq0$.
\end{proposition}

\begin{corollary}\label{cor:VNC78}
Suppose that $\alpha:\R_+ \to \R_+$ is an increasing continuous function such that $\alpha(ct) \geq c\alpha(t)$ for all $t\geq 0$ and $c\geq1$. Denote by $\alpha^{-1}:\R_+ \to \R_+$ the converse function of $\alpha$. Then, the function $(f_\alpha)^*$ is $\sigma$-smooth with $\sigma(r) = 2\int_0^r \alpha^{-1}(u)\,du$, $r\geq0$.
\end{corollary}
\proof
This follows immediately from Propositions \ref{prop:AP2} and \ref{prop:VNC78}.
\endproof
\section{Entropic Legendre transform}\label{sec:Entropic-Legendre}
As we have seen in the preceding section, the class of $R$-convex and $S$-smooth functions are in duality with respect to the classical Legendre transform. The goal of this section is to show a similar correspondence for the \emph{entropic Legendre transform} introduced below. The motivation comes from Section \ref{sec:OT} where our goal will be to derive estimates on the moduli of smoothness and convexity of the entropic Kantorovich potentials which are tied together via this entropic Legendre transform.

Given $\epsilon > 0$ and a positive Borel measure $m$ on $\R^n$, the entropic Legendre transform $\mathcal{L}_{\epsilon,m}(\psi)$ of a measurable function $\psi:\R^n \to \mathbb{R}\cup\{+\infty\}$ with respect to the measure $m$ is defined by
\begin{equation*}
    \mathcal{L}_{\epsilon,m}(\psi)(x) = \epsilon \log\left(\int\exp\left(\frac{\langle x,y \rangle - \psi(y)}{\epsilon}\right)dm(y)\right),\qquad x \in \R^n.
\end{equation*} 
This function is always well defined in $\mathbb{R} \cup \{\pm \infty\}$, but might have an empty domain. 
If $\psi$ is such that $\mathrm{dom} \psi$ has a positive $m$ measure,  then $\mathcal{L}_{\epsilon,m}(\psi)(x)>-\infty$ for all $x\in \R^n$. Note that for $\epsilon=1$, $\mathcal{L}_{1,m}(\psi)$ is the log-Laplace of the measure $e^{-\psi}\,dm$.
Thanks to Hölder's inequality $\mathcal{L}_{\epsilon,\mathcal{H}^n}(\psi)$ is convex and thus has a convex domain. It is also lower semicontinuous, thanks to Fatou's lemma. Thus, provided its domain is nonempty, $\mathcal{L}_{\epsilon,\mathcal{H}^n}(\psi)$ is a proper convex lower semicontinuous function.

In the limit $\epsilon\to 0$, the entropic Legendre transform approaches the classical Legendre transform. More precisely, denoting by $\mathcal{H}^n$ the Lebesgue measure on $\R^n$ and assuming that $\mathcal{L}_{\epsilon,\mathcal{H}^n}(\psi)(x) <+\infty$ for all $\epsilon>0$ small enough, then 
\[
    \mathcal{L}_{\epsilon,\mathcal{H}^n}(\psi)(x) \to \mathrm{essup}_{y\in \R^n} \{\langle x,y \rangle - \psi(y)\}
\]
as $\epsilon \to 0$. In particular, if $\psi$ is continuous and has a superlinear growth at $\infty$, then for all $x\in \R^n$
\[
    \mathcal{L}_{\epsilon,\mathcal{H}^n}(\psi)(x) \to \sup_{y\in \R^n} \{\langle x,y \rangle - \psi(y)\} = \psi^*(x)
\]
as $\epsilon \to 0$.

In what follows, we will in particular extend properties $(i)$ and $(ii)$ of Propositions \ref{prop:AP2} and \ref{prop:directional-AP2} to the entropic Legendre transform for $m=\mathcal{H}^n$.

\subsection{$\psi$ $\rho$-convex implies $\mathcal{L}_{\epsilon,\mathcal{H}^n}(\psi)$ $\rho^\ast$-smooth}
The following result is an entropic analog of property $(ii)$ of Propositions \ref{prop:AP2} and \ref{prop:directional-AP2} (and formally gives back these results by letting $\epsilon \to 0$). For any $\ell \leq n$, we will denote by $\mathcal{H}^\ell$ the $\ell$ dimensional Hausdorff measure.
\begin{proposition}\label{prop:rho-convex}\
\begin{itemize}
\item[(a)] If $\psi: \R^n \to \R\cup\{+\infty\}$ has a convex domain with positive $\mathcal{H}^n$ measure and is such that $\rho_\psi > -\infty$, then the function $\mathcal{L}_{\epsilon,\mathcal{H}^n}(\psi)$ is  $\rho_\psi^*$-smooth. In other words, $\sigma_{\mathcal{L}_{\epsilon,\mathcal{H}^n}(\psi)} \leq \rho_\psi^*$.
\item[(b)]More generally, if $\psi : \R^n \to  \R\cup\{+\infty\}$ has a nonempty convex domain and is directionally $R$-convex, that is such that
\[
\psi((1-t)y_0+ty_1) + t(1-t)R(y_1-y_0) \leq (1-t)\psi(y_0) + t \psi(y_1),\qquad \forall y_0,y_1 \in \mathrm{dom} \psi, \forall t\in [0,1],
\]
where $R:\R^n \to \R \cup\{+ \infty\}$, then denoting by $m = \mathcal{H}^\ell_{|L}$ where $L$ is the affine hull of the domain of $\psi$ and $\ell$ the dimension of $H$, the function $\mathcal{L}_{\epsilon,m}(\psi)$ is directionally $R^\ast$-smooth which means it satisfies:
\[
\mathcal{L}_{\epsilon,m}(\psi)((1-t)x_0+tx_1)+ t(1-t)R^*(x_1-x_0) \geq (1-t) \mathcal{L}_{\epsilon,m}(\psi)(x_0) + t\mathcal{L}_{\epsilon,m}(\psi)(x_1) ,
\]
for all $x_0,x_1 \in \R^n$ such that $(1-t)x_0+(1-t)x_1 \in \mathrm{dom}\mathcal{L}_{\epsilon,m}(\psi).$
\end{itemize}
\end{proposition}
The proof of this result will make use of the well known Prékopa-Leindler \cite{Pre71, Pre73, Lei72} inequality, that we now recall : if $f_0,f_1, h : L \to \R_+$ are measurable functions defined on some affine subspace $L$ of $\R^n$ such that, for some $t \in ]0,1[$, it holds
\[
h((1-t)y_0 + ty_1) \geq f_0^{1-t}(y_0)f_1^t(y_1),\qquad \forall y_0,y_1\in L
\]
then 
\[
\int_L h \geq \left(\int_L f_0\right)^{1-t}\left(\int_L f_1\right)^{t},
\]
where integration is understood with respect to the Lebesgue measure on $L$, that is $\mathcal{H}^\ell_{|L}$ with $\ell$ the dimension of $L$. This result is usually stated for $L=\R^n$ but its extension to a general affine subspace is classical and straightforward.

\proof Let us first prove Item $(b)$.
Let $x_0,x_1 \in \R^n$  and $y_0,y_1 \in \mathrm{dom}\psi$. For $t \in ]0,1[$ set $x_t = (1-t)x_0 +t x_1$, $y_t = (1-t)y_0 +t y_1$, and assume that $x_t \in\mathrm{dom}\mathcal{L}_{\epsilon,m}(\psi)$.
By $R$-convexity of $\psi$ since $R > -\infty$ we have $y_t \in \mathrm{dom}\psi$ and a direct computation grants the following inequalities
\begin{align*}
    \langle x_t,y_t \rangle - \psi(y_t) &\geq t(1-t) (R(y_1-y_0) -  \langle x_1-x_0,y_1-y_0 \rangle)+(1-t)(\langle x_0,y_0\rangle - \psi(y_0)) + t(\langle x_1,y_1\rangle - \psi(y_1))\\
    &\geq -t(1-t)R^\ast(x_1-x_0)+(1-t)(\langle x_0,y_0\rangle - \psi(y_0)) + t(\langle x_1,y_1\rangle - \psi(y_1))
\end{align*}
Note that the last inequality holds trivially if $R^\ast(x_1-x_0) = +\infty$.
Therefore, the functions $f_0,f_1,h$ defined for all $y \in L$ by
\begin{equation*}
h(y) = \exp\left(\frac{\langle x_t,y \rangle - \psi(y)}{\epsilon}\right),\qquad f_0(y) =  \exp\left(\frac{\langle x_0,y \rangle - \psi(y)}{\epsilon}\right),\qquad  f_1(y) =  \exp\left(\frac{\langle x_1,y \rangle - \psi(y)}{\epsilon}\right)
\end{equation*}
satisfy
\begin{equation*}
    h(y_t) \geq \exp\left(\frac{-t(1-t)R^\ast(x_1-x_0)}{\epsilon}\right)f_0^{1-t}(y_0) f_1^t(y_1),
\end{equation*}
for all $y_0,y_1 \in \mathrm{dom}\psi$. Note that this inequality is also true if $y_0$ or $y_1$ is not in the domain of $\psi$, indeed the inequality will hold trivially because $f_i(y_i) = 0$ and $h \geq 0$. Therefore by Prékopa-Leindler inequality we have
\begin{equation*}
     \int_L h \geq \exp\left(\frac{-t(1-t)R^\ast(x_1-x_0)}{\epsilon}\right)\left(\int_L f_0\right)^{1-t} \left(\int_L f_1\right)^t.
\end{equation*}
Since $m(\mathrm{dom}\psi) > 0$ the integrals are strictly positive. Taking the logarithm and multiplying by $\epsilon$ yields
\begin{equation*}
\mathcal{L}_{\epsilon,m}(\psi)(x_t) \geq -t(1-t)R^*( x_1-x_0)+ (1-t)\mathcal{L}_{\epsilon,m}(\psi)(x_0)+t\mathcal{L}_{\epsilon,m}(\psi)(x_1),
\end{equation*}
where once again the inequality is vacuous if $R^\ast(x_1-x_0) = + \infty$. This holds for any $x_0,x_1$ such that $x_t \in \mathrm{dom}\mathcal{L}_{\epsilon,m}(\psi)$ and thus item $(b)$ is proven.\\
To prove $(a)$, observe that if $R(u)=\rho_\psi(\vert u\vert)$, $u\in \R^n$, then $R^*(v) = \rho_\psi^*(\vert v\vert)$, $v\in \R^n$,  and  $\psi$ is directionnally $R$-convex.
 
So, according to $(b)$, for all $x_0,x_1 \in \R^n$, it holds 
\begin{equation*}
   \mathcal{L}_{\epsilon,\mathcal{H}^n}(\psi)((1-t)x_0+tx_1)+ t(1-t)\rho_\psi^*(\vert x_1-x_0\vert) \geq (1-t) \mathcal{L}_{\epsilon,\mathcal{H}^n}(\psi)(x_0) + t\mathcal{L}_{\epsilon,\mathcal{H}^n}(\psi)(x_1).
\end{equation*}
So, by definition of $ \sigma_{\mathcal{L}_{\epsilon,\mathcal{H}^n}(\psi)}$, we get $\rho_\psi^\ast(r) \geq \sigma_{\mathcal{L}_{\epsilon,\mathcal{H}^n}(\psi)}(r)$ for all $r\geq0$, which completes the proof.
\endproof

\subsection{$\psi$ $\sigma$-smooth implies $\mathcal{L}_{\epsilon,\mathcal{H}^n}(\psi)$ $\sigma^\ast$-convex}
The following result now establishes an entropic analog of property $(i)$ of Propositions \ref{prop:AP2} and \ref{prop:directional-AP2}.
\begin{proposition}\label{prop:sigma-smooth}\
\begin{itemize}
    \item[(a)] Let $\psi : \R^n \to \R$ be a function such that $\sigma_\psi > -\infty$ ; then $\mathcal{L}_{\epsilon,\mathcal{H}^n}(\psi)$ is $ \sigma_\psi^*$-convex. In other words, 
$\rho_{\mathcal{L}_{\epsilon,\mathcal{H}^n}(\psi)} \geq  \sigma_\psi^*$. 
    \item[(b)] More generally let $\psi : \R^n \to \R \cup \{ + \infty \}$ be such that $K:=\mathrm{dom} \psi$ is an affine subspace of $\R^n$. 
    Moreover assume that $\psi$ is directionally $S$-smooth i.e.
    \begin{equation*}
    \psi((1-t)y_0+ty_1)+ t(1-t)S(y_1-y_0) \geq (1-t)\psi(y_0) + t \psi(y_1), \forall t\in [0,1],
    \end{equation*}
    for all $y_0,y_1 \in \R^n$ such that $(1-t)y_0+ty_1 \in \mathrm{dom}\psi$, 
    where $S:\R^n \to \R \cup \{+\infty\}$ is such that $\mathrm{dom}S \subset \overset{\longrightarrow}{K}$. Denote by $k$ the dimension of $K$ and set $ m =\mathcal{H}^k_{| K}$. Then $\mathcal{L}_{\epsilon,m}(\psi)$  is $S^\ast$-convex, that is
    \begin{equation*}
        \mathcal{L}_{\epsilon,m} (\psi)((1-t)x_0+tx_1) + t(1-t)S^\ast(x_1-x_0) \leq (1-t)\mathcal{L}_{\epsilon,m}(\psi)(x_0) + t \mathcal{L}_{\epsilon,m}(\psi)(x_1)
    \end{equation*}
    for all $x_0,x_1 \in \R^n$.
\end{itemize}

\end{proposition}
\proof
As in Proposition \ref{prop:rho-convex} let us start with the proof of Item $(b)$.
Let us denote by $K$ the domain of $\psi$ which is assumed to be an affine subspace of $\R^n$ and let $\overset{\longrightarrow}{K} =  \{x_1-x_0 : x_0,x_1 \in K\}$ be the associated vector subspace of directions.
Let $x_0,x_1 \in \mathrm{dom} \mathcal{L}_{\epsilon,m}(\psi)$ (otherwise there is nothing to prove) and $v \in \overset{\longrightarrow}{K}$. For $y \in \R^n$, applying the $S$-convexity of $\psi$ to $y_0=y$ and $y_1=y+v$, grants
\begin{equation*}
    \langle x_t,y + tv \rangle - \psi(y+tv) \leq t(1-t)(S(v)-\langle x_1-x_0,v\rangle) + (1-t)(\langle x_0,y\rangle - \psi(y)) + t(\langle x_1,y+v \rangle - \psi(y+v)).
\end{equation*}
Note that if $y+tv \notin \mathrm{dom}\psi$ the inequality holds vacuously.
Therefore, the functions $f_0,f_1,h$ defined by
\begin{equation*}
h(y) = \exp\left(\frac{\langle x_t,y+tv \rangle - \psi(y+tv)}{\epsilon}\right),\, f_0(y) =  \exp\left(\frac{\langle x_0,y \rangle - \psi(y)}{\epsilon}\right),\,  f_1(y) =  \exp\left(\frac{\langle x_1,y+v \rangle - \psi(y+v)}{\epsilon}\right)
\end{equation*}
satisfy
\begin{equation*}
    h(y) \leq \exp\left(\frac{t(1-t)(S(v)-\langle x_1-x_0,v \rangle)}{\epsilon}\right)f_0^{1-t}(y) f_1^t(y).
\end{equation*}
Thus by Hölder's inequality we have
\begin{equation*}
    \int h dm \leq \exp\left(\frac{t(1-t)(S(v)-\langle x_1-x_0,v \rangle)}{\epsilon}\right) \left(\int f_0(y) dm\right)^{1-t} \left(\int f_1(y)dm\right)^t
\end{equation*}
Note that by the  invariance of the measure $m$ with respect to the translation by a vector in $\overset{\longrightarrow}{K}$ we have
\begin{equation*}
    \epsilon \log\left(\int h dm\right) =  \mathcal{L}_{\epsilon,m}(\psi)(x_t), \quad  \epsilon \log\left(\int f_1(y)dm\right) = \mathcal{L}_{\epsilon,m}(\psi)(x_1).
\end{equation*}
Thus taking the logarithm and multiplying by $\epsilon$ in the inequality above grants
\begin{equation*}
    \mathcal{L}_{\epsilon,m}(\psi)(x_t) \leq (1-t)\mathcal{L}_{\epsilon,m}(\psi)(x_0) + t \mathcal{L}_{\epsilon,m}(\psi) (x_1) + t(1-t)( S(v)- \langle x_1-x_0,v \rangle)
\end{equation*}
for any $v \in \overset{\longrightarrow}{K}$. Finally, optimizing over $v\in \overset{\longrightarrow}{K}$ grants the result (since $\mathrm{dom} S \subset \overset{\longrightarrow}{K}$ by assumption).\\
The proof of Item $(a)$ follows by setting $S(v) = \sigma_\psi(\vert v \vert)$ in a similar fashion to Proposition \ref{prop:rho-convex}.
\endproof
\subsection{A remark on $\rho$-convexity along the HJB flow}

Let $g \in C^1(\mathbb{R}^n)$ and define for $0 \leq t \leq T$ and $x \in \mathbb{R}^n$ the following function
\begin{equation*}
    U_t^{T,g}(x) = - \log\left(\frac{1}{(2\pi(T-t))^{n/2}} \int \exp\left(-\frac{\vert y - x \vert^2}{2(T-t)}- g(y)\right)dy\right).
\end{equation*}
It is well known that under mild assumptions on $g$, the map $[0,T]\times \mathbb{R}^n \ni (t,x) \mapsto U_t^{T,g}(x)$ is a classical solution of the HJB equation 
\begin{equation*}
    \begin{cases}
    \partial_t \phi_t(x) + \frac{1}{2} \Delta \phi_t(x) -\frac{1}{2} \vert \nabla \phi_t(x) \vert^2 = 0,\\
    \phi_T(x) = g(x).
    \end{cases}
\end{equation*}
It is also classical that as soon as $g$ is convex so is $U_t^{T,g}$ for any $t \in [0,T]$. More generally it has been proven in \cite{Conforti2024} that other classes of functions are stable. The following notion of weak convexity is introduced.
\begin{definition}
Let $\psi:\mathbb{R}^n\to \mathbb{R}$ be a $C^1$ function. We define for $r\geq0$ the following modulus of weak semiconvexity
\begin{equation*}
    \kappa_\psi(r) = \inf\{\langle \nabla \psi(x)-\nabla \psi(y),x-y\rangle \mid \vert x-y \vert=r\}.
\end{equation*}
\end{definition}
First observe the link between the notion of $\kappa$ weak semiconvexity and $\rho$-convexity.
\begin{lemma}
Let $\psi:\mathbb{R}^n \to \mathbb{R}$ be a $C^1$ function then we have the following inequalities
    \begin{equation*}
        \kappa_\psi(r)\geq 2\rho_\psi(r)/r^2, \quad \rho_\psi(r) \geq r^2\int_0^1 \kappa_\psi(tr) t dt .
    \end{equation*}
\end{lemma}
\proof
Since $\psi$ is $\rho_\psi$-convex we have as in the proof of proposition \ref{prop:directional-regularity}
\begin{equation*}
    \psi(y)\geq \psi(x)+\langle \nabla \psi(x),y-x\rangle +\rho_\psi(\vert y-x\vert)
\end{equation*}
and by symmetry the same holds with $x$ and $y$ exchanged. Adding the two inequalities gives
\begin{equation*}
    \langle \nabla \psi(x)-\nabla \psi(y),x-y\rangle \geq 2\rho_\psi(\vert x-y\vert).
\end{equation*}
Thus $\kappa_\psi(r)\geq 2\rho_\psi(r)/r^2$ for all $r \geq 0$.
Let $x,y$ such that $\vert x-y \vert = r$ we have
\begin{equation*}
    \psi(y) - \psi(x)-\langle \nabla \psi(x),y-x\rangle = \int_0^1 \langle \nabla \psi(x+t(y-x))-\nabla \psi(x),t(y-x)\rangle \frac{1}{t}dt \geq r^2\int_0^1 \kappa_\psi(tr) t dt 
\end{equation*}
which ensures $\rho_\psi(r) \geq r^2\int_0^1 \kappa_\psi(tr) t dt $ for $r\geq 0$.
\endproof
We are now ready to state the aforementioned result.
\begin{theorem}[\cite{Conforti2024}]\label{thm:conforti}
Let $L \in \mathbb{R}$ and $f:\mathbb{R}_+\to\mathbb{R}$ be a function such that $ff^\prime(r) + f^{\prime\prime}(r) = 0$ and $f(0)=0,f^\prime(0) = L$. If $g$ is such that $\kappa_g(r) \geq -r^{-1}f(r)$ for every $r>0$, then for $0 \leq t \leq T < + \infty$ we have $\kappa_{U_t^{T,g}}(r) \geq-r^{-1}f(r)$ for every $r>0$.
\end{theorem}
Proposition \ref{prop:rho-convex} also allows to link the $\rho$-convexity of $g$ with the one of $U_t^{T,g}$. However it does not exhibit a stable class of $\rho$-convexity. Indeed, first remark that 
\begin{equation*}
    U_t^{T,g}(x) = \frac{\vert x \vert^2}{2(T-t)}-\frac{1}{T-t}\mathcal{L}_{(T-t)}((T-t)g+\frac{\vert\,.\,\vert^2}{2})(x)+\frac{n}{2}\log(2\pi(T-t)).
\end{equation*}
Thus Proposition \ref{prop:rho-convex} ensures that $U_t^{T,g}$ is $\rho_g \square \frac{(\,.\,)^2}{2(T-t)}$-convex. This estimates degrades as $T \to + \infty$ whereas the estimate of \cite{Conforti2024} is stable, provided that $\kappa_g$ satisfies the hypothesis of Theorem \ref{thm:conforti}.

\subsection{A remark on  the Cram\'er transform}
Let $X$ be a random vector with values in $\R^n$. We recall that the log-Laplace transform of $X$ is the function $\Lambda_X : \R^n \to \R\cup\{+\infty\}$ defined by
\[
\Lambda_X(x) = \log\E[e^{\langle x,X\rangle}], \qquad x\in \R^n.
\]
When $X$ has a density of the form $e^{-\psi}$ with respect to Lebesgue, with $\psi:\R^n \to \R \cup \{+\infty\}$, then 
\[
\Lambda_X = \mathcal{L}_{1,\mathcal{H}^n}(\psi)
\]
The Cram\'er transform of $X$, denoted $\Lambda_X^*$, is defined as the Legendre transform of $\Lambda_X$:
\[
\Lambda_X^*(y) = \sup_{x\in \R^n}\{\langle x,y \rangle - \Lambda_X(x)\},\qquad y\in \R^n
\]
and plays an important role in Large Deviation Theory, see e.g \cite{DZ}.
Combining the results of Propositions \ref{prop:rho-convex}, \ref{prop:sigma-smooth} and \ref{prop:AP2}, we obtain the following:
\begin{proposition}Let $X$ be a random vector with density $e^{-\psi}$ with respect to Lebesgue.
\begin{enumerate}
\item The function $\Lambda_X^*$ is $\rho_{\psi}^{**}$-convex. In other words, $\rho_{\Lambda_X^*} \geq \rho_{\psi}^{**}$.
\item If $\psi$ is finite valued, then  $\Lambda_X^*$ is  $\sigma_\psi$-smooth. In other words, $\sigma_{\Lambda_X^*} \leq  \sigma_{\psi}$.
\end{enumerate}
\end{proposition}
A similar correspondence could be stated for directional moduli.

\section{Application to optimal transport}\label{sec:OT}

Let  $\mu,\nu$ be two probability measures on $\R^n$ with finite second moments. In this section, we consider the entropic regularization of the quadratic optimal transport problem between $\mu$ and $\nu$: for all $\epsilon>0$, let 
\begin{equation}\label{eq:Reg-ent}
\mathcal{C}_\epsilon(\mu,\nu) = \inf_{\pi \in \Pi(\mu,\nu)}\int \frac{1}{2}\vert x-y\vert^2\,d\pi + \epsilon H(\pi | \mu\otimes \nu),
\end{equation}
where $\Pi(\mu,\nu)$ is the set of probability measures on $\R^n \times \R^n$ admitting $\mu,\nu$ as marginals. We refer to e.g. \cite{Nutz21} for a detailed introduction to this topic. 
According to \cite{Carlier2017}, 
\[
\mathcal{C}_\epsilon(\mu,\nu) \to \frac{1}{2}W_2^2(\mu,\nu)
\]
as $\epsilon \to 0$. 
For any $\epsilon>0$, the problem \eqref{eq:Reg-ent} admits a unique minimizer $\pi_\epsilon$, which is of the form
\[
\pi_\epsilon(dxdy) = e^{\frac{f_\epsilon(x) +g_\epsilon(y)}{\epsilon}- \frac{\vert x-y\vert^2}{2\epsilon}}\,\mu(dx)\nu(dy),
\]
with $f_\epsilon, g_\epsilon: \R^n \to \R$ solutions of the following system of non linear equations :
\begin{align*}
f_\epsilon(x) &= - \epsilon \log \left(\int e^{\frac{g_\epsilon(y)}{\epsilon}- \frac{\vert x-y\vert^2}{2\epsilon}}\,\nu(dy)\right),\qquad \forall x\in \R^n\\
g_\epsilon(y) &= - \epsilon \log \left(\int e^{\frac{f_\epsilon(x)}{\epsilon}- \frac{\vert x-y\vert^2}{2\epsilon}}\,\mu(dx)\right),\qquad \forall y\in \R^n.
\end{align*}
Let us consider the functions $\phi_\epsilon,\psi_\epsilon$ defined by
 \[
 \phi_\epsilon(x) = \frac{1}{2}\vert x\vert^2 - f_\epsilon(x),\qquad x\in \R^n, \qquad \text{and}\qquad  \psi_\epsilon(y) = \frac{1}{2}\vert y\vert^2 - g_\epsilon(y),\qquad y\in \R^n.
 \] 
 The couple $(\phi_\epsilon,\psi_\epsilon)$ is now solution of the following system
\begin{align}
\label{eq:sinkhorn1}\phi_\epsilon(x) &= \epsilon \log \left(\int e^{\frac{\langle x,y \rangle-\psi_\epsilon(y)}{\epsilon}}\,\nu(dy)\right),\qquad \forall x\in \R^n\\
\label{eq:sinkhorn2}\psi_\epsilon(y) &= \epsilon \log \left(\int e^{\frac{\langle x,y \rangle-\phi_\epsilon(x)}{\epsilon}}\,\mu(dx)\right),\qquad \forall y\in \R^n,
\end{align}
which can be rewritten as
\begin{equation*}
    \mathcal{L}_{\epsilon,\nu}(\psi_\epsilon) = \phi_\epsilon,\quad  \mathcal{L}_{\epsilon,\mu}(\phi_\epsilon) = \psi_\epsilon.
\end{equation*}
In this section, we apply the results of Section \ref{sec:Entropic-Legendre} to estimate the modulus of smoothness of the convex function $\phi_\epsilon$. In doing so we will be able to deduce an upper bound on the modulus of smoothness of $\phi$ the Kantorovich potential of the optimal transport problem ($\epsilon =0$). 

\begin{remark}
We use here the same entropic regularization as in \cite{Chewi2022entropic}. 
In \cite{FGP20}, a different entropic regularization of the quadratic transport problem was considered, related to Schr\"odinger bridges for the Ornstein-Uhlenbeck process (we refer to \cite{Leo14} for a survey on the  Schr\"odinger problem).  
\end{remark}

\subsection{Case of absolutely continuous measures}\label{subsec:abs-cont-measures}
Let us first assume that $\mu$ and $\nu$ are both absolutely continuous with respect to Lebesgue:
\[
\mu(dx)=e^{-V(x)}\,dx\qquad \text{and}\qquad\nu(dy) = e^{-W(y)}\,dy
\] 
with $V:\R^n \to \R$ and $W:\R^n \to \R\cup \{+\infty\}$  (in particular $\mu$ has full support). In this case, the system of equations \eqref{eq:sinkhorn1}, \eqref{eq:sinkhorn2} satisfied by $(\phi_\epsilon,\psi_\epsilon)$ is equivalent to 
\begin{equation*}
    \phi_\epsilon = \mathcal{L}_{\epsilon,\mathcal{H}^n}(\psi_\epsilon + \epsilon W),\quad \psi_\epsilon = \mathcal{L}_{\epsilon,\mathcal{H}^n}(\phi_\epsilon + \epsilon V).
\end{equation*}
The following result, based on Propositions \ref{prop:rho-convex} and \ref{prop:sigma-smooth}, provides an ($\epsilon$-independent) estimate of the smoothness modulus of $\phi_\epsilon$ in terms of the smoothness modulus of $V$ and the convexity modulus of $W$.
\begin{theorem}\label{thm:smoothness-phi-epsilon}
Let $\mu(dx) = e^{-V(x)}dx$ and $\nu(dy) = e^{-W(y)}dy$ be two measures on $\R^n$ such that $V$ is $\sigma$-smooth with $\mathrm{dom}\sigma = \R^n$ and $W$ such that $\rho_W>-\infty$ and $\rho_W^*(v)<+\infty$ for some $v>0$. Then it holds for all $\epsilon>0$
\[
    \sigma_{\phi_\epsilon}(r) \leq \int_0^r (\rho_W^{\ast\ast})^{-1}(\sigma_V(s))\,ds,\qquad \forall r\geq0,
\]
where 
\[
(\rho_W^{\ast\ast})^{-1}(t) = \sup \{s \geq0 \mid \rho_W^{\ast\ast}(s) \leq t\},\qquad t\geq0.
\]
\end{theorem}

Note that $\rho_W^*(v)<+\infty$ for some $v>0$ is equivalent to $\rho_W^{**}$ is not identically $0$. 

In the proof below, we will make use of the infimum convolution operator $\square$ defined as follows: for any functions $f,g: \R^n \to \R \cup \{+\infty\}$, 
\[
f\square g (x) = \inf_{y\in \R^n} \{f(y)+g(x-y)\},\qquad x\in \R^n.
\]
\proof
As stated above, $\psi_\epsilon = \mathcal{L}_{\epsilon,\mathcal{H}^n}(\phi_\epsilon + \epsilon V)$. Since $\mathrm{dom}V = \R^n$ and $\psi_\epsilon \neq + \infty$ on the nonempty support of $\nu$, we deduce that $\psi_\epsilon$ is a proper l.s.c. convex function, thus $\rho_{\psi_\epsilon} \geq 0$. 
Moreover by hypothesis $\rho_W > -\infty$, this ensures that $\rho_{\psi_\epsilon+\epsilon W}\geq \rho_{\psi_\epsilon}+\epsilon \rho_W > - \infty$. Thus, using Proposition \ref{prop:rho-convex}, we deduce that $\sigma_{\phi_\epsilon} \leq \rho_{\psi_\epsilon + \epsilon W}^\ast \leq (\rho_{\psi_\epsilon} + \epsilon \rho_W)^\ast $. 
Similarly to $\psi_\epsilon$, the function $\phi_\epsilon$ is a convex proper l.s.c. function, thus $\sigma_{\phi_\epsilon} \geq 0$ and by hypothesis $\sigma_V > - \infty$. 
This allows us to apply Proposition \ref{prop:sigma-smooth} to get $\rho_{\psi_\epsilon}  \geq (\sigma_{\phi_\epsilon} + \epsilon \sigma_V)^\ast$. 
Combining both inequalities above yields 
\[
\sigma_{\phi_\epsilon} \leq ((\sigma_{\phi_\epsilon}+ \epsilon \sigma_V)^\ast+\epsilon \rho_W)^\ast
\]
because the Legendre transform is non-increasing.
Recall that if $f,g:\R_+ \to \R\cup \{+\infty\}$ are arbitrary functions, then 
$(f+g)^* \leq f^*\square g^*$ and $f^{\ast\ast} \leq f$.
Thus 
\[
\sigma_{\phi_\epsilon} \leq (\sigma_{\phi_\epsilon}+ \epsilon \sigma_V) \square \epsilon \rho_W^*(\,\cdot\,/\epsilon).
\]
Let $v>0$ be such that $\rho_W^*(v)<+\infty$. Applying this inequality at $\epsilon v$ gives
\[
\sigma_{\phi_\epsilon}(\epsilon v) \leq (\sigma_{\phi_\epsilon}+ \epsilon \sigma_V) \square \epsilon \rho_W^*(\,\cdot\,/\epsilon) (\epsilon v) \leq \epsilon \rho_W^*(v) <+\infty.
\]
Therefore $\mathrm{dom} \sigma_{\phi_\epsilon}$ is not reduced to $\{0\}$ and so, according to Item $(c)$ of Proposition \ref{prop:AP1},  $\mathrm{dom} \sigma_{\phi_\epsilon} = \R_+$. Take $u, v\geq0$ ; evaluating now the inequality at $u+\epsilon v$, one gets
\begin{equation*}
    \frac{\sigma_{\phi_\epsilon}(u+\epsilon v)- \sigma_{\phi_\epsilon}(u)}{\epsilon} \leq \sigma_V(u) + \rho_W^\ast(v).
\end{equation*}
Now, by the convexity of $\sigma_{\phi_\epsilon}$ given by Proposition \ref{prop:AP1}, we get
\[
     v  \sigma_{\phi_\epsilon}'(u) - \rho_W^\ast(v) \leq \sigma_V(u),
\]
where $\sigma_{\phi_\epsilon}'$ denotes say the right derivative of $\sigma_{\phi_\epsilon}.$
Taking the supremum in $v$ grants 
\[
\rho_W^{\ast\ast}(\sigma_{\phi_\epsilon}'(u)) \leq \sigma_V(u)
\] 
and so, by definition of $(\rho_W^{\ast\ast})^{-1}$, 
\[
\sigma_{\phi_\epsilon}'(u) \leq (\rho_W^{\ast\ast})^{-1}(\sigma_V(u)).
\] 
Finally integration grants
\[
\sigma_{\phi_\epsilon}(r) \leq \int_0^r (\rho_W^{\ast\ast})^{-1}(\sigma_V(s))\,ds.
\]
\endproof

\begin{corollary}\label{cor:OT-reg}
Under the assumptions of Theorem \ref{thm:smoothness-phi-epsilon}, the Brenier transport map $T$ sending $\mu$ onto $\nu$ satisfies the following: for $\mu$ almost all $x\neq y\in \R^n$,
\begin{equation}\label{eq:bound-T}
    \vert T(x) - T(y) \vert \leq \frac{2}{\vert x-y \vert} \sigma(|x-y|),
\end{equation}
and 
\begin{equation}\label{eq:bound-Tbis}
\sigma^*(|T(x)-T(y)|) \leq \sigma(|x-y|)
\end{equation}
with $\sigma(r)=\int_0^{r} (\rho_W^{\ast\ast})^{-1}(\sigma_V(s))\,ds,$ $r\geq0$. 
\end{corollary}
Note that in the case where the function $\frac{1}{r}\int_0^{r} (\rho_W^{\ast\ast})^{-1}(\sigma_V(s))\,ds \to 0$ as $r \to 0^+$, the function $T$ is uniformly continuous and the bound \eqref{eq:bound-T} holds for all $x\neq y\in \R^n.$

\proof By Theorem \ref{thm:smoothness-phi-epsilon} $\phi_\epsilon$ is $\sigma$-smooth for any $\epsilon > 0$ where $\sigma(r) = \int_0^r (\rho_W^{\ast\ast})^{-1}(\sigma_V(s))\,ds$, $r\geq0.$
According to \cite[Theorem 1.1]{Nutz2021}, as $\epsilon \to 0$, the potential $\phi_\epsilon$ converges in $L^1(\mu)$ towards $\phi$ the (up to constant unique) Brenier potential for the transport of $\mu$ onto $\nu$, up to extraction of a subsequence. Thus, up to extraction of a subsequence, $\phi_\epsilon$ converges $\mu$ almost everywhere to $\phi$. Since $\mu$ has full support this convergence holds almost everywhere. Moreover, for all $\epsilon>0$, it holds 
\[
\phi_{\epsilon}((1-t)x_0+tx_1)+t(1-t) \sigma(\vert x_1-x_0 \vert)\geq (1-t)\phi_{\epsilon}(x_0)+t\phi_{\epsilon}(x_1),\qquad \forall x_0,x_1 \in \R^n, \quad\forall  t \in ]0,1[
\]
So passing to the limit when $\epsilon \to 0$, one sees that $\phi$ satisfies the same inequality almost everywhere. But $\phi$ has full domain and is convex thus it is continuous and the inequality holds everywhere. In other words, $\phi$ is $\sigma$-smooth. Applying Proposition \ref{prop:continuity} gives that, for all $x\neq y \in \R^n$ and $a\in \partial \phi(x)$, $b\in \partial \phi(y)$,
\[
\vert b-a\vert \leq \frac{2}{\vert y-x\vert}\sigma(|y-x|)\qquad \text{and}\qquad \sigma^*(|b-a|) \leq \sigma(|y-x|).
\]
Since $\phi$ is almost everywhere differentiable and $\partial \phi(x) = \{\nabla \phi(x)\} = \{T(x)\}$ at every point of differentiability, this completes the proof.
\endproof

\subsection{Case of measures supported on affine subspaces}

The purpose of this section is twofold: adapt the Theorem \ref{thm:smoothness-phi-epsilon} to directional moduli in order to take into account anisotropies of the measures and use this result to derive regularity estimates of the optimal transport between singular measures (which is a pathological case of anisotropy). Here we assume that there are $K,L$ two affine subspaces of $\R^n$ and two functions $V,W$ defined respectively on $K,L$ such that: 
\begin{equation*}
    \mu(dx) = e^{-V(x)}\,\mathcal{H}_K(dx)\qquad \text{and}\qquad\nu(dy) = e^{-W(y)}\,\mathcal{H}_L(dy),
\end{equation*}
where $\mathcal{H}_K$ and $\mathcal{H}_L$ denote the Lebesgue measures on $K,L$, that is $\mathcal{H}_K = \mathcal{H}_{|K}^k$ and $\mathcal{H}_{L} = \mathcal{H}_{|L}^{\ell}$ with $k,l$ the respective dimensions of $K$ and $L$. 
Similarly to paragraph \ref{subsec:abs-cont-measures}, the Schrödinger potentials $(\phi_\epsilon,\psi_\epsilon)$ satisfy :
\begin{equation*}
    \phi_\epsilon = \mathcal{L}_{\epsilon,\mathcal{H}_L}(\psi_\epsilon + \epsilon W),\quad \psi_\epsilon = \mathcal{L}_{\epsilon,\mathcal{H}_K}(\phi_\epsilon + \epsilon V).
\end{equation*}
In the same fashion we derive an ($\epsilon$-independent) estimate on the directional smoothness modulus of $\phi_\epsilon$.
\begin{theorem}\label{thm:phi-epsilon-dir-smoothness}
Let $K,L$ be two affine subspaces of $\R^n$. Let $\mu(dx) = e^{-V(x)}\mathcal{H}_K(dx)$ and $\nu(dy) = e^{-W(y)}\mathcal{H}_L(dy)$ be two probability measures. Assume that $V$ is $S$-smooth, for some $S:\R^n \to \R_+\cup \{+\infty\}$ such that $\mathrm{dom} S = \overset{\longrightarrow}{K}$ and $W$ is $R$-convex with $R:\R^n \to \R \cup\{+\infty\}$ such that $R^\ast(v) < + \infty$ for some $v \neq 0$. Further assume that $\mathrm{dom}S \subset \mathrm{Vect}(\mathrm{dom}R^\ast)$. Then, for all $\epsilon > 0$, the following holds 
\begin{equation}\label{eq:Majoration-S_ep}
    S_{\phi_\epsilon}(d) \leq \int_0^1 \sup_{R^{\ast\ast}(p) \leq S (td)} \langle p,d\rangle \,dt,\qquad \forall d\in \R^n.
\end{equation}
\end{theorem}

\proof
The proof is essentially the same as the one of Theorem \ref{thm:smoothness-phi-epsilon}.
Note that $\psi_\epsilon$ is a proper l.s.c. convex function. 
In particular $R_{\psi_\epsilon} \geq 0$ and, by hypothesis, $R_{\psi_\epsilon} +\epsilon R > -\infty$ thus, by Proposition \ref{prop:rho-convex}, $S_{\phi_\epsilon} \leq (R_{\psi_\epsilon} +\epsilon R)^\ast$.
Likewise, $\psi_\epsilon$ is a proper l.s.c. convex function and thus $S_{\phi_\epsilon} \geq 0$. Moreover, $\mathrm{dom}(S_{\phi_\epsilon} + \epsilon S) \subset \overset{\longrightarrow}{K}$. 
Since $\phi_\epsilon + \epsilon V$ is $(S_{\phi_\epsilon} + \epsilon S)$-smooth, Proposition \ref{prop:sigma-smooth} ensures $R_{\psi_\epsilon} \geq (S_{\phi_\epsilon} + \epsilon S)^\ast$. 
Combining the two inequalities above gives $S_{\phi_\epsilon} \leq ((S_{\phi_\epsilon}+\epsilon S)^\ast + \epsilon R)^\ast$, from which we deduce 
\begin{equation}\label{eq:S^*R^*}
S_{\phi_\epsilon} \leq \left(S_{\phi_\epsilon}+\epsilon S\right) \square \epsilon R^\ast(\,\cdot\,/\epsilon).
\end{equation}
If $ v \in \mathrm{dom}R^\ast$, then \eqref{eq:S^*R^*} yields
\begin{equation*}
    S_{\phi_\epsilon}(\epsilon v) \leq \epsilon R^\ast(v) < + \infty.
\end{equation*}
Thus $\epsilon v \in \mathrm{dom}S_{\phi_\epsilon}$. Since, according to Proposition \ref{prop:directional-AP1}, $\mathrm{dom}S_{\phi_\epsilon}$ is a vector space, we conclude that $\mathrm{Vect}(\mathrm{dom}R^\ast)\subset \mathrm{dom}S_{\phi_\epsilon}$. 
 
Let $u \in \mathrm{dom} S_{\phi_\epsilon}$ and $v \in \R^n$, then \eqref{eq:S^*R^*} yields
\begin{equation*}
   \frac{S_{\phi_\epsilon}(u+\epsilon v) - S_{\phi_\epsilon}(u) }{\epsilon} \leq S(u) + R^\ast(v)
\end{equation*}
which grants, by convexity of $S_{\phi_\epsilon}$, that for any $a \in \partial S_{\phi_\epsilon}(u)$, it holds
\begin{equation*}
    \langle a,v\rangle - R^\ast(v) \leq S(u).
\end{equation*}
Note that $\partial S_{\phi_\epsilon}(u)$ is not empty since the domain of $S_{\phi_\epsilon}$ is a vector subspace of $\R^n.$
Optimizing over $v$ gives $R^{\ast\ast}(a) \leq S(u)$, for all $a\in \partial S_{\phi_\epsilon}(u)$.
Finally, for any $d \in \mathrm{dom} S_{\phi_\epsilon}$, it holds
\[
S_{\phi_\epsilon}(d) = \int_0^1 s'(t)\,dt
\]
with $s'(t) = \lim_{r\to 0^+} \frac{S_{\phi_\epsilon}((t+r)d)- S_{\phi_\epsilon}(td)}{r}$, the directional derivative of $S_{\phi_\epsilon}$ at $td$ in the direction $d$. Using the well known identity $s'(t) = \sup_{a \in \partial S_{\phi_\epsilon}(td)} \langle a,d\rangle$, one gets
\[
s'(t) \leq \sup_{R^{\ast\ast}(p) \leq S (td)} \langle p,d\rangle
\]
which proves \eqref{eq:Majoration-S_ep} for $d\in \mathrm{dom} S_{\phi_\epsilon}$. 
Since, by hypothesis, $\mathrm{dom}S \subset \mathrm{Vect}(\mathrm{dom}R^\ast)$ and we have seen that $\mathrm{Vect}(\mathrm{dom}R^\ast)\subset \mathrm{dom}S_{\phi_\epsilon}$, we conclude that $\mathrm{dom} S \subset \mathrm{dom} S_{\phi_\epsilon}$. Thus, if $d \notin \mathrm{dom} S_{\phi_\epsilon}$, then for any $t>0$, $td \notin \mathrm{dom} S_{\phi_\epsilon}$ and so $td \notin \mathrm{dom} S$, in which case 
\[
\sup_{R^{\ast\ast}(p) \leq S (td)} \langle p,d\rangle =\sup_{p \in \R^n} \langle p,d\rangle= +\infty
\] for all $t\in (0,1]$. Thus \eqref{eq:Majoration-S_ep} is also true in this case, which completes the proof.
\endproof

Letting $\epsilon \to 0$ enables to estimate the directional modulus of smoothness of Kantorovich potentials between $\mu$ and $\nu$. Recall that a Kantorovich potential for the transport of $\mu$ onto $\nu$ is a lower semi continuous convex function $\phi : \R^n \to \R \cup \{+\infty\}$ such that 
\[
\int \phi\,d\mu + \int \phi^*\,d\nu = \sup_{\pi \in \Pi(\mu,\nu)} \iint \langle x,y \rangle \,\pi(dxdy).
\]
Kantorovich potentials are not always unique. Nevertheless, under the assumptions of Theorem \ref{thm:phi-epsilon-dir-smoothness}, Kantorovich potentials are uniquely determined on $K$ (see e.g \cite[Proposition B.2]{BGN}): if $\phi_1$ and $\phi_2$ are two Kantorovich potentials, then $\phi_1-\phi_2$ is constant on $K$. In all what follows, we will always work with the Kantorovich potential $\phi$ such that $\phi=+\infty$ on $\R^n \setminus K$ (which is thus uniquely determined on $\R^n$ up to constant).

\begin{corollary}\label{cor:OT-reg-dir}
Under the assumptions of Theorem \ref{thm:phi-epsilon-dir-smoothness}, the Kantorovich potential $\phi$  for the optimal transport problem from $\mu$ to $\nu$ satisfies
\begin{equation}\label{eq:S-Kant-pot}
\phi((1-t)x_0+tx_1)+t(1-t) \bar{S}(x_1-x_0)\geq (1-t)\phi(x_0)+t\phi(x_1),
\end{equation}
for all $x_0,x_1 \in \R^n$ and for all  $t \in ]0,1[$, where 
\begin{equation*}
    \bar{S}(d) = \int_0^1 \sup_{R^{\ast\ast}(p) \leq S(td)} \langle p,d\rangle \,dt,\qquad d\in \R^n.
\end{equation*}
\end{corollary}

\proof
If $\phi((1-t)x_0 + tx_1) =+\infty$ or $\bar{S}(x_1-x_0)=+\infty$, then \eqref{eq:S-Kant-pot} obviously holds.
Suppose that $\phi((1-t)x_0 + tx_1) <+\infty$ and $\bar{S}(x_1-x_0)<+\infty$.
Note that $\mathrm{dom} \bar{S} \subset \mathrm{dom} S = \overset{\rightarrow}{K}$. Thus, $(1-t)x_0+tx_1 \in K$ and $x_1-x_0 \in \overset{\rightarrow}{K}$. Since $K$ is an affine subspace, we conclude that $x_0 \in K$ and $x_1\in K$.
So it is enough to prove \eqref{eq:S-Kant-pot} when $x_0,x_1 \in K$.
Following the proof of Corollary \ref{cor:OT-reg}, we conclude with the same limiting argument that, for a fixed $t \in ]0,1[$, $\phi$ satisfies \eqref{eq:S-Kant-pot} for $\mu$-almost every $x_0,x_1$.
However $\mathrm{dom}\phi$ contains the support of $\mu$ which is $K$. Thus $\phi$ is continuous on $K$ with respect to the relative topology on $K$ and this ensures that \eqref{eq:S-Kant-pot} holds for any $x_0,x_1 \in K$, and completes the proof.
\endproof

\section{Examples}

\subsection{Global H\"older estimates}
First let us give a proof of the initial contraction result by Caffarelli \cite{Caffarelli2000} stated in Theorem \ref{thm:Caf-intro}, using the arguments developed in this paper.

\proof[Proof of Theorem \ref{thm:Caf-intro}.]
Using the result of Section \ref{sec:ex-R-conv-S-smooth} we deduce that 
\begin{equation*}
    \sigma_V(r) = \alpha_V \frac{r^2}{2}, \qquad\text{and}\qquad \rho_W(r) = \beta_W \frac{r^2}{2},\qquad r\in\R_+.
\end{equation*}
Thus Corollary \ref{cor:OT-reg} ensures
\begin{equation*}
    \vert \nabla \phi(x) -\nabla \phi(y) \vert \leq \frac{2}{\vert x-y \vert} \int_0^{\vert x-y \vert} (\rho_W^{\ast\ast})^{-1}(\sigma_V(s)) ds
\end{equation*}
for any $x,y \in \R^n$. A direct computation grants $(\rho_W^{\ast\ast})^{-1}(\sigma_V(s)) = \sqrt{\alpha_V/\beta_W} s$ and the result follows by integration.
\endproof
The following generalization of Theorem \ref{thm:Caf-intro}, was obtained by Kolesnikov in \cite{Kol10}, extending the result to less regular densities and obtaining global Hölder estimates:
\begin{theorem}[Kolesnikov \cite{Kol10}]\label{thm:kol} Let $\mu(dx) = e^{-V(x)}dx$ and $\nu(dy) = e^{-W(y)}dy$ be two measures on $\R^n$ where  $V,W \in C^1$ are such that $\sigma_V(r) \leq \alpha_V r^{p}$ and $\rho_W(r) \geq \beta_W r^{q}$ with $1 \leq p \leq 2 \leq q$. Then, the optimal transport map $\nabla \phi$ from $\mu$ to $\nu$ satisfies
\begin{equation*}
    \vert \nabla \phi(x)-\nabla \phi(y)\vert \leq \left(\frac{q}{p}\right)^{\frac{p}{p+q}} \left(\frac{\alpha_V}{\beta_W}\right)^{\frac{1}{q}} \vert x-y \vert^{\frac{p}{q}},\qquad \forall x,y \in \R^n.
\end{equation*}
\end{theorem}

\proof
The reasoning is the same as before. Note that $\rho:r \mapsto \beta_W r^q$ is convex therefore $\rho^{\ast\ast} = \rho$. From this we deduce that
$(\rho_W^{\ast\ast})^{-1}(\sigma_V(s)) \leq \left(\frac{\alpha_V}{\beta_W}\right)^{\frac{1}{q}} s^{\frac{p}{q}}$. Therefore, $\phi$ is $S$-smooth, with $S(t) = \left(\frac{\alpha_V}{\beta_W}\right)^{1/q} \frac{p}{p+q} t^{1+\frac{p}{q}}$, $t\geq0$. A simple calculation shows that $S^*(u) = \left(\frac{\alpha_V}{\beta_W}\right)^{-1/p} \frac{u^{1+\frac{q}{p}}}{1+\frac{q}{p}}$, $u\geq0$. A direct application of \eqref{eq:bound-Tbis} of Corollary \ref{cor:OT-reg} gives the result.
\endproof
\begin{remark}
Note that applying \eqref{eq:bound-T} instead of \eqref{eq:bound-Tbis} in the proof above, yields to the bound
\begin{equation*}
    \vert \nabla \phi(x)-\nabla \phi(y)\vert \leq \frac{2q}{p+q} \left(\frac{\alpha_V}{\beta_W}\right)^{\frac{1}{q}} \vert x-y \vert^{\frac{p}{q}},\qquad \forall x,y \in \R^n,
\end{equation*}
which is slightly worse, since $\left(\frac{q}{p}\right)^{\frac{p}{p+q}} \leq \frac{2q}{p+q}$ as follows from Jensen inequality.
\end{remark}
The regularity results presented in Section \ref{sec:OT} allow for more freedom in the choice of the smoothness modulus of $V$. The two following results on the regularity of the optimal transport from a log-Lipschitz measure to a Gaussian and from a Cauchy distribution to a Gaussian demonstrate this. As underlined in Section \ref{sec:ex-R-conv-S-smooth}, the modulus of smoothness of a given function $f$ is finite when it is Lipschitz. This situation encompasses the transport of the exponential measure to the Gaussian and is the object of the following result which is an application of Theorem \ref{thm:kol} in the particular case $p=1,q=2$ and $\alpha_V = 4L,\beta_W = 1/2$.
 
\begin{corollary}
Let $\mu(dx) = e^{-V(x)}dx$ be a probability measure with finite second moment such that $V$ is $L$-Lipschitz and $\nu(dy) = e^{-W(y)}dy = \gamma(dy)$ be the standard Gaussian measure on $\R^n$. Then the optimal transport map $\nabla \phi$ from $\mu$ to $\nu$ is $\frac{1}{2}$-H\"older with a H\"older norm equal to $\frac{8\sqrt{2L}}{3}$.
\end{corollary}

\proof
As stated in Section \ref{sec:ex-R-conv-S-smooth}, the smoothness modulus of $V$ satisfies $\sigma_V(r) \leq 4Lr$ and the convexity modulus of $W$ satisfies $\rho_W(r) \geq \frac{r^2}{2}$. Thus by Corollary \ref{cor:OT-reg} the optimal transport map $\nabla \phi$ from $\mu$ to $\nu$ satisfies 
\begin{equation*}
    \vert \nabla \phi(x) - \nabla \phi(y) \vert \leq \frac{2}{\vert x-y \vert} \int_0^{\vert x-y \vert} \sqrt{8Ls}ds = \frac{8\sqrt{2L}}{3} \vert x-y \vert^{\frac{1}{2}}
\end{equation*}
\endproof
Finally, when $\mu(dx) = e^{-V(x)}dx$ is a Cauchy distribution, then $V(x) = n\log(1+\vert x\vert^2) + C$ with $C$ the normalizing constant. In particular, note that the modulus of smoothness of $V$ is quadratic close to $0$ but enjoys a linear behaviour at infinity. Thus, the optimal transport from $\mu$ to a Gaussian is Lipschitz at close range and H\"older at long range as demonstrated by the following proposition.
\begin{proposition}
Let $\mu(dx) = e^{-V(x)}dx$ be a probability measure on $\R^n$ with $V(x) = n\log(1+\vert x\vert^2) + C$ with $C$ the normalizing constant and $\nu(dy) = \gamma(dy)$ be the standard Gaussian measure on $\R^n$. Then the optimal transport map $\nabla \phi$ from $\mu$ to $\nu$ satisfies
\begin{equation*}
    \vert \nabla \phi(x) -\nabla \phi(y) \vert \leq  \min\left(2\sqrt{3n}\vert x-y \vert,\frac{8\sqrt{2}}{3}\vert x-y \vert^{\frac{1}{2}}\right).
\end{equation*}
\end{proposition}

\proof
As before the result will follow from a direct application of Corollary \ref{cor:OT-reg}. To derive the desired property of the smoothness modulus of $V$ we will use the link between the continuity modulus of the gradient of $V$ and the smoothness modulus of $V$ showed in Section \ref{sec:ex-R-conv-S-smooth}. We denote by $\omega$ the continuity modulus of $\nabla V$. Let $r\in \R_+^*$ and $x,y \in \R^n$ be such that $\vert x-y\vert = r$. Assume w.l.o.g. that $\vert x \vert \geq \vert y\vert$, then
\begin{align*}
    \vert \nabla V(x) -\nabla V(y) \vert  &= \vert \frac{2n}{1+\vert x\vert^2}x - \frac{2n}{1+\vert y\vert^2}y\vert\\
    &\leq \frac{2n}{1+\vert x\vert^2} \vert x-y \vert + \vert \frac{2n}{1+\vert x\vert^2} - \frac{2n}{1+\vert y\vert^2} \vert \vert y \vert\\
    &\leq 2nr + \frac{2n\langle x-y,y+x\rangle}{(1+\vert x \vert^2)(1+\vert y \vert^2)}\vert y \vert\\
    &\leq 2nr  + \frac{2n\vert x-y\vert \vert y+x\vert}{(1+\vert x \vert^2)(1+\vert y \vert^2)}\vert y \vert\\
    &\leq 2nr + 2nr \frac{\vert y \vert\vert y+x\vert}{(1+\vert x \vert^2)(1+\vert y \vert^2)}\\
    &\leq 3nr
\end{align*}
because of Cauchy-Schwarz inequality for the third inequality and $\frac{\vert y \vert\vert y+x\vert}{(1+\vert x \vert^2)(1+\vert y \vert^2)} \leq \frac{1}{2}$. Moreover observe that $\vert \nabla V(x) \vert \leq n$ thus
\begin{equation*}
     \omega(r)\leq \min(3nr,2n).
\end{equation*}
Using Section \ref{sec:ex-R-conv-S-smooth}, we deduce that $\sigma_V(r) \leq \min(6nr^2,4nr)$. Since $\nu$ is the standard Gaussian measure, Corollary \ref{cor:OT-reg} ensures
\begin{equation*}
    \vert \nabla \phi(x) -\nabla \phi(y) \vert \leq \frac{2}{\vert x-y \vert}\int_0^{\vert x-y \vert} \sqrt{\min(12s^2,8s)}ds \leq \min\left(2\sqrt{3n}\vert x-y \vert,\frac{8\sqrt{2}}{3}\vert x-y \vert^{\frac{1}{2}}\right).
\end{equation*}
\endproof

\subsection{Anisotropic measures}
The next theorem gives an anisotropic extension of Caffarelli's theorem.
\begin{theorem}\label{thm:CP}
Let $\mu(dx) = e^{-V(x)}dx$ and $\nu(dy) = e^{-W(y)}dy$ be two probability measures on $\R^n$ where $V,W \in C^2$ and are such that $\nabla^2V \leq A^{-1}$ and $\nabla^2 W \geq B^{-1}$ with $A,B$ two symmetric positive definite matrices. Then the optimal transport map $\nabla \phi$ from $\mu$ to $\nu$ satisfies
\begin{equation}\label{eq:Hess-Bound}
    \nabla^2 \phi \leq B^{1/2}\left(B^{-1/2}A^{-1}B^{-1/2}\right)^{1/2}B^{1/2}.
\end{equation}
Equality holds when $\mu=\mathcal{N}(0, A)$ and $\nu = \mathcal{N}(0,B)$.
\end{theorem}

Equality case in \eqref{eq:Hess-Bound} is well known and stems from the fact that 
the linear map
\[
T(x) = B^{1/2}\left(B^{-1/2}A^{-1}B^{-1/2}\right)^{1/2}B^{1/2} x = A^{-1/2}\left(A^{1/2}BA^{1/2}\right)^{1/2}A^{-1/2}x
\]
is optimal for the transport between $\mu=\mathcal{N}(0, A)$ and $\nu = \mathcal{N}(0,B)$, see \cite{DL82,Gel90}.

The conclusion of Theorem \ref{thm:CP} was first obtained by Valdimarsson \cite[Theorem 1.2]{Val07} in the special case where $\mu$ is a Gaussian measure and the matrices $A,B$ commute, in which case the bound simply reads 
\[
\nabla^2 \phi \leq A^{-1/2}B^{1/2}.
\] 
In \cite{Chewi2022entropic}, Chewi and Pooladian recovered Valdimarsson's result for a general $\mu$ using their entropic regularization method. 
Here we finally remove the commutation assumption.

It turns out that the $\epsilon$-independent bound provided by Theorem \ref{thm:phi-epsilon-dir-smoothness} is suboptimal in this case.
Namely, a simple calculation shows that under the assumptions of Theorem \ref{thm:CP}, Theorem \ref{thm:phi-epsilon-dir-smoothness} yields
\[
S_{\phi}(d) \leq \frac{1}{2}\vert A^{-1/2}d\vert \vert B^{1/2}d\vert,\qquad \forall d\in \R^n,
\]
which is suboptimal by Cauchy-Schwarz (say in the commuting case). Note however that this bound coincides with the conclusion of Theorem \ref{thm:CP} in the case $A^{-1} = B$.

In the proof below, we adapt the reasoning of Theorem \ref{thm:phi-epsilon-dir-smoothness} to obtain an explicit optimal $\epsilon$-dependent bound for the smoothness modulus of the entropic potentials $\phi_\epsilon$ and we derive the desired inequality by letting $\epsilon \to 0.$

\proof
As shown in the proof of Theorem \ref{thm:phi-epsilon-dir-smoothness}, the smoothness modulus $S_{\phi_\epsilon}$ satisfies the following functional inequality:
\[
S_{\phi_\epsilon} \leq (S_{\phi_\epsilon}+\epsilon S)\square \epsilon R^*(\,\cdot\,/\epsilon),
\]
with, by assumption on $V,W$, 
\begin{equation*}
S(d) = \frac{1}{2}\langle d, A^{-1}d\rangle,\quad R(d) = \frac{1}{2}\langle d, B^{-1}d\rangle,\qquad d\in \R^n.
\end{equation*}

Noticing that $R^*(d) = \frac{1}{2}\langle d, Bd\rangle$, $d\in \R^n$, and that $S_{\phi_\epsilon}(0)=0$, we deduce from the functional inequality above that for all integer $N\geq 1$ and sequence of points $u_0,u_1,\ldots, u_N$ with $u_0=d$ and $u_N=0$ it holds
\[
S_{\phi_\epsilon}(d) \leq \epsilon \sum_{i=1}^N S(u_i) + \frac{1}{\epsilon} \sum_{i=0}^{N-1} R^*(u_i-u_{i+1}).
\]
Denoting by $\ell^2(\R^n)$ the space of sequences $u=(u_i)_{i\geq0}$ with values in $\R^n$ such that $\sum_{i\geq0} \vert u_i\vert^2<+\infty$, we thus get that
\[
S_{\phi_\epsilon}(d) \leq \inf_{u \in \ell^2(\R^n), u_0=d}\left\{ \epsilon \sum_{i=1}^{+\infty} S(u_i) + \frac{1}{\epsilon} \sum_{i=0}^{+\infty} R^*(u_i-u_{i+1})\right\}:= S_\epsilon(d).
\]
Let us calculate explicitly $S_\epsilon$. Since the value of $S_\epsilon$ is given by a minimization problem involving a strictly convex l.s.c proper and coercive function on a closed convex set of the Hilbert space $\ell^2(\R^n)$, there exists a unique $\bar{u}\in \ell^2(\R^n)$ (depending on $d$ and $\epsilon$) such that $\bar{u}_0=d$ and
\[
S_\epsilon(d) = \epsilon \sum_{i=1}^{+\infty} S(\bar{u}_i) + \frac{1}{\epsilon} \sum_{i=0}^{+\infty} R^*(\bar{u}_i-\bar{u}_{i+1}).
\]
By the first order condition, we get 
\[
\epsilon A^{-1}\bar{u}_i+\frac{1}{\epsilon}\left(2B(\bar{u}_i) -B\bar{u}_{i+1} -B\bar{u}_{i-1}\right)=0,\qquad i\geq 1.
\]
Let $\bar{v}_i = B^{1/2}\bar{u}_i$, $i\geq 1$.
Then, the sequence $\bar{v}$ satisfies 
\[
\bar{v}_{i+1} = (\epsilon^2M+2I)\bar{v}_i - \bar{v}_{i-1},
\]
with $M=B^{-1/2}A^{-1}B^{-1/2}$. Since $M$ is symmetric, there exists an othogonal matrix $P$ such that $P^TMP=D$ with $D$ a diagonal matrix. Set $\bar{w}_i = P^T \bar{v}_i$, $i\geq 1$, then
\[
\bar{w}_{i+1} = (\epsilon^2 D +2 I) \bar{w}_i - \bar{w}_{i-1},\qquad i\geq 1,
\]
and $\bar{w}_0=P^TB^{1/2}d.$
Denote by $a_1,\ldots,a_n>0$ the diagonal coefficients of $D$. Then, for all $k\in \{1,\ldots, n\}$, one gets
\begin{equation}\label{eq:rec}
\bar{w}_{i+1,k} = \left(2+a_k\epsilon^2\right)\bar{w}_{i,k}-\bar{w}_{i-1,k},\qquad \forall i\geq 1.
\end{equation}
The characteristic equation of this second order one dimensional scheme is $r^2 - \left(2+a_k\epsilon^2\right)r+1=0$, whose roots are the following:
\[
r_k(\epsilon) = \frac{1}{2}\left(2+a_k\epsilon^2 - \sqrt{\left(2+a_k\epsilon^2\right)^2-4}\right)\quad \text{and}\quad s_k(\epsilon)=\frac{1}{2}\left(2+a_k\epsilon^2 + \sqrt{\left(2+a_k\epsilon^2\right)^2-4}\right).
\]
Therefore, we conclude that $(\bar{w}_{i,k})_{i\geq 0}$ is a linear combination of 
$(r_k^i)_{i\geq0}$ and $(s_k^i)_{i\geq0}$. Since $s_k >1$, the sequence $(s_k^i)_{i\geq0}$ does not belong to $\ell^2(\R)$. Therefore, the only possibility is that $\bar{w}_{i,k} = \bar{w}_{0,k} r_k^i$, $i\geq0$.
Note that
\[
S(\bar{u}_i) = \frac{1}{2}\bar{u}_i^TA^{-1}\bar{u}_i =\frac{1}{2}\bar{v}_i^TM\bar{v}_i= \frac{1}{2} w_i^T Dw_i = \frac{1}{2}\sum_{k=1}^n a_k\bar{w}_{i,k}^2 = \frac{1}{2}\sum_{k=1}^n a_kr_k^{2i}\bar{w}_{0,k}^2
\]
and, similarly,
\[
R^*(\bar{u}_i-\bar{u}_{i+1}) = \frac{1}{2}(\bar{w}_i-\bar{w}_{i+1})^TI(\bar{w}_i-\bar{w}_{i+1})=\frac{1}{2}\sum_{k=1}^n(1-r_k)^2 r_k^{2i}\bar{w}_{0,k}^2.
\]
Thus,
\begin{align*}
S_\epsilon(d) &= \frac{\epsilon}{2} \sum_{i=1}^{+\infty} \sum_{k=1}^n a_kr_k^{2i}\bar{w}_{0,k}^2 + \frac{1}{2\epsilon} \sum_{i=0}^{+\infty}\sum_{k=1}^n(1-r_k)^2 r_k^{2i}\bar{w}_{0,k}^2\\
& = \frac{\epsilon}{2}  \sum_{k=1}^n a_k \frac{r_k^2}{1-r_k^2}\bar{w}_{0,k}^2 + \frac{1}{2\epsilon} \sum_{k=1}^n(1-r_k)^2 \frac{1}{1-r_k^2}\bar{w}_{0,k}^2\\
\end{align*}
A simple calculation shows that 
\[
r_k(\epsilon) = 1 - \sqrt{a_k}\epsilon + o(\epsilon)
\]
as $\epsilon\to 0$. This easily implies that
\[
S_\epsilon(d) \to \frac{1}{2} \sum_{k=1}^n \sqrt{a_k}\bar{w}_{0,k}^2 = \frac{1}{2} \bar{w}_0^T D^{1/2} \bar{w}_0 = \frac{1}{2} \bar{v}_0^T M^{1/2}\bar{v}_0 =\frac{1}{2} d^TB^{1/2} M^{1/2}B^{1/2}d,
\]
as $\epsilon \to 0$.
With the same limiting argument as in Corollary \ref{cor:OT-reg}, we conclude that the Kantorovich potential $\phi$ from $\mu$ to $\nu$ satisfies 
\[
S_\phi(d) \leq \frac{1}{2} \langle d,B^{1/2}M^{1/2}B^{1/2}d \rangle,\qquad \forall d\in \R^n,
\]
which completes the proof.
\endproof

It is also possible to look at degenerate anisotropies, in particular when $\mu,\nu$ are supported on affine subspaces. Before stating the result, let us recall some elementary facts about subgradient of convex functions defined on strict affine subspaces of $\R^n$. Let $\phi : \R^n \to \R\cup\{+\infty\}$ be such that $\mathrm{dom} \phi = K$ where $K$ is some affine subspace of $\R^n$. Then 
\begin{equation}\label{eq:subgradient-dec}
\partial \phi (x) = \partial \phi_K(x)+ \overset{\rightarrow}{K}^{\perp},
\end{equation}
where $\overset{\rightarrow}{K}^{\perp}$ is the orthogonal complement of $\overset{\rightarrow}{K}$ and $\partial \phi_K(x) \subset \overset{\rightarrow}{K}$ is the subgradient at $x$ of $\phi_K : K \to \R : y \mapsto \phi(y)$ (the restriction of $\phi$ to $K$). In other words, $v\in \overset{\rightarrow}{K}$ belongs to $\partial \phi_K(x)$ if 
\[
\phi(y) \geq \phi(x) + \langle v,y-x \rangle,\qquad \forall y \in K.
\]

\begin{theorem}\label{thm:Caf-subspaces}
Let $\mu(dx) = e^{-V(x)}\,\mathcal{H}_K(dx)$ and $\nu(dy) = e^{-W(y)}\,\mathcal{H}_L(dy)$ be two probability measures where $K,L$ are affine subspaces of $\R^n$ and $V,W$ are twice continuously differentiable on $K$ and $L$ respectively and satisfy $\nabla^2V \leq \alpha_V I_n$ and $\nabla^2W \geq \beta_W I_n$, with $\alpha_V,\beta_W>0$, and let $\phi$ be the Kantorovich potential $\phi$ from $\mu$ to $\nu$. Then the restriction $\phi_K$ of $\phi$ to the affine subspace $K$ is continuously differentiable and $\nabla \phi_K : K \to \overset{\rightarrow}{K}$ satisfies: for all $x_0,x_1 \in K$ 
\[
\vert\nabla \phi_K(x_1)-\nabla \phi_K(x_0)\vert \leq \sqrt{\frac{\alpha_V}{\beta_W}} C_{K,L} \vert x_1-x_0\vert,
\]
with $C_{K,L}:= \sup_{e \in\overset{\rightarrow}{K}, \vert e\vert=1 } \vert\mathrm{Proj}_{\overset{\rightarrow}{L}}(e)\vert$ and where $\mathrm{Proj}_{\overset{\rightarrow}{L}}$  denotes the orthogonal projection on $L$.

\end{theorem}
In the above statement, the hessian of $V$ and $W$ are understood as operators defined on the tangent spaces $\overset{\rightarrow}{K}$ and $\overset{\rightarrow}{L}$ respectively.

\begin{remark}Let $\mu,\nu$ and $\phi$ be as in Theorem \ref{thm:Caf-subspaces}. 
\begin{itemize}
\item If $K=\R^n$, then $C_{K,L}=1$ and one obviously recovers the conclusions of Theorem \ref{thm:Caf-intro}.
\item If $\pi^*$ is an optimal coupling between $\mu$ and $\nu$, then it is well known that 
\[
\pi^*(\cup_{x\in K}\{x\}\times \partial\varphi(x))=1.
\]
Equivalently, writing $\pi^*(dxdy) = \mu(dx) \pi^*_x(dy)$ the disintegration of $\pi^*$ with respect to its first marginal, one gets
\[
\pi^*_x(\partial \varphi(x)) = 1
\]
for $\mu$ almost every $x\in K$. According to Theorem \ref{thm:Caf-subspaces} and \eqref{eq:subgradient-dec}, one gets
\[
\pi^*_x(\nabla \phi_K(x)+ \overset{\rightarrow}{K}^{\perp}) = 1
\]
for $\mu$ almost every $x\in K$ and $\nabla \phi_K$ is Lipschitz continuous on $K$. In other words, the mass coming from $x$ is moved ``horizontally'' through a Lipschitz map $\nabla \phi_K$ along $\overset{\rightarrow}{K}$ and then distributed ``vertically'' along $\overset{\rightarrow}{K}^\perp.$ Only the horizontal component of mass transports is controlled.
\item In the case where $K$ and $L$ are two orthogonal vector subspaces, then $C_{K,L}=0$ and so $\nabla \phi_K$ is constant (in fact equal to $0$). The transport of mass is purely vertical in this case. In this degenerate situation, any coupling $(X,Y)$ with $X \sim \mu$ and $Y \sim \nu$ is optimal.
Indeed, by orthogonality,
\[
\vert X-Y\vert^2 = \vert X\vert^2 + \vert Y\vert^2.
\]
\item If $K,L$ are close from being orthogonal, then the constant $C_{K,L}$ can be small, and the Lipschitz constant of $\nabla \phi_K$ can be strictly smaller than $\sqrt{\alpha_V/\beta_W}.$
\end{itemize}
\end{remark}

\proof
The moduli satisfy the following $S_V(d) \leq \frac{\alpha_V}{2} \vert d \vert^2$ for $d\in \overset{\rightarrow}{K}$, $R_W(d) \geq \frac{\beta_W}{2} \vert d\vert^2$ for $d \in \overset{\rightarrow}{L}$ and $+\infty$ otherwise. Thus Corollary \ref{cor:OT-reg-dir} ensures $S_\phi(d) \leq \bar{S}(d)$, with 
\[
\bar{S}(d) = \frac{1}{2}\sqrt{\frac{\alpha_V}{\beta_W}}\vert d \vert \vert \mathrm{Proj}_{\overset{\rightarrow}{L}}(d)\vert
\]
for $d \in \overset{\rightarrow}{K}$ and $+\infty$ otherwise. A simple calculation shows that the Legendre transform of $\bar{S}$ is given by 
\[
\bar{S}^*(s)= \frac{1}{2}\sqrt{\frac{\beta_W}{\alpha_V}} \sup_{e\in \overset{\rightarrow}{K}, \vert e \vert =1} \frac{\langle s ,e\rangle^2}{\vert \mathrm{Proj}_{\overset{\rightarrow}{L}}(e)\vert},\qquad s\in \R^n,
\]
with the convention $0/0=0$. 
Note in particular that $\bar{S}^*(s) = +\infty$ if $s \in \overset{\rightarrow}{L}^\perp \setminus \{0\}$.
Thus, by Proposition \ref{prop:directional-regularity}, we have for $x_0,x_1 \in K$ and $y_0 \in \partial \phi (x_0), y_1\in \partial \phi(x_1)$
\begin{equation}\label{eq:barS}
\bar{S}^*(y_1-y_0) \leq \bar{S}(x_1-x_0).
\end{equation}
In particular, $\bar{S}^*(y_1-y_0)<+\infty$. Thus, if $e \in \overset{\rightarrow}{K} \cap \left(\overset{\rightarrow}{L}^\perp\right)$, then $\langle y_1-y_0, e\rangle = 0$.
With this remark, we then see that \eqref{eq:barS} is equivalent to 
\[
\langle y_1-y_0, e \rangle^2 \leq \frac{\alpha_V}{\beta_W} \vert x_1-x_0\vert \vert\mathrm{Proj}_{\overset{\rightarrow}{L}}(x_1-x_0)\vert \vert\mathrm{Proj}_{\overset{\rightarrow}{L}}(e)\vert,
\]
for all $e \in \overset{\rightarrow}{K}$ with $\vert e\vert=1$.
By definition of $C_{K,L}$ we deduce that
\[
\langle y_1-y_0, e \rangle^2 \leq \frac{\alpha_V}{\beta_W}C_{K,L}^2 \vert x_1-x_0\vert^2,
\]
for all $e \in \overset{\rightarrow}{K}$ with $\vert e\vert=1$. Optimizing over $e$ yields
\[
\vert\mathrm{Proj}_{\overset{\rightarrow}{K}}(y_1)-\mathrm{Proj}_{\overset{\rightarrow}{K}}(y_0) \vert \leq \sqrt{\frac{\alpha_V}{\beta_W}}C_{K,L} \vert x_1-x_0\vert.
\]
In particular, taking $x_0=x_1=x \in K$ one concludes, in view of \eqref{eq:subgradient-dec}, that $\partial \phi_K(x)$ is reduced to a single point. This proves that $\phi_K$ is continuously differentiable on $K$ and $\partial \phi_K(x) = \mathrm{Proj}_{\overset{\rightarrow}{K}}(\partial \phi (x)) = \{\nabla \phi_K(x)\}$, $x\in K$, and completes the proof.
\endproof

\subsection{A non Euclidean version of Caffarelli's theorem}
For $p\geq 1$, denote by $\|\,\cdot\,\|_p$ the $p$-norm on $\R^n$, defined by $\|x\|_p^p=\sum_{i=1}^n|x_i|^p$, for all $x\in \R^n.$ 

The following lemma is very classical in the Banach spaces literature.
\begin{lemma}\ Consider the function $f_p(x) = \frac{1}{p}\|x\|_p^p$, $x\in \R^n$.
\begin{itemize}
\item If $2 \leq p<\infty$, the function $f_p$ is $R_p$-convex with $R_p := c_p \frac{\|\,\cdot\,\|_p^p}{p}$, for some constant $c_p \geq 2^{1-p}$. 
\item If $1< p\leq 2$, the function $f_p$ is $S_p$-smooth, with $S_p = c_{p^*}^{1-p} \frac{\|\,\cdot\,\|_p^p}{p}$ and $p^* = \frac{p}{p-1}.$
\end{itemize}
\end{lemma}
\proof
Let $p\geq 2$. According to the classical Clarkson's inequality, it holds
\[
\|\frac{a+b}{2}\|_p^p+\frac{1}{2^p}\|a-b\|_p^p \leq \frac{1}{2}(\|a\|_p^p+\|b\|_p^p),\qquad \forall a,b \in \R^n.
\]
Applying Lemma \ref{lem:BL}, thus yields that $f_p$ is $\frac{1}{p}2^{1-p}\|\,\cdot\,\|_p^p$-convex. The case $1\leq p\leq 2$ follows from Proposition \ref{prop:directional-AP2}, noting that $f_p = (f_{p^*})^*$ and that $p^*\geq 2$.
\endproof
The following result provides a version of the Caffarelli's contraction theorem for the transport between a log-convex perturbation of the product measure with density proportional to $e^{-\frac{\|\,\cdot\,\|_p^p}{p}}$ for $p\leq 1/2$ and a log-concave perturbation of the product measure with density proportional to $e^{-\frac{\|\,\cdot\,\|_{p^*}^{p^*}}{p^*}}$. 
\begin{theorem}\label{thm:lpnorm}
Let $p\in (1,2]$ and consider two probability measures $\mu,\nu$ on $\R^n$ of the form 
\[
\mu(dx) = \frac{1}{Z}e^{h-\alpha f_p}\,dx\quad \text{and}\quad \nu(dy) = \frac{1}{Z'}e^{-k-\beta f_{p*}}\,dy,
\]
with $h,k$ two convex functions on $\R^n$ and $\alpha,\beta>0$. Then, the optimal transport map $\nabla \phi$ sending $\mu$ on $\nu$ is such that
\[
\|\nabla \phi(x) - \nabla \phi(y)\|_{p^*} \leq \left(\frac{\alpha}{\beta}\right)^{1/p^*} c_{p^*}^{-(p-1)} \frac{1}{(p-1)^{2/p^*}} \|x-y\|_p^{p-1},\qquad \forall x,y \in \R^n.
\]
\end{theorem}
For $p=2$, for which $c_2=1$, one recovers the usual Caffarelli's contraction theorem.
Note that applying Kolesnikov's theorem \ref{thm:kol} under the above assumptions on $\mu$ and $\nu$ (and comparing $p$ and $p^*$ norms to the Euclidean norm) would lead to a similar Hölder estimate but with constants depending on the dimension. Of course, in dimension $1$, one can easily check that Theorem \ref{thm:kol} gives the same conclusion.
\proof
The function $V= \alpha f_p - h$ is $\alpha c_{p^*}^{1-p}\frac{\|\,\cdot\,\|_p^p}{p}$-smooth and the function $W=k+\beta f_{p^*}$ is $\beta c_{p^*} \frac{\|\,\cdot\,\|_{p^*}^{p^*}}{p^*}$-convex. So, according to Corollary \ref{cor:OT-reg-dir}, the Kantorovich potential $\phi$ is $\bar{S}$-smooth, with 
$\bar{S}$ given by 
\begin{align*}
\bar{S}(d) &= \int_0^1 \sup_{\beta c_{p^*} \frac{\|v\|_{p^*}^{p^*}}{p^*} \leq \alpha c_{p^*}^{1-p} \frac{\|td\|_{p}^{p}}{p}} \langle v,d \rangle\,dt\\
& = \int_0^1 \sup_{ \|v\|_{p^*} \leq \left(\frac{\alpha}{\beta}\right)^{1/p^*} c_{p^*}^{-(p-1)} \frac{\|td\|_{p}^{p-1}}{(p-1)^{1/p^*}}} \langle v,d \rangle\,dt\\
& = \int_0^1 \|d\|_p\left(\frac{\alpha}{\beta}\right)^{1/p^*} c_{p^*}^{-(p-1)} \frac{\|td\|_{p}^{p-1}}{(p-1)^{1/p^*}}\,dt\\
& = \kappa\frac{\|d\|_p^p}{p},
\end{align*}
with $\kappa = \left(\frac{\alpha}{\beta}\right)^{1/p^*} c_{p^*}^{-(p-1)} \frac{1}{(p-1)^{1/p^*}}.$
According to Proposition \ref{prop:directional-regularity}, we finally get that the function $\phi$ is differentiable and, for all $x,y \in \R^n$, 
\[
\kappa^{1-p^*} \frac{\|\nabla \phi(x) - \nabla \phi(y)\|_{p^*}^{p^*}}{p^*} \leq \kappa \frac{\|x - y\|_{p}^{p}}{p}
\]
and so 
\[
\|\nabla \phi(x) - \nabla \phi(y)\|_{p^*} \leq \frac{\kappa}{(p-1)^{1/p^*}} \|x - y\|_{p}^{p-1},
\]
which ends the proof.
\endproof
 
\subsection{Growth estimates}
It is important to note that the estimation of the smoothness modulus of the optimal transport map is still valid when the target measure is not log-concave. The estimates still hold for $\rho_W$ which is not non-negative and are meaningful as soon as $\lim_{r\to + \infty} \rho_W(r) = +\infty$. In particular it is possible to obtain growth estimates for the optimal transport map (deduced from the modulus of continuity between far apart points). 
\begin{corollary}\label{cor:growth}
Let $\mu(dx) = e^{-V(x)}\,dx$ and $\nu(dy) = e^{-W(y)}\,dy$ be two probability measures of $\R^n$. Assume that $V$ is $\sigma$-smooth with $\sigma:\R \to \R_+$ non-decreasing and that $W$ is $\rho$-convex with $\rho:\R \to \R\cup \{+\infty\}$. Then the optimal transport $T$ from $\mu$ to $\nu$ satisfies
\begin{equation*}
    \vert T(x) \vert \leq \vert T(0) \vert + 2(\rho^{\ast\ast})^{-1}(\sigma(\vert x \vert)).
\end{equation*}
\end{corollary}
\proof
By Corollary \ref{cor:OT-reg} we have
\begin{equation*}
    \vert T(x) - T(0)\vert \leq \frac{2}{\vert x \vert} \int_0^{\vert x \vert} (\rho^{\ast\ast})^{-1}(\sigma(s))ds.
\end{equation*}
Since $\sigma \geq 0$ is non-decreasing and $(\rho^{\ast\ast})^{-1}$ is non-decreasing over $\R_+$ we have the result.
\endproof
This result is to be compared with the growth estimates obtained via concentration of measure in \cite{fathi2024}.
\begin{theorem}[\cite{fathi2024}]\label{thm:fathi-growth}
Let $\alpha$ be a continuous decreasing function such that $\nu(\vert \cdot \vert \geq r) \leq \alpha(r)$ for $r\geq r_0$. Then for any $\lambda > 0$ and $x \in \R^n$, we have
\begin{equation*}
    \vert T(x) \vert \leq \max(3r_0,3\alpha^{-1}(\mu(B(x+2\lambda u,\lambda))))
\end{equation*}
where $u(x) = \frac{T(x)}{\vert T(x)\vert}$ and $T$ is the optimal transport map from $\mu$ to $\nu$.
\end{theorem}
In order to compare the two results above let us first recall some results on the concentration of measure phenomenon.
\begin{proposition}[\cite{Gozlan2010}]
Let $p \in \mathcal{P}(\R^n)$ such that
\begin{equation*}
    \alpha(W_1(p,q)) \leq H(q \mid p)
\end{equation*}
for any $q \in \mathcal{P}(\R^n)$ where $\alpha:\R_+\to \R_+$ is a non-decreasing function such that $\alpha(0)= 0$. We say that $p$ satisfies a transport-entropy inequality of modulus $\alpha$. Then $p$ satisfies the following concentration property: for any $1$-Lipschitz function $f$ and $r\geq0$ we have
\begin{equation*}
    p\left(f>\int f dp + r_0 + r\right) \leq e^{-\alpha(r)}
\end{equation*}
where $r_0 = \alpha^{-1}(\log(2))$.
\end{proposition}
In order to deduce a concentration result for $\nu$ it is sufficient to prove that under the conditions of Corollary \ref{cor:growth} $\nu$ satisfies a transport-entropy inequality. This is the goal of the following lemma.
\begin{lemma}
Let $\nu(dx) = e^{-W(x)}dx$ be a probability measure over $\R^n$ with finite second moment. Assume that $W$ is $\rho$-convex. Then $\nu$ satisfies
\begin{equation*}
    \rho^{\ast\ast}(W_1(\nu,p))^+ \leq H(p\mid \nu)
\end{equation*}
for any $p \in \mathcal{P}(\R^n)$. 
\end{lemma}

\proof
Let $x,y \in \R^n$. Following the proof of Proposition \ref{prop:continuity} we know that
\begin{equation*}
    V(y) - V(x) - \langle \nabla V(x) , y-x \rangle \geq \rho(\vert x-y\vert) \geq \rho^{\ast\ast}(\vert x-y\vert)
\end{equation*}
Thus by combining Proposition 1.1 of \cite{Cordero2017} which ensures that $\nu$ satisfies a transport-entropy inequality of modulus $\rho$ and the convexity of $\rho^{\ast\ast}$ we have that $\rho^{\ast\ast}(W_1(\nu,\cdot\,))\leq H(\,\cdot\mid \nu)$. Finally the positivity of entropy grants the result. 
\endproof
Combining the two results above ensure that in the context of Corollary \ref{cor:growth} we have the following concentration estimate 
\begin{equation*}
    \nu(\vert \,\cdot\, \vert \geq r ) \leq e^{-\rho^{\ast\ast}(r-r_0)^+}
\end{equation*} 
which holds for any $r\geq r_0 = \int \vert \,\cdot\, \vert d\nu + ((\rho^{\ast\ast})^+)^{-1}(\log(2))$. It now remains to lower bound the ball probability  of $\mu$ in the context of Corollary \ref{cor:growth}. This is the goal of the following Lemma. 

\begin{lemma}
Let $\mu(dx) = e^{-V(x)}dx$ be a probability measure over $\R^n$ with finite second moment. Assume that $V$ is $\sigma$-smooth with $\sigma$ non decreasing and that $V$ admits a minimizer $x^*$ at which it is differentiable. Then for any $x,u \in \R^n$ such that $\vert u \vert =1$ we have that
\begin{equation*}
    \log(\mu(B(x+2\vert x \vert u,\vert x \vert)) \geq - \min V - 2 \sigma(\vert  x^\ast\vert + 3\vert x\vert) + \log(\omega_n \vert x \vert^n)
\end{equation*}
where $\omega_n$ is the volume of the unit ball.
\end{lemma}

\proof
Let $x\in \R^n$ we have by Jensen's inequality
\begin{align*}
    \log(\mu(B(x+2\vert x \vert u,\vert x \vert)) &= \log\left(\int_{B(x+2\vert x \vert u,\vert x \vert)}e^{-V(y)}dy\right)\\
    &\geq -\frac{1}{\omega_n \vert x \vert^n}\int_{B(x+2\vert x \vert u,\vert x \vert)}V(y)dy + \log(\omega_n \vert x \vert^n).
\end{align*}
Since $V$ is $\sigma$-smooth Lemma \ref{lem:sigma-growth} ensures
\begin{align*}
    &\frac{1}{\omega_n \vert x \vert^n}\int_{B(x+2\vert x \vert u,\vert x \vert)}V(y)dy \leq\\
    &\frac{1}{\omega_n \vert x \vert^n}\int_{B(x+2\vert x \vert u,\vert x \vert)}V(x+2\vert x \vert u) + \langle\nabla V(x+2\vert x \vert u ),y-(x+2\vert x \vert u)\rangle + \sigma(\vert y - (x+2\vert x \vert u)\vert) dy\\
    &\leq V(x+2\vert x \vert u) + \sigma(\vert x\vert).
\end{align*}
Finally using the same Lemma \ref{lem:sigma-growth} we have $V(x+2\vert x \vert u) \leq V(x^\ast) + \sigma(\vert x+2\vert x \vert u-x^\ast \vert)$. We conclude using the non-decreasingness of $\sigma$. 
\endproof
Finally, under the assumptions of Corollary \ref{cor:growth}, combining the estimates for $\mu$ and $\nu$ gives the following growth estimate on the optimal transport map according to Theorem \ref{thm:fathi-growth}:
\begin{equation*}
    \vert T(x) \vert \leq 9r_0 + 3 ((\rho^{\ast\ast})^+)^{-1}\left((\min V) - n\log(\vert x \vert) - \log(\omega_n) + 2\sigma(\vert x^\ast \vert + 3\vert x \vert) \right).
\end{equation*}
So, assuming that $V$ achieves its minimum at $0$ and that $\sigma(\vert x \vert)) - n\log(\omega_n\vert x \vert) \geq 0$ as $x \to +\infty$, one gets 
\begin{equation*}
    \vert T(x) \vert \leq 6 ((\rho^{\ast\ast})^+)^{-1}\left( \sigma( 3\vert x \vert) \right),
\end{equation*}
as soon as $\vert x\vert$ is large enough, which is of the same order than the bound given by Corollary \ref{cor:growth}. This shows that Fathi's bound gives back the conclusion of Corollary \ref{cor:growth} at a large scale.

To conclude, let us observe now that if $\mu$ is the standard Gaussian measure and $\nu$ is a \emph{rotationnaly} invariant probability measure, then the bound of Corollary \ref{cor:growth} is always better than the bound given by Theorem \ref{thm:fathi-growth}. Indeed, in this case $\sigma(u)=u^2/2$ and $T(x) = \tau(\vert x\vert) \frac{x}{\vert x\vert}$, for all $x\neq 0$, with $\tau:\R_+ \to \R_+$, and so we get
\begin{align*}
\mu(B(x+2\lambda u, \lambda)) &\leq \mu \left(\{ y \in \R^n : \langle y,u \rangle \geq \langle x, u \rangle + \lambda\}\right)\\
&\leq \exp(- \frac{1}{2}(\langle x, u \rangle + \lambda)^2)\\
&= \exp(- \frac{1}{2}(\vert x\vert + \lambda)^2)\\
& \leq \exp(- \frac{1}{2}\vert x\vert^2),
\end{align*}
where $u = T(x)/\vert T(x)\vert = x/\vert x\vert$.
Denoting $\alpha(r) = e^{-\rho^{**}(r-r_0)}$, $r\geq r_0:=\int \vert \,\cdot\, \vert d\nu + ((\rho^{\ast\ast})^+)^{-1}(\log(2))$, one gets 
\begin{align*}
\inf_{\lambda>0} \max(3r_0,3\alpha^{-1}(\mu(B(x+2\lambda u,\lambda))))& \geq 3 \alpha^{-1}(\exp(- \frac{1}{2}\vert x\vert^2))\\
&= 3r_0+3 ((\rho^{\ast\ast})^+)^{-1}(\sigma(\vert x\vert))\\
&\geq \vert T(0)\vert+ 2 ((\rho^{\ast\ast})^+)^{-1}(\sigma(\vert x\vert)),
\end{align*}
since $T(0)=0.$

\subsection{Log-Lipschitz perturbations and optimal transport}
The content of this section is motivated by the following conjecture which appeared in \cite{FMS24} : If $\nu$ is a log-Lipschitz perturbation of the standard Gaussian measure  $\gamma$ on $\R^n$, then the optimal transport map sending $\gamma$ onto $\nu$ is globally Lipschitz. 

In what follows, for any Lipschitz function $a:\R^n \to \R$, we will denote by 
\[
\gamma_a(dx)= \frac{1}{Z_a}e^{-a(x) - \frac{|x|^2}{2}}\,dx
\]
with $Z_a$ a normalizing constant. We will denote by $T_{a,b}:\R^n\to \R^n$ the Brenier transport map sending $\gamma_a$ on $\gamma_b$, and by $T_a:\R^n\to \R^n$ the Brenier map from $\gamma$ to $\gamma_a$. According to \cite[Theorem 1]{CEF19}, these applications are homeomorphisms.

In relation to the above conjecture, we can prove the following 
\begin{theorem}\label{prop:Log-Lip-perturb}
For all Lipschitz functions $a,b : \R^n \to \R$, the Brenier map $T_{a,b}$ from $\gamma_a$ to $\gamma_b$ satisfies
\begin{equation}\label{eq:Tab}
-8L +  |x-y|\leq |T_{a,b}(x)-T_{a,b}(y)| \leq 8L +  |x-y|,\qquad \forall x,y \in \R^n,
\end{equation}
where $L=\max(L_a,L_b)$ and $L_a,L_b$ are the Lipschitz constants of $a,b$.

In particular, the Brenier map $T_a$ between $\gamma$ and $\gamma_a$ satisfies
\[
-8L_a + |x-y|\leq |T_a(x)-T_a(y)| \leq 8L_a +  |x-y|,\qquad \forall x,y \in \R^n
\]
and, if $a$ is even, the following growth estimate holds
\begin{equation}\label{eq:growth-Ta}
-8L_a + |x|\leq |T_a(x)| \leq8L_a+|x|,\qquad \forall x\in \R^n.
\end{equation}
\end{theorem}
Inequality \eqref{eq:Tab} means that the Brenier map between two log-Lipschitz perturbations of the standard Gaussian is an approximate isometry. 
It follows from the first inequality in \eqref{eq:Tab} that if $T_{a,b}$ is Lipschitz, then its Lipschitz constant must be greater than or equal to $1$. 

In comparison with other results of the recent literature, we note that the dependence in $L$ in the growth estimates above is surprisingly light. 
For instance, Proposition 7 of \cite{FMS24} (relying on \cite{CF21}) establishes that whenever $\gamma_a$ is centered it holds
\[
|T_a(x)| \leq C_1 \exp(C_2 \max(L_a ; L_a^2)) \sqrt{n+|x|^2},\qquad \forall x\in \R^n.
\]
Similarly, under the assumptions of the preceding result, it is shown in \cite{FMS24}, that the (non-optimal) Langevin transport map $S_a$ between $\gamma$ and $\gamma_a$ is globally Lipschitz,  with a Lipschitz constant depending exponentially on $L_a$ (see \cite{BP24} for an improved dependence in $L_a$). As discussed in \cite{FMS24}, the exponential dependence in $L_a$ is necessary; see also the example below. 
In contrast, in Proposition \ref{prop:Log-Lip-perturb}, there is no dependence in $n$ and the dependence in $L$ or $L_a$ is linear (and only additive). In particular, if $a$ is even, we get the following precise asymptotic for the norm of $T_a$:
\[
\frac{|T_a(x)|}{|x|} \to 1, \qquad \text{ as } |x| \to \infty.
\]

Theorem \ref{prop:Log-Lip-perturb} will follow from the following Log-Lipschitz perturbation of Caffarelli's theorem.

\begin{theorem}\label{thm:Caff-log-Lip}
Let $\mu(dx) = e^{-V(x)}\,dx$ and $\nu(dy) = e^{-W(y)}\,dy$ be two probability measures with $V: \R^n \to \R$ and $W : \R^n \to \R\cup\{+\infty\}$ such that $D^2V \leq \alpha_V I_n$ and $D^2W \geq \beta_W I_n$ with $\alpha_V,\beta_W>0$. Let $a,b : \R^n \to \R$ be two Lipschitz functions with respective Lipschitz constants $L_a,L_b$, and consider $\mu_a$ and $\nu_b$ the probability distributions with densities proportional to $e^{-(a+V)}$ and $e^{-(b+W)}$ respectively. Then
the Kantorovich potential $\phi_{a,b}$ for the transport of $\mu_a$ onto $\nu_b$ is $\sigma_{a,b}$-smooth, with
\[
 \sigma_{a,b}(r)= \int_0^r \sqrt{\frac{4L_b^2}{\beta_W^2}+\frac{4L_a}{\beta_W}s+\frac{\alpha_V}{\beta_W}s^2}+\frac{2L_b}{\beta_W}\,ds,\qquad  \forall r\geq0.
\]
The Brenier map $T_{a,b}$ sending $\mu_a$ onto $\nu_b$ is such that
\begin{equation}\label{eq:Tab-gen-1}
 |T_{a,b}(x)-T_{a,b}(y)| \leq \frac{4L_b}{\beta_W}+ 4\sqrt{\frac{L_b^2}{\beta_W^2}+\frac{L_a}{\beta_W}|x-y|} + \sqrt{\frac{\alpha_V}{\beta_W}}|x-y|,\qquad \forall x,y\in \R^n.
\end{equation}
\end{theorem}

\proof[Proof of Theorem \ref{prop:Log-Lip-perturb}.]
Applying Theorem \ref{thm:Caff-log-Lip} to $V=W = \frac{|\,\cdot\,|^2}{2}$, one gets 
 that the Kantorovich potential $\phi_{a,b}$ is $\sigma_{a,b}$-smooth, with $\sigma_{a,b}$ such that
\[
    \sigma_{a,b}(r)\leq \int_0^r \sqrt{4L^2+4Ls+s^2}+2L\,ds = 4Lr + \frac{r^2}{2}\qquad \forall r\geq0.
\]
According to Corollary \ref{cor:OT-reg}, one gets
\[
|T_{a,b}(x)-T_{a,b}(y)| \leq 8L + |x-y|,
\]
for $\gamma_a$ almost every $x, y\in \R^n.$ Since $T_{a,b}$ is continuous this inequality holds for all $x,y \in \R^n$, which proves the second inequality in \eqref{eq:Tab}.
 The map $T_{b,a}$ satisfies the same inequality. Noting that $T_{b,a} = T_{a,b}^{-1}$ completes the proof of \eqref{eq:Tab}. 
If $a$ is even, then $\gamma_a$ is symmetric, and so $T_a(0)=0$. Thus, taking $y=0$, gives the announced growth estimate for $T_a$.
\endproof

\proof[Proof of Theorem \ref{thm:Caff-log-Lip}.]
Let $x_0,x_1 \in \R^n$, $t \in ]0,1[$ and denote $x_t = (1-t)x_0+tx_1$. Since $a$ is $L_a$-Lipschitz, it holds
\[
a(x_t)-a(x_0) \geq -L_a|x_t-x_0| = -L_at|x_1-x_0|
\]
and, similarly, 
\[
a(x_t)-a(x_1) \geq -L_a(1-t)|x_1-x_0|.
\]
Multiplying the first inequality by $(1-t)$, the second by $t$ and summing gives
\begin{equation*}
a(x_t)\geq (1-t)a(x_0)+ta(x_1) - 2L_at(1-t) |x_1-x_0|.
\end{equation*}
Similarly, the function $b$ satisfies
\begin{equation*}
b(x_t)\leq (1-t)b(x_0)+tb(x_1) + 2L_bt(1-t) |x_1-x_0|.
\end{equation*}
Denote $V_a:=a+V+C_a$ and $W_b:=b+W+ D_b$, where $C_a,D_b$ are normalizing constants, so that $\mu_a(dx) = e^{-V_a(x)}\,dx$ and $\nu_b(dy)=e^{-W_b(y)}\,dy$. 
Letting $L= \max(L_a;L_b)$, we see that $V_a,W_b$ satisfy
\[
V_a((1-t)x_0+tx_1)\geq (1-t)V_a(x_0)+tV_a(x_1) - 2L_at(1-t) |x_1-x_0| - t(1-t)\alpha_V\frac{|x_1-x_0|^2}{2}
\]
\[
W_b((1-t)x_0+tx_1)\leq (1-t)W_b(x_0)+tW_b(x_1) + 2L_bt(1-t) |x_1-x_0| - t(1-t)\beta_W\frac{|x_1-x_0|^2}{2}
\]
for all $x_0,x_1 \in \R^n$ and $t\in [0,1]$. Therefore, $V_a$ is $\sigma$-smooth with 
\[
\sigma_a(u) = \alpha_V\frac{u^2}{2}+2L_au,\qquad \forall u\geq0 
\]
and $W_b$ is $\rho$-convex, with 
\[
\rho_b(u) = \beta_W\frac{u^2}{2}-2L_bu,\qquad \forall u\geq0.
\]
For all $t\geq0$,
\begin{align*}
\rho_b^*(t) = \sup_{u\geq0} \left\{(t+2L_b)u -\beta_W\frac{u^2}{2}\right\} = \frac{(t+2L_b)^2}{2\beta_W}
\end{align*}
and so, for all $u\geq \frac{2L_b}{\beta_W}$,
\begin{align*}
\rho_b^{**}(u) = \sup_{t\geq0} \left\{ut -\frac{(t+2L_b)^2}{2\beta_W}\right\}  = \frac{\beta_Wu^2}{2} - 2L_bu
\end{align*}
and for $u \in [0,(2L_b)/\beta_W]$, 
\begin{align*}
\rho_b^{**}(u) = \sup_{t\geq0} \left\{ut -\frac{(t+2L_b)^2}{2}\right\} = -2\frac{L_b^2}{\beta_W}
\end{align*}
A simple calculation shows that
\[
(\rho_b^{**})^{-1}(t) = \frac{2L_b}{\beta_W} + \sqrt{\frac{4L_b^2}{\beta_W^2}+\frac{2t}{\beta_W}},\qquad \forall t\geq0.
\]
Applying Theorem \ref{thm:smoothness-phi-epsilon} yields that the Kantorovich potential $\phi_{a,b}$ is $\sigma_{a,b}$-smooth, with $\sigma_{a,b}$ given by
\begin{align*}
    \sigma_{a,b}(r)&= \int_0^r \sqrt{\frac{4L_b^2}{\beta_W^2}+\frac{4L_a}{\beta_W}s+\frac{\alpha_V}{\beta_W}s^2}+\frac{2L_b}{\beta_W}\,ds\\
    & \leq \frac{2L_b}{\beta_W}r+ \int_0^r\sqrt{\frac{4L_b^2}{\beta_W^2}+\frac{4L_a}{\beta_W}s}\,ds + \sqrt{\frac{\alpha_V}{\beta_W}}\frac{r^2}{2}\\
        & \leq \frac{2L_b}{\beta_W}r+ r\sqrt{\frac{4L_b^2}{\beta_W^2}+\frac{4L_a}{\beta_W}r} + \sqrt{\frac{\alpha_V}{\beta_W}}\frac{r^2}{2},
\end{align*}
for all $r\geq0.$
According to Corollary \ref{cor:OT-reg}, one gets
\[
|T_{a,b}(x)-T_{a,b}(y)| \leq \frac{4L_b}{\beta_W}+ 4\sqrt{\frac{L_b^2}{\beta_W^2}+\frac{L_a}{\beta_W}|x-y|} + \sqrt{\frac{\alpha_V}{\beta_W}}|x-y|,
\]
for $\mu_a$ almost every $x, y\in \R^n.$ Since $T_{a,b}$ is continuous (according to \cite{CEF19}) this inequality holds for all $x,y \in \R^n$, which proves \eqref{eq:Tab-gen-1}.
\endproof

\textbf{Example:}  In this example, we study in details the case where $a(x) = -L|x|$ on $\R$. First of all, by the usual Caffarelli theorem, we see that the Brenier map $T_{a,0}$ between $\gamma_a$ and $\gamma$ is $1$-Lipschitz, and so its inverse $T_a$ satisfies
\[
|T_a(x)-T_a(y)| \geq |x-y|,\qquad \forall x,y \in \R.
\]
On the other hand, let us check by direct calculations that $T_a$ satisfies 
\begin{equation}\label{eq:verifdirect}
|T_a(x)| \leq c+ |x|,\qquad \forall x \in \R.
\end{equation}
for some $c\geq0$. Since $T_a$ is even, it is enough to check the inequality for $x\geq0$. Denoting by $F_a$ and $\Phi$ the distribution functions of $\gamma_a$ and $\gamma$, we want to prove that
\[
F_a^{-1}\circ \Phi (x) \leq c+x,\qquad \forall x\geq0
\]
that is
\[
\Phi(x) \leq F_a(c+x),\qquad \forall x\geq0
\]
or, equivalently,
\[
\frac{1}{\sqrt{2\pi}}\int_x^{+\infty} e^{-u^2/2}\,du \geq \frac{1}{Z_a} \int_{c+x}^{+\infty} e^{L u - \frac{u^2}{2}}\,du.
\]
Using that 
\[
\int_x^{+\infty} e^{-u^2/2}\,du \sim \frac{e^{-x^2/2}}{x} \qquad \text{and} \qquad \int_{c+x}^{+\infty} e^{L x - \frac{u^2}{2}}\,du \sim \frac{e^{L (x+c) - \frac{(x+c)^2}{2}}}{x+c-L}
\]
as $x \to +\infty$, we conclude that $c$ must be larger than $L$.
Note that
\[
\int_{c+x}^{+\infty} e^{Lu - \frac{u^2}{2}}\,du = \int_{x}^{+\infty} e^{Lu+Lc - \frac{c^2+2uc + u^2}{2}}\,du = e^{Lc - \frac{c^2}{2}} \int_{x}^{+\infty} e^{u(L-c) - \frac{u^2}{2}}\,du.
\]
If $c \geq L$ is chosen so that 
\[
e^{Lc-c^2/2} \leq \frac{Z_a}{\sqrt{2\pi}}
\]
the claim is proved.
Note that
\[
Z_a = \int_{-\infty}^{+\infty} e^{L|u|-\frac{u^2}{2}}\,du \geq \sqrt{2\pi}.
\]
Therefore, \eqref{eq:verifdirect} holds with $c=2L$. We conclude that the best constant $c_{opt}$ in \eqref{eq:verifdirect} is such that $c_{opt} \in [L, 2L]$. This shows that the dependence in $L$ in the bound \eqref{eq:growth-Ta} is sharp up to a numerical multiplicative constant.

Let us finally estimate the Lipschitz constant of $T_a$. For all $x \in \R$, we get
\[
T_a'(x) = (F_a^{-1})'\circ \Phi(x) \frac{e^{-x^2/2}}{\sqrt{2\pi}} = Z_ae^{-L |F_a^{-1}(\Phi(x))| + \frac{F_a^{-1}(\Phi(x))^2}{2}}\frac{e^{-x^2/2}}{\sqrt{2\pi}}
\]
and so, taking $x=0$,
\[
T_a'(0) = \frac{Z_a}{\sqrt{2\pi}}.
\]
But 
\[
Z_a = 2 \int_0^{+\infty} e^{Lu - \frac{u^2}{2}}\,du = 2e^{L^2/2} \int_0^{+\infty} e^{ - \frac{(u-L)^2}{2}}\,du = 2e^{L^2/2} \int_{-L}^{+\infty} e^{ - \frac{u^2}{2}}\,du
\]
and so, as $L$ to $\infty$, 
\[
T_a'(0) = \frac{Z_a}{\sqrt{2\pi}} \sim 2 e^{L^2/2}.
\]
In particular, the Lipschitz constant of $T_a$ explodes exponentially fast as $L\to +\infty$.

Let us now explain how Theorem \ref{prop:Log-Lip-perturb} can be used to obtain precise Gaussian concentration inequalities for log-Lipschitz perturbations of the Gaussian measure.

\begin{corollary}\label{cor:Gisop} Let $\nu = e^{-W}\,dx$ be a probability measure with $W:\R^n \to \R\cup \{+\infty\}$ such that $D^2W \geq \beta_W I_n$ with $\beta_W>0$. If $b : \R^n \to \R$ is a Lipschitz function with Lipschitz constant $L$, then the probability $\nu_b$ satisfies the following concentration inequality: for any Borel set $A \subset \R^n$, it holds
\begin{equation}\label{eq:Gisop}
\nu_b(A_r) \geq \Phi\left(\Phi^{-1}(\nu_b(A))+\left[\sqrt{\beta_W}r-\frac{8L}{\sqrt{\beta_W}}\right]_+\right),\qquad \forall r\geq 0,
\end{equation}
where $A_r = \{y \in \R^n : \inf_{x\in A} |x-y| \leq r\}$ denotes the the $r$-enlargement of $A$, and $\Phi(t) = \frac{1}{\sqrt{2\pi}}\int_{-\infty}^t e^{-x^2/2}\,dx$, $t\in \R$, is the distribution function of the standard one dimensional Gaussian measure. In particular, if $A\subset \R^n$ is such that $\nu_b(A) \geq 1/2$, it holds 
\begin{equation}\label{eq:concentration}
\nu_b(A_r) \geq 1 -\frac{1}{2}\exp\left(-\frac{1}{2}\left[\sqrt{\beta_W}r-\frac{8L}{\sqrt{\beta_W}}\right]_+^2\right), \qquad \forall r\geq0.
\end{equation}
Moreover, for any $1$-Lipschitz function $f:\R^n\to \R$, 
\[
\mathrm{Var}_{\nu_b}(f) \leq \frac{1}{\beta_W} + \frac{16L}{\beta_W^{3/2}}\sqrt{\frac{2}{\pi}} + \frac{64 L^2}{\beta_W^2}.
\]
\end{corollary}
It is well known that, under the assumption of the preceding result, $\nu$ satisfies the following Gaussian type isoperimetric inequality:
\[
\nu(A_r) \geq \Phi\left(\Phi^{-1}(\nu_b(A))+\sqrt{\beta_W}r\right),\qquad \forall r\geq 0.
\]
This follows for instance immediately from the Caffarelli contraction theorem and from the isoperimetric inequality for the standard Gaussian.
We thus see that a log-Lipschitz perturbation only shifts the concentration profile of $\nu$ by a constant $8L/\sqrt{\beta_W}$. In \cite{FMS24} and \cite{BP24}, Lipschitz transport maps are constructed between $\gamma$ and $\nu_b$, but with a Lipschitz constant depending exponentially on $L$. This is used in particular in \cite{FMS24} to show that log-Lipschitz perturbations of the Gaussian satisfy a dimension free Gaussian type isoperimetric inequality (see \cite[Theorem 7]{FMS24}), but with constants also depending exponentially on $L$. Using these results, one could only achieve concentration bounds of the form
\[
\nu_b(A_r) \geq 1 - \frac{1}{2}\exp\left(-\frac{r^2}{2C}\right), \qquad \forall r\geq0,
\]
with a constant $C$ of the order $e^{L^2}$ which is clearly worse than \eqref{eq:concentration}.

\proof
According to the classical Gaussian isoperimetric inequality (\cite{ST74, Bor75}), for all Borel set $B \subset \R^n$, it holds
\[
\gamma(B_r) \geq \Phi(\Phi^{-1}(\nu_b(A))+r),\qquad \forall r\geq 0.
\]
Take a Borel set $A \subset \R^n$ and apply the preceding inequality to $B = T^{-1}(A)$, with $T$ the Brenier map transporting $\gamma$ on $\nu_b$, this yields
\[
\gamma\left((T^{-1}(A))_r\right) \geq \Phi(\Phi^{-1}(\nu_b(A))+r),\qquad \forall r\geq 0.
\]
According to Theorem \ref{thm:Caff-log-Lip}, $T$ satisfies the following inequality
\[
|T(x)-T(y)| \leq \frac{1}{\sqrt{\beta_W}}|x-y| + \frac{8L}{\beta_W},\qquad \forall x,y \in \R^n.
\]
This easily implies that
\[
(T^{-1}(A))_r \subset T^{-1}(A_{\frac{r}{\sqrt{\beta_W}}+\frac{8L}{\beta_W}})
\]
and so
\[
\nu_b\left(A_{\frac{r}{\sqrt{\beta_W}}+\frac{8L}{\beta_W}}\right) \geq \Phi(\Phi^{-1}(\nu_b(A))+r),\qquad \forall r\geq 0
\]
which ends the proof.
Now, if $f:\R^n \to \R$ is a $1$-Lipschitz function with median $m$, we have 
\begin{align*}
\mathrm{Var}_{\nu_b}(f)  \leq \int (f-m)^2 \,d\nu_b
 = \int_0^{+\infty} 2t \nu_b(|f-m|>t)\,dt
 \leq \int_0^{+\infty} 4t \overline{\Phi}\left(\left[\sqrt{\beta_W}t-\frac{8L}{\sqrt{\beta_W}}\right]_+\right)\,dt,
\end{align*}
where the last inequality comes from the deviation inequality 
\[
\nu_b(|f-m|>t)
 \leq 2 \overline{\Phi}\left(\left[\sqrt{\beta_W}t-\frac{8L}{\sqrt{\beta_W}}\right]_+\right)
\]
that easily follows from \eqref{eq:Gisop} (see \cite{Ledoux}).
A simple calculation finally shows that
\[
\int_0^{+\infty} 4t \overline{\Phi}\left(\left[\sqrt{\beta_W}t-\frac{8L}{\sqrt{\beta_W}}\right]_+\right)\,dt = \frac{1}{\beta_W} + \frac{16L}{\beta_W^{3/2}}\sqrt{\frac{2}{\pi}} + \frac{64L^2}{\beta_W^2}.
\]
\endproof

In the following corollary, we use the preceding concentration result to derive $\epsilon$ dependent hessian bounds for the entropic potential in the spirit of the paper \cite{Conforti2024}.
\begin{corollary}
Let $\mu(dx) = e^{-V(x)}\,dx$ and $\nu(dy) = e^{-W(y)}\,dy$ be two probability measures with $V: \R^n \to \R$ and $W : \R^n \to \R\cup\{+\infty\}$ such that $D^2V \leq \alpha_V I_n$ and $D^2W \geq \beta_W I_n$ with $\alpha_V,\beta_W>0$. Let $b : \R^n \to \R$ be an $L$-Lipschitz function and let $\nu_b$ be the probability distribution with density proportional to $e^{-(b+W)}$. Then
the entropic Kantorovich potential $\phi_{\varepsilon}$ for the transport of $\mu$ onto $\nu_b$ is $\alpha_\epsilon \frac{|\,\cdot\,|^2}{2}$-smooth, with
\[
\alpha_\epsilon\leq \frac{(C-\epsilon^2\alpha_V\beta_W) + \sqrt{(C-\epsilon^2\alpha_V\beta_W)^2 + 4C\epsilon^2 \alpha_V\beta_W}}{2\epsilon \beta_W}
\]
and 
\[
C = 64 \frac{L^2}{\beta_W} +\frac{16L}{\beta_W^{1/2}} \sqrt{2/\pi}.
\]
\end{corollary}
\proof
Let $\alpha_\epsilon \geq0$ and $\beta_\epsilon \geq0$ be the optimal constants such that $\varphi_\epsilon$ is $\alpha_\epsilon \frac{|\,\cdot\,|^2}{2}$-smooth  and $\psi_\epsilon$ is $\beta_\epsilon \frac{|\,\cdot\,|^2}{2}$-convex.
Reasoning as in the proof of Theorem \ref{thm:smoothness-phi-epsilon}, we see that
\[
\beta_\epsilon \geq \frac{1}{\alpha_\epsilon + \epsilon\alpha_V}.
\]
On the other hand, for all $x\in \R^n$,
\[
\phi_\epsilon(x) = \epsilon \log \int e^{\frac{\langle x,y \rangle - \psi_\epsilon(y) - \epsilon W(y) - \epsilon b(y)}{\epsilon}} \,dy
\]
and so, for all unit vector $u$, 
\[
u^T \mathrm{D^2}\phi_\epsilon(x)u = \frac{1}{\epsilon} u^T\mathrm{Cov}(\nu_{b,\epsilon,x})u,
\]
where $\nu_{b,\epsilon,x}(dy)$ is the probability distribution with a density proportional to $y\mapsto e^{\frac{\langle x,y \rangle - \psi_\epsilon(y) - \epsilon W(y) - \epsilon b(y)}{\epsilon}}$. We see that $\nu_{b,\epsilon,x}$ is a Lipschitz perturbation of the uniformly log-concave distribution $\tilde{\nu}(dy)$ whose density is proportional to $y\mapsto e^{\frac{\langle x,y \rangle - \psi_\epsilon(y) - \epsilon W(y)}{\epsilon}}$. Applying Corollary \ref{cor:Gisop} to $\tilde{\nu}$ yields that
\[
\mathrm{Var}_{\nu_{b,\epsilon,x}}(f) \leq \frac{1}{\frac{\beta_\epsilon}{\epsilon}+ \beta_W}+ 64 \frac{L^2}{(\frac{\beta_\epsilon}{\epsilon}+ \beta_ W)^2} + \frac{16L}{(\frac{\beta_\epsilon}{\epsilon}+\beta_W)^{3/2}} \sqrt{2/\pi},
\]
for all $1$-Lipschitz function $f.$ Applying this bound to $f(x)=\langle u,x\rangle$, $x\in \R^n$, and $|u|=1$, we get
\[
u^T \mathrm{D^2}\phi_\epsilon(x)u \leq \frac{D}{\beta_\epsilon + \epsilon \beta_W},
\]
with 
\[
D = 1+64 \frac{L^2}{\beta_W} +\frac{16L}{\beta_W^{1/2}} \sqrt{2/\pi}.
\]
Since this bound is independent of $x$, we conclude that
\[
\alpha_\epsilon \leq \frac{D}{\beta_\epsilon + \epsilon \beta_W},
\]
which together with the lower bound on $\beta_\epsilon$ yields
\[
\alpha_\epsilon \leq \frac{D}{\frac{1}{\alpha_\epsilon + \epsilon\alpha_V} + \epsilon \beta_W}
\]
and so
\begin{align*}
\alpha_\epsilon &\leq \frac{-(1+\epsilon^2\beta_W\alpha_V-D) + \sqrt{(1+\epsilon^2\beta_W\alpha_V-D)^2 + 4C\epsilon^2 \alpha_V\beta_W}}{2\epsilon \beta_W}\\
&= \frac{(C-\epsilon^2\beta_W\alpha_V) + \sqrt{(C-\epsilon^2\beta_W\alpha_V)^2 + 4C\epsilon^2 \alpha_V\beta_W}}{2\epsilon \beta_W},
\end{align*}
with $C=D-1.$
\endproof

\section{The divergence of Brenier transport map}
The aim of this section is to recover and extend the recent result by De Philippis and Shenfeld \cite{DePS24} about the divergence of the Brenier transport map between a log-subharmonic probability measure $\mu$ and a strongly log-concave measure $\nu$. Before stating our version of this result, we need to introduce some notation and definitions.

A lower semicontinuous function $f:\R^n \to \R$  is called $\alpha$-superharmonic, $\alpha \in \R$, whenever $\bar{f}:=f-\alpha\frac{\vert\,\cdot\,\vert^2}{2}$ is superharmonic. When $f$ is twice continuously differentiable, this means that $\Delta f \leq \alpha n$. In the general case, $f:\R^n\to \R$ is $\alpha$-superharmonic if, for all $x\in \R^n$ and $r\geq 0$, it holds
\[
\E[\bar{f}(x+rR)] \leq \bar{f}(x) 
\]
where $R$ is any rotationnaly invariant random vector.
In the sequel, we will always take $R\sim \mathcal{N}(0,I_n)$ for convenience. 
In terms of $f$, the latter condition writes
\begin{equation}\label{eq:super-h}
\E[f(x+rR)] \leq f(x) + \frac{\alpha}{2}\left(\E[|x+rR|^2] - |x|^2\right) = f(x) + \frac{\alpha nr^2}{2},\qquad \forall x\in \R^n, \forall r\geq0.
\end{equation}

\begin{remark}
In the usual definition of the super-harmonicity of $\bar{f}:\R^n\to \R$, \eqref{eq:super-h} is required to hold only for $R$ distributed according to the uniform probability measure on $\mathbb{S}^{n-1}$. In particular, the fact that $\E[\bar{f}(x+rR)]$ is well defined in this case, follows from the lower semicontinuity of $\bar{f}$ which implies that $\bar{f}$ is bounded from below on any compact set. It is not difficult to see, using polar coordinates, that whenever \eqref{eq:super-h} holds for $R$ uniformly distributed over $\mathbb{S}^{n-1}$, then \eqref{eq:super-h} holds for any rotationnaly invariant $R$. Details are left to the reader.
\end{remark}
Finally, let us introduce a class of potentials that will play an important role in what follows.
\begin{definition}Let $\beta>0$ ; a proper lower semicontinuous convex function $g:\R^n \to \R\cup \{+\infty\}$ belongs to the class $\mathrm{SH}^*(\beta)$ if $g^*$ is finite valued over $\R^n$ and $1/\beta$-superharmonic.
\end{definition}

Our aim is to establish the following result.

\begin{theorem}\label{thm:DePS}
Let $\mu(dx) = e^{-V(x)}\, dx$ and $\nu(dy) = e^{-W(y)}\,dy$ be two probability measures, with $V:\R^n \to \R$ being $\alpha_V$-superharmonic and $W:\R^n \to \R\cup \{+\infty\}$ being a lower semicontinuous convex function such that $W \in \mathrm{SH}^*(\beta)$. 
For all $\epsilon>0$, the entropic potential $\varphi_\epsilon$ is then $- \frac{\epsilon\alpha_V}{2} + \sqrt{\frac{\alpha_V}{\beta_W}+\epsilon^2\frac{\alpha_V^2}{4}}$-superharmonic. In particular, the Kantorovich potential $\varphi$ is $\sqrt{\alpha_V/\beta_W}$-superharmonic.
\end{theorem}
This result extends \cite[Theorem 1.4]{DePS24} where the superharmonicity of $\varphi$ was obtained under the  assumption $\nabla^2W \geq \beta I_n$ which is strictly stronger according to what follows.  
\begin{lemma}\label{lem:SH^*hessian0}
Let $g:\R^n \to \R$ be a twice continuously differentiable function. If $\nabla^2g \geq \beta I_n$, with $\beta>0$, then $g \in \mathrm{SH}^*(\beta)$.
\end{lemma}
\proof
If $\nabla^2g \geq \beta I_n$ then the convexity modulus of $g$ is such that $\rho_g(r) \geq \beta r^2/2$, $r\geq0$. So, according to Proposition \ref{prop:AP2}, the smoothness modulus of $g^*$ is such that $\sigma_{g^*}(s) \leq s^2/(2\beta)$, $s \geq0$. This means that $g^* - |\,\cdot\,|^2/(2\beta)$ is concave, and thus superharmonic. 
\endproof
Even though, as stated above, the converse implication does not hold, (smooth) functions in $\mathrm{SH^\ast(\beta)}$ have a Hessian that satisfies a dimension-dependent lower bound.

\begin{lemma}\label{lem:SH^*hessian}
If $g\in \mathrm{SH}^\ast(\beta)$, with $\beta > 0$, then $g - \frac{\beta}{2n}\vert \,.\, \vert^2$ is convex. In particular, if $g$ is twice continuously differentiable, then  $\nabla^2 g \geq \frac{\beta}{n}$. 
\end{lemma}
\proof
By definition of $\mathrm{SH}^\ast(\beta)$ the Legendre transform of $g$ is $\frac{1}{\beta}$-superharmonic. Since $g^\ast$ is convex we deduce that $\frac{n}{2\beta}\vert\, .\, \vert^2 - g^\ast$ is convex. Thus Proposition \ref{prop:AP2} ensures that $g - \frac{\beta}{2n}\vert \,.\, \vert^2$ is convex.
\endproof
In particular functions in $\mathrm{SH}^\ast(\beta)$ have superlinear growth.

Before turning to the proof, let us highlight a simple corollary.

\begin{corollary}\label{cor:V-V*}
Let $\mu(dx) = e^{-V(x)}\, dx$ be a probability measure, with $V:\R^n \to \R$ convex and   $\alpha_V$-superharmonic and let $\nu(dy) = \frac{1}{Z}e^{-V^*(y)}\,dy$, where $Z$ is a normalizing constant. Then, the Kantorovich potential $\varphi$ is $\alpha_V$-superharmonic.
\end{corollary}
In particular, under the assumptions of the preceding result, the Brenier map $\nabla \varphi$ is $n\alpha_V$-Lipschitz, which means that $\nabla^2 \varphi \leq n\alpha_V$. This latter bound also follows from Caffarelli's theorem, since $\nabla^2 V \leq n\alpha_V$ and $\nabla^2 W \geq (n\alpha_V)^{-1}$. Corollary \ref{cor:V-V*} improves the bound $\nabla^2 \varphi \leq n\alpha_V$ into $\Delta \varphi \leq n\alpha_V$ which is strictly better.
\proof
By definition $W = V^* + \log Z$ belongs to $\mathrm{SH}^*(1/\alpha_V)$. So, the conclusion follows immediately from Theorem \ref{thm:DePS}.
\endproof

\begin{remark}\ 
\begin{itemize}
\item Let us compare Theorem \ref{thm:DePS} to Theorem \cite[Theorem 1.4]{DePS24}.
Let $\mu$ be a probability measure as in Theorem \ref{thm:DePS}, and let $\nu$ be defined through 
\[
W(y) =\sum_{i=1}^{n-1}\frac{y_i^2}{2}+  \epsilon\frac{y_n^2}{2} +C, \qquad y\in \R^n,
\] 
where $C$ is a normalizing constant and $\epsilon \in (0,1)$. Then $\nabla^2W \geq \epsilon I_n$ is the best possible lower bound on the Hessian matrix of $W$. Applying \cite[Theorem 1.4]{DePS24} thus yields in this case, that the Kantorovich potential $\varphi$ for the transport between $\mu$ and $\nu$ is $\sqrt{\alpha_V/\epsilon}$-superharmonic.  On the other hand, 
\[
W^*(x) =\sum_{i=1}^{n-1}\frac{x_i^2}{2}+\frac{1}{\epsilon} \frac{x_n^2}{2} - C, \qquad x\in \R^n
\] 
and so $\Delta W^* =(n-1)+\frac{1}{\epsilon}$ and $W^*$ is $\frac{(n-1)}{n}+\frac{1}{n\epsilon}$-superharmonic. So $W \in \mathrm{SH}^*(\frac{n\epsilon}{\epsilon(n-1)+1})$ and applying Theorem \ref{thm:DePS} yields that $\varphi$ is $\sqrt{\alpha_V\left(\frac{(n-1)}{n}+\frac{1}{n\epsilon}\right)}$-superharmonic. Since $0<\epsilon <1$, this bound is always strictly better than the previous one, and even improves when the dimension $n$ is large.
\item Moreover, note that in the case where $\mu$ is precisely the Gaussian measure with potential 
\[
V(x) = W^*(x) + D,\qquad x \in \R^n,
\]
with $D$ another normalizing constant, then $\alpha_V = \frac{(n-1)}{n}+\frac{1}{n\epsilon}$ and so, according to Corollary \ref{cor:V-V*}, the Kantorovich potential $\varphi$ is $\frac{(n-1)}{n} + \frac{1}{n\epsilon}$-superharmonic. Since, in this case, the optimal transport map is given by 
\[T(x) = (x_1,\ldots,x_{n-1}, x_n/\epsilon),\qquad x\in \R^n\]
we have $\mathrm{div}(T) = \Delta \varphi = n\left(\frac{(n-1)}{n} + \frac{1}{n\epsilon}\right)$. This shows that the bound given in Corollary \ref{cor:V-V*} is optimal in this case.
\end{itemize}
\end{remark}

The following result gives an alternative characterization of the class $\mathrm{SH}^*(\beta)$, that will be used in the proof of Theorem \ref{thm:DePS}.

\begin{lemma}\label{lem:CharSH^*}
Let $\beta>0$ ; a proper lower semicontinuous function $g:\R^n \to \R\cup\{+\infty\}$ belongs to $\mathrm{SH}^*(\beta)$ if, and only if, for any couple of square integrable random vectors $(U,R)$ with $R \sim \mathcal{N}(0,I_n)$, it holds 
\begin{equation}\label{eq:SH^*}
\E[g(U)] \geq g(\E[U]) + \frac{\beta}{2n} \E[\langle U,R\rangle]^2.
\end{equation}
\end{lemma}
\begin{remark}
Note that the conclusions of Lemmas \ref{lem:SH^*hessian0} and \ref{lem:SH^*hessian} easily follow from Lemma \ref{lem:CharSH^*}. For instance, to get back the conclusion of Lemma \ref{lem:SH^*hessian}, it is enough to apply the preceding result to 
\[
U= \frac{x_0+x_1}{2} + \frac{|x_1-x_0|}{2} \mathrm{sign}(e\cdot R)e,
\] 
with $x_1\neq x_0$ and $e = \frac{x_1-x_0}{|x_1-x_0|}$. Then \eqref{eq:SH^*} yields
\[
\frac{1}{2}(g(x_0)+g(x_1)) - g(\frac{x_0+x_1}{2}) \geq \frac{\beta}{8n} |x_1-x_0|^2.
\]
\end{remark}

\proof
First, let us assume that $g$ satisfies \eqref{eq:SH^*}. Then it is clear that $g$ is convex. Since $g$ is proper there is $x \in \mathbb{R}^n$ such that $g(x)< + \infty$. Thus $g^\ast(y) \geq \langle x,y \rangle -g(x)$ for every $y\in \mathbb{R}^n$. Let $M > \vert x \vert$ and set $g_M = g + \chi_{B_M}$, where $\chi_{B_M}$ is the convex indicator function of $B_M$. Then $g_M^\ast(y) \geq \langle x,y \rangle - g_M(x)$. Let $a\in\mathbb{R}^n,r\geq 0$ and $R \sim \mathcal{N}(0,I_n)$. Consider a square integrable random vector $U$ such that 
\begin{equation*}
    g_M^\ast(a+rR) = \langle a+rR,U\rangle - g(U)
\end{equation*}
almost surely. The existence of such a $U$ is detailed in Appendix \ref{app:measurability}.
This is possible because $g_M^\ast$ is proper thus $U$ won't take values outside $B_M$ which in turns ensures the square integrability of $U$ and allows to replace $g_M$ by $g$ in the right hand side.
Then applying \eqref{eq:SH^*},
\begin{align*}
    \E[g_M^*(a+rR)] &\leq \langle a,\E[U] \rangle - g(\E[U]) + r \E[\langle R,U \rangle] - \frac{\beta}{2n}\E[\langle R,U \rangle]^2\\
& \leq g_M^*(a) + \frac{r^2n}{2\beta} \leq g^\ast(a) + \frac{r^2n}{2\beta}.
\end{align*}
Now remark that $g_M^\ast$ converges pointwise in a non-decreasing fashion towards $g^\ast$. By the monotone convergence theorem, which is applicable here because all the functions are lower bounded by the same integrable function $y \mapsto \langle x,y\rangle - g(x)$, we deduce that
\begin{equation*}
    \E[g^\ast(a+rR)] \leq g^\ast(a) + \frac{r^2n}{2\beta}.
\end{equation*}
In particular, $g^*$ takes finite values over $\R^n$ and is $1/\beta$-superharmonic.

Conversely, let $g \in \mathrm{SH}^*(\beta)$. Let us prove that $g$ satisfies \eqref{eq:SH^*}. Since $g$ is convex and lower semicontinuous, it satisfies $g(u) = \sup_{v\in \R^n} \langle u,v \rangle - g^\ast(v)$, $u\in \R^n$. Therefore, if $(U,R)$ is a couple of square integrable random vectors with $R \sim \mathcal{N}(0,I_n)$, one gets for all $a \in \R^n$ and $r\geq0$,
\begin{align*}
\E[g(U)] & \geq \E[\langle U, a+rR \rangle - g^*(a+rR)]\\
& \geq \langle a,\E[U] \rangle - g^*(a) + r\E[\langle R,U \rangle]  - \frac{r^2n}{2\beta},
\end{align*}
where the second inequality comes from the fact that $g^*$ is $1/\beta$-superharmonic.
Optimizing over $a$ and $r$ yields \eqref{eq:SH^*}.
\endproof

To prove Theorem \ref{thm:DePS}, we need to adapt the method developed in the preceding sections. The first task is to understand the effect of the entropic Legendre transform on superharmonic functions.
\begin{proposition}\label{lem:Lemma1}
If $f$ is $\alpha$-superharmonic with $\alpha >0$,  then $\mathcal{L}_{\epsilon,\mathcal{H}^n}(f)$ belongs to $\mathrm{SH}^*(1/\alpha)$.
\end{proposition}

\proof
Fix $x\in \R^n$  and let $(U,R)$ be a couple of square integrable random vectors with $\E[U]=x$ and $R \sim \mathcal{N}(0,I_n)$. 
Due to the translation invariance of the Lebesgue measure, we get for all $s \geq0$
\[
\epsilon\E\left[\log \left( \int e^{\frac{\langle U,z\rangle- f(z)}{\epsilon} }\,dz\right) \right] = \epsilon\E\left[\log \left( \int e^{\frac{\langle U,z+sR\rangle - f(z+sR)}{\epsilon} }\,dz\right) \right].
\]
Using the convexity of the Log-Laplace functional at the first line, the fact that $\E[U-x] = \E[R] = 0$ at the second, the $\alpha$ superharmonicity property of $f$ at the third, and the fact that $\E[R] = 0$ at the last, we get
\begin{align*}
\epsilon\E\left[\log \left( \int e^{\frac{\langle U,z+sR\rangle -f(z+sR)}{\epsilon} }\,dz\right) \right] &\geq \epsilon\log \left( \int e^{\E[\frac{\langle U-x,z+sR\rangle + \langle x,z+sR\rangle- f(z+sR)}{\epsilon} ]}\,dz\right)\\
& = s\E[\langle U-x,R\rangle] + \epsilon\log \left( \int e^{\frac{\langle x,z\rangle -\E[ f(z+sR)]}{\epsilon}}\,dz\right)\\
& \geq  s\E[\langle U-x,R\rangle] + \epsilon\log \left( \int e^{\frac{\langle x,z\rangle - f(z)}{\epsilon}}\,dz\right) - \frac{\alpha ns^2}{2}\\
& \geq \mathcal{L}_{\epsilon,\mathcal{H}^n}(f)(\E[U])+ s\E[\langle U,R\rangle]  - \frac{\alpha ns^2}{2}.
\end{align*}
Optimizing over $s$ yields that
\[
\E[\mathcal{L}_{\epsilon,\mathcal{H}^n}(f)(U)] \geq \mathcal{L}_{\epsilon,\mathcal{H}^n}(f)(\E[U]) + \frac{1}{2n\alpha}\E[\langle U,R \rangle]^2
\]
which, according to Lemma \ref{lem:CharSH^*}, completes the proof.
\endproof

Proposition \ref{lem:Lemma1} shows that the entropic Legendre transform maps a superharmonic function onto an element of $\mathrm{SH}^*$. The next result explores the converse direction.
\begin{proposition}\label{prop:SH^*}
If $g\in \mathrm{SH}^*(\beta)$ for some $\beta>0$, then $\mathcal{L}_{\epsilon, \mathcal{H}^n}(g)$ is $1/\beta$-superharmonic.
\end{proposition}

In what follows, we assume that we work on a probability space $(\Omega,\mathcal{A},\mathbb{P})$ that is sufficiently large to define all the required random variables.
The following lemma will be crucial to establish Proposition \ref{prop:SH^*}.

\begin{lemma}\label{lem:lemma2}
Given $g : \R^n \to \R$ measurable and such that $\int (1+|z|^2+|g(z)|)e^{\langle x,z \rangle - g(z)}\,dz<+\infty$, for all $x\in \R^n$. If $X$ is a random vector taking values in a finite set of cardinality $N$, one can construct a random vector $Y$ taking values in $\R^n$ and a random vector $Z$ independent of $X$ and taking values in  $(\R^n)^N$ such that
\[
\E[\mathcal{L}_{\epsilon,\mathcal{H}^n}(g)(X)] \leq \mathcal{L}_{\epsilon,\mathcal{H}^n}(g) (\E[X]) + \E[\langle X-\E[X], Y-\E[Y|Z] \rangle]- \E[g(Y)-g(\E[Y|Z])].
\]
\end{lemma}
\begin{remark}\label{rem:measurability}
By construction, $Y = F(X,Z)$ for some measurable function $F$ on $\R^n \times (\R^n)^N \to \R^n$. 
Note that when $g$ is convex (which will always be the case for us), one can always assume that $Y$ is of this form. Indeed, if it is not the case, then it is easy to check that $\tilde{Y} = \E[Y | X,Z]$ still satisfies the desired inequality.
\end{remark}

\proof It is enough to prove the result for $\epsilon=1$. We will simply denote $\mathcal{L}$ instead of $\mathcal{L}_{1,\mathcal{H}^n}$.
Let $x_1,\ldots,x_N$ the possible distinct outcomes of $X$ and set $\lambda_i = \mathbb{P}(X_i=x_i)$, $1\leq i \leq N$, and $\bar{x} = \E[X].$
Let $h(z) = \langle \bar{x},z\rangle - g(z)$, $z\in \R^n$, and for all $i$, denote by $\nu_i$ the probability distribution with a density proportional to $e^{f_i}$ with $f_i(z) = \langle x_i,z\rangle - g(z)$, $z\in \R^n$. Thanks to the integrability condition on $g$, the following integrals are finite $\int e^{f_i},\int e^h < + \infty$ and $H(\nu_i):= \int \log \frac{d\nu_i}{dz}\,d\nu_i$ is finite.

Thus a direct application of the quantitative Prékopa-Leindler inequality of Proposition \ref{prop:QPL} grants
\begin{equation*}
\E[\mathcal{L}(g)(X)] - \mathcal{L}(g)(\bar{x}) \leq  \int \sum_{i=1}^N\lambda_i  \langle x_i-\bar{x},z_i-\bar{z}\rangle+g(\bar{z})- \sum_{i=1}^N\lambda_ig(z_i)\,\bar{\pi}(dz),
\end{equation*}
where $\bar{\pi}$ is the solution of the multimarginal transport problem with marginals $\nu_1,\ldots,\nu_N$ for the cost $(z_1,\ldots,z_N)\mapsto \sum_{i\neq j} \lambda_i\lambda_j\vert z_i-z_j\vert^2$.
Now, for all $z \in (\R^n)^N$, define $T_z: \{x_1,\ldots,x_N\} \to \{z_1,\ldots,z_N\}$ by $T_z(x_i)=z_i$, for all $1\leq i\leq N$ and set $Y = T_Z(X)$ where $Z$ is independent of $X$ and is distributed according to $\bar{\pi}.$ With these notations, we get
\[
\sum_{i=1}^N\lambda_i  \langle x_i-\bar{x},z_i-\bar{z}\rangle = \sum_{i=1}^N\lambda_i  \langle x_i-\bar{x},T_z(x_i)-\E[T_z(X)]\rangle = \E[\langle X-\bar{x},T_z(X)-\E[T_z(X)]\rangle]
\]
and so
\[
 \int \sum_{i=1}^N\lambda_i  \langle x_i-\bar{x},z_i-\bar{z}\rangle\,\bar{\pi}(dz) = \E[\langle X-\bar{x},Y-\E[Y|Z]\rangle].
\]
Moreover, 
\[
 \int g(\bar{z})- \sum_{i=1}^N\lambda_ig(z_i)\,\bar{\pi}(dz) = \E[g(\E[Y|Z]) -g(Y)],
\]
which completes the proof.
\endproof

\proof[Proof of Proposition \ref{prop:SH^*}]
Since $g \in \mathrm{SH}^\ast$, Lemma \ref{lem:SH^*hessian} ensures that $g$ has superlinear growth (at least quadratic) and thus it is possible to apply Lemma \ref{lem:lemma2}. To explain the argument with more clarity, let us first assume that the conclusion of Lemma \ref{lem:lemma2} remains valid without the assumption that $X$ takes a finite number of values. 
Let $X =a+rR$ with $a \in \R^n$, $r\geq 0$ and $R \sim \mathcal{N}(0,I_n)$. 

Applying Lemma \ref{lem:lemma2} to $g$ which is convex, we get
\[
\E[\mathcal{L}_{\epsilon, \mathcal{H}^n}(g)(a+rR)] \leq \mathcal{L}_{\epsilon, \mathcal{H}^n}(g)(a)+r\E[\langle R, Y-\E[Y|Z] \rangle]- \E[g(Y)-g(\E[Y|Z])],
\]
with $Y,Z$ two random vectors defined on the same probability space and $Z$ independent of $X$, with $Y=F(X,Z)$ for some measurable function $F$ (see Remark \ref{rem:measurability}).

Applying \eqref{eq:SH^*} to $U = F(X,z)$ which is distributed according to the law of $Y$ knowing $Z=z$, one gets
\begin{align*}
\E[g(Y) | Z=z] &\geq g(\E[Y| Z=z]) + \frac{\beta}{2n}\E[\langle F(X,z) ,R\rangle]^2\\
&= g(\E[Y| Z=z]) + \frac{\beta}{2n}\E[\langle Y ,R\rangle | Z=z]^2
\end{align*}
and so, taking the expectation and applying Jensen inequality yields
\begin{align*}
\E[g(Y) -g(\E[Y| Z])]&\geq \frac{\beta}{2n}\E[\E[\langle Y,R\rangle|Z]^2] \\
&\geq \frac{\beta}{2n}\E[\langle Y,R\rangle]^2\\
&=\frac{\beta}{2n}\E[\langle Y - \E[Y|Z],R\rangle]^2.
\end{align*}
Therefore, 
\begin{align*}
\E[\mathcal{L}_{\epsilon, \mathcal{H}^n}(g)(a+rR)] &\leq \mathcal{L}_{\epsilon, \mathcal{H}^n}(g)(a)+r\E[\langle R, Y-\E[Y|Z] \rangle]- \frac{\beta}{2n}\E[\langle Y - \E[Y|Z],R\rangle]^2\\
& \leq  \mathcal{L}_{\epsilon, \mathcal{H}^n}(g)(a)+ \frac{n}{2\beta}r^2,
\end{align*}
since $ur- \frac{\beta}{2d} u^2 \leq \frac{d}{2\beta}r^2$, for all $u\in \R.$ This shows that $\mathcal{L}_{\epsilon, \mathcal{H}^n}(g)$ is $1/\beta$-superharmonic.

Let us now complete the argument by approximating $X$ by finite range random variables. Fix $a\in \R^n$ and $r\geq0$.
Let $R\sim \mathcal{N}(0,I_n)$ and $X = a + r R$. 
Consider a family $(\mathcal{P}_k)_{k\geq 0}$ of finite partitions of $\R^n$ such that $\mathcal{P}_{k+1}$ is thinner than $\mathcal{P}_k$ for all $k\geq0.$
Suppose that  $\cup_{k\geq 0} \mathcal{P}_k$ generates the Borel $\sigma$-field on $\R^n$.
Denote by $\mathcal{F}_k$ the sub sigma field of $\mathcal{A}$ generated by $\{ \{R \in A\} : A \in \mathcal{P}_k\}$ and define
\[
R_k = \E[R \mid \mathcal{F}_k],\qquad k\geq0.
\]
By construction, $R_k$ only takes a finite number of values and is such that $(R_k)_{k\geq0}$ is a martingale for the filtration $(\mathcal{F}_k)_{k\geq 0}$ and $R_k \to R$ as $k \to \infty$ almost surely and in all the $L^p$ spaces $p\geq 1$. Define $X_k = a+rR_k$, $k\geq0$. Applying Lemma \ref{lem:lemma2} to $g$, we get for all $k\geq 0$
\[
\E[\mathcal{L}_{\epsilon, \mathcal{H}^n}(g)(X_k)] \leq \mathcal{L}_{\epsilon, \mathcal{H}^n}(g)(a)+r\E[\langle R_k, Y_k-\E[Y_k|Z_k] \rangle]- \E[g(Y_k)-g(\E[Y_k|Z_k])],
\]
with $Z_k$ a random vector defined on the same probability space independent of $R$ (and so of $R_k$) and $Y_k= F_k(X_k,Z_k)$ for some measurable function $F_k$ taking values in $\R^n.$
Applying \eqref{eq:SH^*} to $U= F_k(X_k,z)$ which is distributed according to the law of $Y_k$ knowing $Z_k=z$, one gets
\[
\E[g(Y_k) -g(\E[Y_k| Z_k])]\geq \frac{\beta}{2n}\E[\langle Y_k - \E[Y_k| Z_k],R\rangle]^2.
\]
Note that
\begin{align*}
\E[\langle Y_k - \E[Y_k| Z_k],R\rangle] &= \E[\langle Y_k ,R\rangle] = \int \E[\langle F_k(z,X_k) ,R\rangle] \,\mathbb{P}_{Z_k}(dz)\\
&\stackrel{(*)}{=} \int \E[\langle F_k(z,X_k) ,R_k\rangle] \,\mathbb{P}_{Z_k}(dz) = \E[\langle Y_k ,R_k\rangle] =\E[\langle Y_k - \E[Y_k| Z_k],R_k\rangle] ,
\end{align*}
where $(*)$ follows from the fact that $(R_k)_{k\geq 0}$ is a martingale.
So we get, 
\begin{align*}
\E[\mathcal{L}_{\epsilon, \mathcal{H}^n}(g)(a+rR_k)] &\leq \mathcal{L}_{\epsilon, \mathcal{H}^n}(g)(a)+r\E[\langle R_k, Y_k-\E[Y_k|Z_k] \rangle]- \frac{\beta}{2n}\E[\langle Y_k - \E[Y_k| Z_k],R_k\rangle]^2\\
& \leq \mathcal{L}_{\epsilon, \mathcal{H}^n}(g)(a) + \frac{n}{2\beta}r^2.
\end{align*}
Letting $k \to \infty$, we conclude by Fatou's lemma that $\mathcal{L}_{\epsilon, \mathcal{H}^n}(g)$ is $1/\beta$-superharmonic.
\endproof

We are now ready to prove Theorem \ref{thm:DePS}.

\proof[Proof of Theorem \ref{thm:DePS}]
The entropic potentials $\varphi_\epsilon,\psi_\epsilon$  are related by $\varphi_\epsilon = \mathcal{L}_{\epsilon,\mathcal{H}^n}(\psi_\epsilon + \epsilon W)$ and $\psi_\epsilon = \mathcal{L}_{\epsilon,\mathcal{H}^n}(\varphi_\epsilon + \epsilon V)$.
 First, observe that since $\psi_\epsilon$ is convex, the function $\psi_\epsilon + \epsilon W$ belongs to $\mathrm{SH}^*(\epsilon \beta_W)$. So according to Proposition \ref{prop:SH^*}, $\varphi_{\epsilon}$ is $1/(\epsilon \beta_W)$-superharmonic. Let us denote by $\alpha_{\varphi_\epsilon} \geq 0$ the smallest constant $\alpha \geq0$ such that $\varphi_\epsilon$ is $\alpha$-superharmonic. The function $\varphi_\epsilon + \epsilon V$ is then $\alpha_{\varphi_\epsilon} + \epsilon \alpha_V$-superharmonic. So, according to Proposition \ref{lem:Lemma1}, the function $\psi_\epsilon$ belongs to the class $\mathrm{SH}^*(\frac{1}{\alpha_{\varphi_\epsilon} + \epsilon \alpha_V})$. By assumption on $W$, the function $\psi_\epsilon + \epsilon W$ belongs to the class $\mathrm{SH}^*(\frac{1}{\alpha_{\varphi_\epsilon} + \epsilon \alpha_V} + \epsilon \beta_W)$. According to Proposition \ref{prop:SH^*}, the function $\varphi_\epsilon$ is therefore $\frac{1}{\frac{1}{\alpha_{\varphi_\epsilon} + \epsilon \alpha_V} + \epsilon \beta_W}$-superharmonic. Hence, it holds
\[
\alpha_{\varphi_\epsilon} \leq \frac{1}{\frac{1}{\alpha_{\varphi_\epsilon} + \epsilon \alpha_V} + \epsilon \beta_W}.
\]
Thus,
\[
\alpha_{\varphi_\epsilon} \leq - \frac{\epsilon\alpha_V}{2} + \sqrt{\frac{\alpha_V}{\beta_W}+\epsilon^2\frac{\alpha_V^2}{4}}.
\]
Sending $\epsilon \to 0$ and reasoning as in the preceding sections, we conclude that the (up to constant) unique Kantorovich potential $\varphi$ for the transport from $\mu$ to $\nu$ is $\sqrt{\frac{\alpha_V}{\beta_W}}$-superharmonic.
\endproof

%
%

\begin{thebibliography}{10}

\bibitem{AC11}
Martial Agueh and Guillaume Carlier.
\newblock Barycenters in the {Wasserstein} space.
\newblock {\em SIAM J. Math. Anal.}, 43(2):904--924, 2011.

\bibitem{AP}
Dominique Az{\'e} and Jean-Paul Penot.
\newblock Uniformly convex and uniformly smooth convex functions.
\newblock {\em Ann. Fac. Sci. Toulouse, Math. (6)}, 4(4):705--730, 1995.

\bibitem{BGN}
Espen Bernton, Promit Ghosal, and Marcel Nutz.
\newblock Entropic optimal transport: geometry and large deviations.
\newblock {\em Duke Math. J.}, 171(16):3363--3400, 2022.

\bibitem{BL00}
S.~G. Bobkov and M.~Ledoux.
\newblock From {B}runn-{M}inkowski to {B}rascamp-{L}ieb and to logarithmic
  {S}obolev inequalities.
\newblock {\em Geom. Funct. Anal.}, 10(5):1028--1052, 2000.

\bibitem{Bor75}
Christer Borell.
\newblock The {B}runn-{M}inkowski inequality in {G}auss space.
\newblock {\em Invent. Math.}, 30(2):207--216, 1975.

\bibitem{Bre91}
Yann Brenier.
\newblock Polar factorization and monotone rearrangement of vector-valued
  functions.
\newblock {\em Comm. Pure Appl. Math.}, 44(4):375--417, 1991.

\bibitem{BP24}
G.~Brigati and F.~Pedrotti.
\newblock Heat flow, log-concavity, and lipschitz transport maps, 2024.

\bibitem{Caf92}
Luis~A. Caffarelli.
\newblock The regularity of mappings with a convex potential.
\newblock {\em J. Am. Math. Soc.}, 5(1):99--104, 1992.

\bibitem{Caffarelli2000}
Luis~A. Caffarelli.
\newblock Monotonicity properties of optimal transportation and the fkg and
  related inequalities.
\newblock {\em Communications in Mathematical Physics}, 214(3):547–563,
  November 2000.

\bibitem{Caf02}
Luis~A. Caffarelli.
\newblock Erratum: ``{M}onotonicity of optimal transportation and the {FKG} and
  related inequalities'' [{C}omm. {M}ath. {P}hys. {\bf 214} (2000), no. 3,
  547--563; {MR}1800860 (2002c:60029)].
\newblock {\em Comm. Math. Phys.}, 225(2):449--450, 2002.

\bibitem{Cannarsa2004}
Piermarco Cannarsa and Carlo Sinestrari.
\newblock {\em Semiconcave Functions, Hamilton—Jacobi Equations, and Optimal
  Control}.
\newblock Birkh\"{a}user Boston, 2004.

\bibitem{Carlier2017}
Guillaume Carlier, Vincent Duval, Gabriel Peyré, and Bernhard Schmitzer.
\newblock Convergence of entropic schemes for optimal transport and gradient
  flows.
\newblock {\em SIAM Journal on Mathematical Analysis}, 49(2):1385–1418,
  January 2017.

\bibitem{CFS24}
Guillaume Carlier, Alessio Figalli, and Filippo Santambrogio.
\newblock On {Optimal} {Transport} {Maps} {Between} 1 /d-{Concave} {Densities}.
\newblock Preprint, {arXiv}:2404.05456 [math.{AP}] (2024), 2024.

\bibitem{Chewi2022entropic}
Sinho Chewi and Aram-Alexandre Pooladian.
\newblock An entropic generalization of caffarelli’s contraction theorem via
  covariance inequalities.
\newblock {\em Comptes Rendus. Mathématique}, 361(G9):1471–1482, November
  2023.

\bibitem{CF21}
Maria Colombo and Max Fathi.
\newblock Bounds on optimal transport maps onto log-concave measures.
\newblock {\em J. Differential Equations}, 271:1007--1022, 2021.

\bibitem{CFJ17}
Maria Colombo, Alessio Figalli, and Yash Jhaveri.
\newblock Lipschitz changes of variables between perturbations of log-concave
  measures.
\newblock {\em Ann. Sc. Norm. Super. Pisa Cl. Sci. (5)}, 17(4):1491--1519,
  2017.

\bibitem{Conforti2024}
Giovanni Conforti.
\newblock Weak semiconvexity estimates for schr\"{o}dinger potentials and
  logarithmic sobolev inequality for schr\"{o}dinger bridges.
\newblock {\em Probability Theory and Related Fields}, 189(3–4):1045–1071,
  February 2024.

\bibitem{CEF19}
D.~Cordero-Erausquin and A.~Figalli.
\newblock Regularity of monotone transport maps between unbounded domains.
\newblock {\em Discrete Contin. Dyn. Syst.}, 39(12):7101--7112, 2019.

\bibitem{Cor02}
Dario Cordero-Erausquin.
\newblock Some applications of mass transport to {G}aussian-type inequalities.
\newblock {\em Arch. Ration. Mech. Anal.}, 161(3):257--269, 2002.

\bibitem{Cordero2017}
Dario Cordero-Erausquin.
\newblock Transport inequalities for log-concave measures, quantitative forms,
  and applications.
\newblock {\em Canadian Journal of Mathematics}, 69(3):481–501, June 2017.

\bibitem{CFP18}
Thomas~A. Courtade, Max Fathi, and Ashwin Pananjady.
\newblock Quantitative stability of the entropy power inequality.
\newblock {\em IEEE Trans. Inform. Theory}, 64(8):5691--5703, 2018.

\bibitem{DPF17}
Guido De~Philippis and Alessio Figalli.
\newblock Rigidity and stability of {C}affarelli's log-concave perturbation
  theorem.
\newblock {\em Nonlinear Anal.}, 154:59--70, 2017.

\bibitem{DePS24}
Guido De~Philippis and Yair Shenfeld.
\newblock Optimal transport maps, majorization, and log-subharmonic measures.
\newblock Preprint, {arXiv}:2411.12109 [math.{AP}] (2024), 2024.

\bibitem{DZ}
Amir Dembo and Ofer Zeitouni.
\newblock {\em Large deviations techniques and applications.}, volume~38 of
  {\em Stoch. Model. Appl. Probab.}
\newblock Berlin: Springer, 2nd ed., corrected 2nd printing edition, 2010.

\bibitem{DL82}
D.~C. Dowson and B.~V. Landau.
\newblock The {F}r\'echet distance between multivariate normal distributions.
\newblock {\em J. Multivariate Anal.}, 12(3):450--455, 1982.

\bibitem{fathi2024}
Max Fathi.
\newblock Growth estimates on optimal transport maps via concentration
  inequalities, 2024.

\bibitem{FGP20}
Max Fathi, Nathael Gozlan, and Maxime Prod'homme.
\newblock A proof of the {Caffarelli} contraction theorem via entropic
  regularization.
\newblock {\em Calc. Var. Partial Differ. Equ.}, 59(3):18, 2020.
\newblock Id/No 96.

\bibitem{FMS24}
Max Fathi, Dan Mikulincer, and Yair Shenfeld.
\newblock Transportation onto log-{L}ipschitz perturbations.
\newblock {\em Calc. Var. Partial Differential Equations}, 63(3):Paper No. 61,
  25, 2024.

\bibitem{Figalli2017}
Alessio Figalli.
\newblock {\em The Monge–Ampère Equation and Its Applications}.
\newblock EMS Press, January 2017.

\bibitem{Gel90}
Matthias Gelbrich.
\newblock On a formula for the {$L^2$} {W}asserstein metric between measures on
  {E}uclidean and {H}ilbert spaces.
\newblock {\em Math. Nachr.}, 147:185--203, 1990.

\bibitem{Gozlan2010}
Nathael Gozlan and Christian Léonard.
\newblock Transport inequalities. a survey.
\newblock {\em Markov Processes And Related Fields}, 2010.

\bibitem{Har99}
Gilles Harg\'{e}.
\newblock A particular case of correlation inequality for the {G}aussian
  measure.
\newblock {\em Ann. Probab.}, 27(4):1939--1951, 1999.

\bibitem{Har01}
Gilles Harg\'{e}.
\newblock Inequalities for the {G}aussian measure and an application to
  {W}iener space.
\newblock {\em C. R. Acad. Sci. Paris S\'{e}r. I Math.}, 333(8):791--794, 2001.

\bibitem{KM12}
Young-Heon Kim and Emanuel Milman.
\newblock A generalization of {C}affarelli's contraction theorem via (reverse)
  heat flow.
\newblock {\em Math. Ann.}, 354(3):827--862, 2012.

\bibitem{Kol10}
A.~V. Kolesnikov.
\newblock Global {H}\"older estimates for optimal transportation.
\newblock {\em Mat. Zametki}, 88(5):708--728, 2010.

\bibitem{Kuratowski1965}
K.~Kuratowski and C.~Ryll-Nardzewski.
\newblock A general theorem on selectors.
\newblock {\em Bulletin de l’Académie Polonaise des Sciences; Serie des
  Sciences Mathématiques, Astronomiques et Physiques}, 13:397--403, 1965.

\bibitem{Ledoux}
Michel Ledoux.
\newblock {\em The concentration of measure phenomenon}, volume~89 of {\em
  Math. Surv. Monogr.}
\newblock Providence, RI: American Mathematical Society (AMS), 2001.

\bibitem{Lei72}
L.~Leindler.
\newblock On a certain converse of {H}\"older's inequality. {II}.
\newblock {\em Acta Sci. Math. (Szeged)}, 33(3-4):217--223, 1972.

\bibitem{Leo14}
Christian L\'eonard.
\newblock A survey of the {S}chr\"odinger problem and some of its connections
  with optimal transport.
\newblock {\em Discrete Contin. Dyn. Syst.}, 34(4):1533--1574, 2014.

\bibitem{MS23}
Dan Mikulincer and Yair Shenfeld.
\newblock On the {L}ipschitz properties of transportation along heat flows.
\newblock In {\em Geometric aspects of functional analysis}, volume 2327 of
  {\em Lecture Notes in Math.}, pages 269--290. Springer, Cham, [2023]
  \copyright 2023.

\bibitem{MS24}
Dan Mikulincer and Yair Shenfeld.
\newblock The {B}rownian transport map.
\newblock {\em Probab. Theory Related Fields}, 190(1-2):379--444, 2024.

\bibitem{Mil18}
Emanuel Milman.
\newblock Spectral estimates, contractions and hypercontractivity.
\newblock {\em J. Spectr. Theory}, 8(2):669--714, 2018.

\bibitem{Nutz21}
Marcel Nutz.
\newblock Introduction to entropic optimal transport, lecture notes, columbia
  university, 2021.

\bibitem{Nutz2021}
Marcel Nutz and Johannes Wiesel.
\newblock Entropic optimal transport: convergence of potentials.
\newblock {\em Probability Theory and Related Fields}, 184(1–2):401–424,
  November 2021.

\bibitem{Pre71}
Andr\'as Pr\'ekopa.
\newblock Logarithmic concave measures with application to stochastic
  programming.
\newblock {\em Acta Sci. Math. (Szeged)}, 32:301--316, 1971.

\bibitem{Pre73}
Andr\'as Pr\'ekopa.
\newblock On logarithmic concave measures and functions.
\newblock {\em Acta Sci. Math. (Szeged)}, 34:335--343, 1973.

\bibitem{ST74}
Vladimir~N. Sudakov and Boris~S. Cirel{\cprime}son.
\newblock Extremal properties of half-spaces for spherically invariant
  measures.
\newblock {\em Zap. Nauchn. Semin. Leningr. Otd. Mat. Inst. Steklova},
  41:14--24, 165, 1974.
\newblock Problems in the theory of probability distributions, II.

\bibitem{Val07}
Stef\'an~Ingi Valdimarsson.
\newblock On the {H}essian of the optimal transport potential.
\newblock {\em Ann. Sc. Norm. Super. Pisa Cl. Sci. (5)}, 6(3):441--456, 2007.

\bibitem{VNC78}
A.~A. Vladimirov, Yu.~E. Nesterov, and Yu.~N. Chekanov.
\newblock On uniformly convex functionals.
\newblock {\em Mosc. Univ. Comput. Math. Cybern.}, 1978(3):10--21, 1978.

\bibitem{Zal83}
C.~Zalinescu.
\newblock On uniformly convex functions.
\newblock {\em J. Math. Anal. Appl.}, 95:344--374, 1983.

\bibitem{Zal02}
C.~Z{\u{a}}linescu.
\newblock {\em Convex analysis in general vector spaces}.
\newblock Singapore: World Scientific, 2002.

\end{thebibliography}

\def\cprime{$'$}

\appendix

\section{Proof of Proposition \ref{prop:VNC78}}\label{app:proof}
\proof[Proof of Proposition \ref{prop:VNC78}]
Below we will prove that the function $f_\alpha$ is $\rho$-convex, with $\rho(r) = A(r/2)$, $r\geq0$. This is slightly worse than the bound stated in Proposition \ref{prop:VNC78}. We refer to \cite{VNC78} for a proof of this sharper bound.
According to Lemma \ref{lem:BL}, it suffices to prove that, for all $x,y\in \R^n$, we have
\begin{equation}\label{eq:falpha}
f_\alpha\left(\frac{x+y}{2}\right)+\frac{1}{2}A\left(\frac{\vert x-y\vert}{2}\right) \leq \frac{1}{2}f_\alpha(x)+\frac{1}{2}f_\alpha(y).
\end{equation}




Suppose that \eqref{eq:falpha} holds for all $x,y \in \R^n$ such that $x\cdot y \geq0$. Then, if $x,y$ are such that $x\cdot y\leq 0$, applying \eqref{eq:falpha} to $x,-y$ yields 
\[
f_\alpha\left(\frac{x-y}{2}\right)+\frac{1}{2}f_\alpha\left(\frac{x+y}{2}\right) \leq \frac{1}{2}f_\alpha(x)+\frac{1}{2}f_\alpha(y).
\]
Since, in this case, $\vert x+y\vert \leq \vert x-y\vert$, we get that
\begin{align*}
f_\alpha\left(\frac{x-y}{2}\right) &\geq \frac{1}{2}f_\alpha\left(\frac{x+y}{2}\right) + \frac{1}{2}f_\alpha\left(\frac{x-y}{2}\right)\\
&=\frac{1}{2}f_\alpha\left(\frac{x+y}{2}\right)+\frac{1}{2}A\left(\frac{\vert x-y\vert}{2}\right)
\end{align*}
and so \eqref{eq:falpha} also holds in this case. 

To conclude the proof, let us prove \eqref{eq:falpha} for all $x,y$ such that $x\cdot y \geq0$.
Let $u\in [0,1]$ be such that $x \cdot y = \vert x\vert\vert y\vert u$. Then \eqref{eq:falpha} is equivalent to \[
L(u):=A(N(u))+\frac{1}{2}A(N(-u)) \leq \frac{1}{2}f_\alpha(x)+\frac{1}{2}f_\alpha(y),
\]
with $N(u) = \frac{1}{2}\sqrt{\vert x\vert^2+\vert y\vert^2 + 2 \vert x\vert\vert y\vert u}$.
The derivative of the function $L$ is given by 
\[
L'(u) = \frac{\vert x\vert\vert y\vert}{4}\left[ \frac{\alpha(N(u))}{N(u)} - \frac{1}{2}\frac{\alpha(N(-u))}{N(-u)} \right]
\]
For $u\in [0,1]$, $N(u) \geq N(-u)$, and so since the function $\alpha(t)/t$ is non-decreasing, we get that $L'(u)\geq0$. So, it is enough to prove the inequality for $u=1$, that is when $y = \lambda x$ for some $\lambda \geq0$. Set $a = \vert x\vert$, $b=\vert y\vert$ and assume without loss of generality that $a\leq b$. We are reduced to prove that
\[
\int_0^{\frac{a+b}{2}}\alpha(t)\,dt + \frac{1}{2} \int_0^{\frac{b-a}{2}}\alpha(t)\,dt \leq \frac{1}{2}\int_0^{a}\alpha(t)\,dt+\frac{1}{2}\int_0^{b}\alpha(t)\,dt.
\]
Since the right hand side is equal to 
\[
\int_0^{a}\alpha(t)\,dt+\frac{1}{2}\int_{a}^{\frac{a+b}{2}}\alpha(t)\,dt+\frac{1}{2}\int_{\frac{a+b}{2}}^{b}\alpha(t)\,dt
\] this inequality amounts to 
\[
\int_a^{\frac{a+b}{2}}\alpha(t)\,dt +  \int_0^{\frac{b-a}{2}}\alpha(t)\,dt \leq \int_{\frac{a+b}{2}}^{b}\alpha(t)\,dt,
\]
that is 
\begin{equation}\label{eq:reduction}
\int_0^{\frac{b-a}{2}}\alpha(a+t)\,dt +  \int_0^{\frac{b-a}{2}}\alpha(t)\,dt \leq \int_0^{\frac{b-a}{2}}\alpha\left(\frac{a+b}{2}+t\right)\,dt.
\end{equation}
Since $t\mapsto \frac{\alpha (t)}{t}$ is non-decreasing, the function $\alpha$ is super-additive:
\[
\alpha(u+v) \geq \alpha(u)+\alpha(v),\qquad \forall u,v \geq0.
\]
In particular, if $t\in [0, \frac{b-a}{2}]$,
\[
\alpha\left(\frac{a+b}{2}+t\right) = \alpha\left(a+t + \frac{b-a}{2}\right) \geq \alpha(a+t) + \alpha\left(\frac{b-a}{2}\right) \geq \alpha(a+t) + \alpha\left(t\right).
\]
Integrating this inequality gives \eqref{eq:reduction} and completes the proof.
\endproof

\section{Measurability of the subdifferential}\label{app:measurability}

\begin{lemma}
Let $g$ be a compactly supported l.s.c convex function. Then for any Borel set $E$ the set $\partial g(E) = \bigcup_{x\in E} \partial g(x)$ is Borel.
\end{lemma}

\begin{proof}
The proof is an adaptation of the one found in \cite[Theorem 2.3]{Figalli2017}.
Let us define
\begin{equation*}
    \mathcal{F} = \left\{E\mid E \text{ Borel such that  } \partial g(E) \text{ is Borel}\right\}.
\end{equation*}
We will show that $\mathcal{F}$ is a $\sigma$-algebra containing the Borel sets. It is clearly stable by countable union. If $K$ is a compact set then $\partial g(K)$ is closed since $\partial g$ is upper hemicontinuous. Since $\mathbb{R}^n$ is a countable union of compact sets it is in $\mathcal{F}$. It remains to show that $\mathcal{F}$ is closed under complement. Remark that 
\begin{equation*}
    \partial g(\mathbb{R}^n \setminus E) = \left(\partial g(\mathbb{R}^n) \setminus \partial g(E)\right) \cup \left(\partial g(\mathbb{R}^n \setminus E)\cap \partial g (E)\right).
\end{equation*}
The first set is Borel and the second is of null Lebesgue measure because $g$ is convex, thus the union is Borel. Finally $\mathcal{F}$ is a $\sigma$-algebra containing all compact sets. Consequently $\mathcal{F}$ contains all Borel sets.
\end{proof}

\begin{lemma}
Let $X$ be a random variable valued in $\mathbb{R}^n$ and $g$ be an l.s.c proper convex function with compact domain on $\mathbb{R}^n$. Then there is a bounded random variable $U$ valued in $\mathbb{R}^n$ such that 
\begin{equation}
    g^\ast(X) + g(U) - \langle X,U\rangle = 0
\end{equation}
almost surely.
\end{lemma}

\begin{proof}
Since $g$ has compact domain and is l.s.c the correspondence $y \to \partial g^\ast(y)$ has nonempty closed values, it is also valued in the compact domain. Moreover for any Borel set $E$ the set $\{y \mid \partial g^\ast(y) \cap E\}$  is exactly the set $\partial g(E)$ which is measurable thus $\partial g^\ast$ is measurable and thus weakly measurable. The Kuratowski–Ryll-Nardzewski Selection Theorem \cite{Kuratowski1965} ensures the existence of a measurable function $p:\mathbb{R}^n \to \mathbb{R}^n$ such that $p(y) \in \partial g^\ast(y)$ for every $y \in \mathbb{R}^n$. Now define $U = p(X)$. Then since $X$ is a random variable and $p$ is measurable $U$ is a random variable. $\partial g^\ast$ is valued in a compact set thus $U$ is bounded. Finally by construction $U \in \partial g^\ast(X)$ almost surely and thus
\begin{equation*}
    g^\ast(X) + g(U) - \langle X,U\rangle = 0
\end{equation*}
almost surely.
\end{proof}

\section{A quantitative Prékopa-Leindler inequality}

As Proposition \ref{prop:rho-convex} was a consequence of Prékopa-Leindler inequality, Proposition \ref{prop:SH^*} is a consequence (via Lemma \ref{lem:lemma2}) of the following modified Prékopa-Leindler inequality. Indeed, the use of the classical Prékopa-Leindler inequality required a uniform control on the convexity of the function which is not available in the context of Proposition \ref{prop:SH^*}. The inequality obtained in the result below is a simple consequence of the convexity of the entropy functional $H(\nu) := \int \log \frac{d\nu}{dz}\,d\nu$ with respect to Wasserstein barycenters \cite{AC11}.

\begin{proposition}\label{prop:QPL}
Let $N \in \mathbb{N}^\ast$. Let $f_1,\ldots,f_N$ and $h$ be measurable functions on $\mathbb{R}^n$ such that $\int e^{f_i},\int e^h < + \infty$ and $\lambda_1,\ldots, \lambda_N \geq 0$ be such that $\sum_{i=1}^N \lambda_i = 1$. For all $i\in \{1,\ldots, N\}$, let $\nu_i$ be the probability measure on $\mathbb{R}^n$ with density proportional to $e^{f_i}$ and assume that $H(\nu_i)$ is finite and $\int |z|^2\,\nu_i(dz) <+\infty$. Then,
\begin{equation}\label{eq:PLQ}
    \sum_{i=1}^N \lambda_i \log\left(\int e^{f_i}\right) \leq \log\left(\int e^h\right) + \int \sum_{i=1}^{N} \lambda_i f_i(z_i) - h\left(\sum_{i=1}^N \lambda_i z_i\right) \,\bar{\pi}(dz_1,\ldots,dz_N)
\end{equation}
where $\bar{\pi}$ is the barycentric coupling between the measures $\nu_1,\ldots,\nu_N$ with weights $\lambda_1,\ldots, \lambda_N$. More precisely $\bar{\pi})$ is a solution to 
\begin{equation}\label{eq:MM}
    \inf_{\pi \in \Pi(\nu_1,\ldots,\nu_N)} \int \sum_{i\neq j} \lambda_i \lambda_j \vert z_i-z_j\vert^2 \,\pi(dz_1,\ldots,dz_N),
\end{equation}
where $\Pi(\nu_1,\ldots,\nu_N)$ is the set of probability measures on $(\R^n)^N$ admitting $\nu_1,\ldots, \nu_N$ as marginals.
\end{proposition}
When $N=2$ and under the conditions of the classical Prékopa-Leindler inequality $h((1-t)z_1+tz_2) \geq (1-t) f_1(z_1) + t f_2(z_2)$, $z_1,z_2 \in \R^n$, we  recover the usual form of the Prékopa-Leindler inequality.

\begin{proof}
According to \cite{AC11}, the entropy functional $H$ is convex along Wasserstein barycenters between $\nu_1,\ldots,\nu_N$. 
More precisely, if $\bar{\pi}$ is solution of \eqref{eq:MM} and $\nu_\Lambda$ is the pushforward of $\bar{\pi}$ by the map $(z_1,\ldots,z_N) \mapsto \sum_{i=1}^N \lambda_i z_i$, then
\begin{equation*}
    H(\nu_\Lambda) \leq \sum_{i=1}^N \lambda_i H(\nu_i) < + \infty.
\end{equation*}
Since the entropy is finite we have $H(\nu_i) = \int f_i d\nu_i - \log\left(\int e^{f_i} \right)$.  Moreover by duality, 
\[
H(\nu_\Lambda) \geq \int h \,d\nu_\Lambda - \log\left(\int e^h \right)
\]
because $\int e^h < + \infty$. By combining the two observations, we have
\begin{equation*}
     \int h \,d\nu_\Lambda - \log\left(\int e^h\right) \leq \sum_{i=1}^N \lambda_i \left[\int f_i d\nu_i - \log\left(\int e^{f_i} \right)\right].
\end{equation*}
All the terms in this inequality are well defined, thus we can rearrange them and get the wanted inequality. 
\end{proof}

\end{document}